%% file: Paper3_S.tex
\documentclass[reqno]{amsart}
\include{packages}
\bibliography{sources}
\usetikzlibrary{decorations.pathreplacing,decorations.markings,calc}
\usetikzlibrary{arrows, matrix,decorations.pathmorphing, decorations.text, math,calc, hobby, cd, positioning, decorations.markings,arrows.meta, patterns, shapes.geometric}
\usetikzlibrary{decorations.markings, fpu, arrows.meta,decorations.pathmorphing,decorations.markings}
\usepackage{wasysym}

\tikzset{White/.style={
		draw=white, 
		postaction={decorate}, 
		decoration={markings,mark=at position 0.5 with {\arrow[black]{triangle 45} } }
	}
}
\tikzset{Snake/.style={decorate,decoration={snake}}}

\DeclareMathAlphabet{\mathpzc}{OT1}{pzc}{m}{it}

\DeclareFieldFormat[article]{title}{{\it #1}}
\DeclareFieldFormat{journaltitle}{{\rm #1}}
\renewbibmacro{in:}{%
	\ifentrytype{article}{}{\printtext{\bibstring{in}\intitlepunct}}}
\DeclareFieldFormat[incollection]{title}{{\it #1}}
\DeclareFieldFormat{journaltitle}{{\rm #1}}

\usepackage{tabto}
\usepackage{imakeidx}

\begin{document}
	
	\def\subsectionautorefname{Section}
	\def\subsubsectionautorefname{Section}
	\def\sectionautorefname{Section}
	\def\equationautorefname~#1\null{(#1)\null}
	\input{macros}

	
	\title[Fundamental Groups via Exchange Graphs]%
	{Fundamental Groups of Genus-$0$ Quadratic Differential Strata via Exchange Graphs}
	
	\author{Jeong-Hoon So}
	%
	
	\begin{abstract}
		We investigate how exchange-graph techniques can be used to study the topology of strata of meromorphic quadratic differentials. The exchange graph provides natural generators for the fundamental group. By extending the combinatorics of triangulations to weighted mixed-angulations, we generalise the familiar relations arising in the simple-zero case and introduce an additional relation that appears only around higher-order zeroes. In the genus-zero case with four singularities, we show that these relations suffice to give explicit presentations of the fundamental group.
		
	\end{abstract}
	\maketitle

	\tableofcontents
	
	\input{sec_braid}

	\input{sec_eg}
	
	\input{sec_quad}

	\input{sec_ex}

	\par
	\printbibliography
	
\end{document}

%% file: macros.tex
\newcommand{\mynewtheorem}[4]{
	\if\relax\detokenize{#3}\relax 
	\if\relax\detokenize{#4}\relax 
	\newtheorem{#1}{#2}
	\else
	\newtheorem{#1}{#2}[#4]
	\fi
	\else
	\newaliascnt{#1}{#3}
	\newtheorem{#1}[#1]{#2}
	\aliascntresetthe{#1}
	\fi
	\expandafter\def\csname #1autorefname\endcsname{#2}
}

\mynewtheorem{theorem}{Theorem}{}{section}
\mynewtheorem{lemma}{Lemma}{theorem}{}
\mynewtheorem{lem}{Lemma}{theorem}{}
\mynewtheorem{rem}{Remark}{lemma}{}
\mynewtheorem{prop}{Proposition}{lemma}{}
\mynewtheorem{cor}{Corollary}{lemma}{}
\mynewtheorem{definition}{Definition}{lemma}{}
\mynewtheorem{question}{Question}{lemma}{}
\mynewtheorem{assumption}{Assumption}{lemma}{}
\mynewtheorem{example}{Example}{lemma}{}
\mynewtheorem{exm}{Example}{lemma}{}
\mynewtheorem{rmk}{Remark}{lemma}{}
\mynewtheorem{pb}{Problem}{lemma}{}
\newtheorem*{rmk*}{Remark}
\newtheorem*{pb*}{Problem}

\newtheorem{conjecture}[theorem]{Conjecture}
\newtheorem{condition}[theorem]{Condition}
\newtheorem{setup}[theorem]{Setup}

\def\defbb#1{\expandafter\def\csname b#1\endcsname{\mathbb{#1}}}
\def\defcal#1{\expandafter\def\csname c#1\endcsname{\mathcal{#1}}}
\def\deffrak#1{\expandafter\def\csname frak#1\endcsname{\mathfrak{#1}}}
\def\defop#1{\expandafter\def\csname#1\endcsname{\operatorname{#1}}}
\def\defbf#1{\expandafter\def\csname b#1\endcsname{\mathbf{#1}}}

\makeatletter
\def\defcals#1{\@defcals#1\@nil}
\def\@defcals#1{\ifx#1\@nil\else\defcal{#1}\expandafter\@defcals\fi}
\def\deffraks#1{\@deffraks#1\@nil}
\def\@deffraks#1{\ifx#1\@nil\else\deffrak{#1}\expandafter\@deffraks\fi}
\def\defbbs#1{\@defbbs#1\@nil}
\def\@defbbs#1{\ifx#1\@nil\else\defbb{#1}\expandafter\@defbbs\fi}
\def\defbfs#1{\@defbfs#1\@nil}
\def\@defbfs#1{\ifx#1\@nil\else\defbf{#1}\expandafter\@defbfs\fi}
\def\defops#1{\@defops#1,\@nil}
\def\@defops#1,#2\@nil{\if\relax#1\relax\else\defop{#1}\fi\if\relax#2\relax\else\expandafter\@defops#2\@nil\fi}
\makeatother

\defbbs{ZHQCNPALRVWS}
\defcals{ABCDOPQMNXYLTRAEHZKCIV}
\deffraks{apijklgmnopqueRCc}
\defops{IVC, PGL,SL,mod,Spec,Re,Gal,Tr,End,GL,Hom,PSL,H,div,Aut,rk,Mod,R,T,Tr,Mat,Vol,MV,Res,Hur, vol,Z,diag,Hyp,hyp,hl,ord,Im,ev,U,dev,c,CH,fin,pr,Pic,lcm,ch,td,LG,id,Sym,Aut,Log,tw,irr,discrep,BN,NF,NC,age,hor,lev,ram,NH,av,app,Quad,Stab,Per,Nil,Ker,EG,CEG,PB,Conf,MCG,Diff,Pol,TB,ad,Homeo,Co,Rel,B,F}
\defbfs{uvzwp} 

\newcommand*{\Pure}{\operatorname{P}}

\def\ep{\varepsilon}
\def\ve{\varepsilon}
\def\abs#1{\lvert#1\rvert}
\def\dd{\mathrm{d}}
\def\WP{\mathrm{WP}}
\def\inj{\hookrightarrow}
\def\eq{=}

\def\i{\mathrm{i}}
\def\e{\mathrm{e}}
\def\st{\mathrm{st}}
\def\ct{\mathrm{ct}}
\def\rel{\mathrm{rel}}
\def\odd{\mathrm{odd}}
\def\even{\mathrm{even}}

\def\uC{\underline{\bC}}
\def\ol{\overline}

\def\Vrel{\bV^{\mathrm{rel}}}
\def\Wrel{\bW^{\mathrm{rel}}}
\def\twolev{\mathrm{LG_1(B)}}

\def\be{\begin{equation}}   \def\ee{\end{equation}}     \def\bes{\begin{equation*}}    \def\ees{\end{equation*}}
\def\ba{\be\begin{aligned}} \def\ea{\end{aligned}\ee}   \def\bas{\bes\begin{aligned}}  \def\eas{\end{aligned}\ees}
\def\={\;=\;}  \def\+{\,+\,} \def\m{\,-\,}

\newcommand*{\proj}{\mathbb{P}}
\newcommand{\IVCst}[1][\mu]{{\mathcal{IVC}}({#1})}
\newcommand{\barmoduli}[1][g]{{\overline{\mathcal M}}_{#1}}
\newcommand{\moduli}[1][g]{{\mathcal M}_{#1}}
\newcommand{\omoduli}[1][g]{{\Omega\mathcal M}_{#1}}
\DeclareDocumentCommand{\qmoduli}{ O{g} O{}}{{\mathrm{Quad}^{#2}_{#1}}}

\newcommand{\modulin}[1][g,n]{{\mathcal M}_{#1}}
\newcommand{\omodulin}[1][g,n]{{\Omega\mathcal M}_{#1}}
\newcommand{\zomoduli}[1][]{{\mathcal H}_{#1}}
\newcommand{\barzomoduli}[1][]{{\overline{\mathcal H}_{#1}}}
\newcommand{\pomoduli}[1][g]{{\proj\Omega\mathcal M}_{#1}}
\newcommand{\pomodulin}[1][g,n]{{\proj\Omega\mathcal M}_{#1}}
\newcommand{\pobarmoduli}[1][g]{{\proj\Omega\overline{\mathcal M}}_{#1}}
\newcommand{\pobarmodulin}[1][g,n]{{\proj\Omega\overline{\mathcal M}}_{#1}}
\newcommand{\potmoduli}[1][g]{\proj\Omega\tilde{\mathcal{M}}_{#1}}
\newcommand{\obarmoduli}[1][g]{{\Omega\overline{\mathcal M}}_{#1}}
\newcommand{\qbarmoduli}[1][g]{{\Omega^2\overline{\mathcal M}}_{#1}}
\newcommand{\obarmodulio}[1][g]{{\Omega\overline{\mathcal M}}_{#1}^{0}}
\newcommand{\otmoduli}[1][g]{\Omega\tilde{\mathcal{M}}_{#1}}
\newcommand{\pom}[1][g]{\proj\Omega{\mathcal M}_{#1}}
\newcommand{\pobarm}[1][g]{\proj\Omega\overline{\mathcal M}_{#1}}
\newcommand{\pobarmn}[1][g,n]{\proj\Omega\overline{\mathcal M}_{#1}}
\newcommand{\princbound}{\partial\mathcal{H}}
\newcommand{\omoduliinc}[2][g,n]{{\Omega\mathcal M}_{#1}^{{\rm inc}}(#2)}
\newcommand{\obarmoduliinc}[2][g,n]{{\Omega\overline{\mathcal M}}_{#1}^{{\rm inc}}(#2)}
\newcommand{\pobarmoduliinc}[2][g,n]{{\proj\Omega\overline{\mathcal M}}_{#1}^{{\rm inc}}(#2)}
\newcommand{\otildemoduliinc}[2][g,n]{{\Omega\widetilde{\mathcal M}}_{#1}^{{\rm inc}}(#2)}
\newcommand{\potildemoduliinc}[2][g,n]{{\proj\Omega\widetilde{\mathcal M}}_{#1}^{{\rm inc}}(#2)}
\newcommand{\omoduliincp}[2][g,\lbrace n \rbrace]{{\Omega\mathcal M}_{#1}^{{\rm inc}}(#2)}
\newcommand{\obarmoduliincp}[2][g,\lbrace n \rbrace]{{\Omega\overline{\mathcal M}}_{#1}^{{\rm inc}}(#2)}
\newcommand{\obarmodulin}[1][g,n]{{\Omega\overline{\mathcal M}}_{#1}}
\newcommand{\LTH}[1][g,n]{{K \overline{\mathcal M}}_{#1}}
\newcommand{\PLS}[1][g,n]{{\bP\Xi \mathcal M}_{#1}}

\DeclareDocumentCommand{\LMS}{ O{\mu} O{g,n} O{}}{\Xi\overline{\mathcal{M}}^{#3}_{#2}(#1)}
\DeclareDocumentCommand{\Romod}{ O{\mu} O{g,n} O{}}{\Omega\mathcal{M}^{#3}_{#2}(#1)}

\newcommand*{\Tw}[1][\Lambda]{\mathrm{Tw}_{#1}}  
\newcommand*{\sTw}[1][\Lambda]{\mathrm{Tw}_{#1}^s}  

\newcommand{\HH}{{\mathbb H}}
\newcommand{\MM}{{\mathbb M}}
\newcommand{\bbC}{{\mathbb C}}
\newcommand{\TT}{{\mathbb T}}

\newcommand\PP{\mathbb P}
\renewcommand\R{\mathbb R}
\renewcommand\Z{\mathbb Z}
\newcommand\N{\mathbb N}
\newcommand\Q{\mathbb Q}
\renewcommand{\H}{\mathbb{H}}
\newcommand{\halfplane}{\mathbb{H}}
\newcommand{\chalfplane}{\overline{\mathbb{H}}}
\newcommand{\bk}{k}

\newcommand{\bfa}{{\bf a}}
\newcommand{\bfb}{{\bf b}}
\newcommand{\bfd}{{\bf d}}
\newcommand{\bfe}{{\bf e}}
\newcommand{\bff}{{\bf f}}
\newcommand{\bfg}{{\bf g}}
\newcommand{\bfh}{{\bf h}}
\newcommand{\bfm}{{\bf m}}
\newcommand{\bfn}{{\bf n}}
\newcommand{\bfp}{{\bf p}}
\newcommand{\bfq}{{\bf q}}
\newcommand{\bft}{{\bf t}}
\newcommand{\bfP}{{\bf P}}
\newcommand{\bfR}{{\bf R}}
\newcommand{\bfU}{{\bf U}}
\newcommand{\bfu}{{\bf u}}
\newcommand{\bfx}{{\bf x}}
\newcommand{\bfz}{{\bf z}}

\newcommand{\bfl}{{\boldsymbol{\ell}}}
\newcommand{\bfmu}{{\boldsymbol{\mu}}}
\newcommand{\bfeta}{{\boldsymbol{\eta}}}
\newcommand{\bftau}{{\boldsymbol{\tau}}}
\newcommand{\bfomega}{{\boldsymbol{\omega}}}
\newcommand{\bfsigma}{{\boldsymbol{\sigma}}}
\newcommand{\bfnu}{{\boldsymbol{\nu}}}
\newcommand{\bfrho}{{\boldsymbol{\rho}}}
\newcommand{\bfone}{{\boldsymbol{1}}}

\newcommand\cl{\mathcal}
\newcommand{\calH}{\mathcal{H}}

\newcommand{\calA}{\mathcal A}
\newcommand{\calK}{\mathcal K}
\newcommand{\calD}{\mathcal D}
\newcommand{\calC}{\mathcal C}
\newcommand\C{\mathcal C}
\newcommand{\calT}{\mathcal T}
\newcommand{\calB}{\mathcal B}
\newcommand{\calF}{\mathcal F}
\newcommand{\calV}{\mathcal V}
\newcommand{\calQ}{\mathcal Q}
\newcommand{\calX}{\mathcal X}
\newcommand{\calY}{\mathcal Y}
\newcommand{\calP}{\mathcal P}
\newcommand{\calJ}{\mathcal J}
\newcommand{\calS}{\mathcal S}
\newcommand{\calZ}{\mathcal Z}

\newcommand\pvd{\operatorname{pvd}}
\newcommand\per{\operatorname{per}}
\newcommand\thick{\operatorname{thick}}

\newcommand\perf{\operatorname{perf}}
\newcommand{\modules}{\operatorname{mod}}
\newcommand{\Loc}{\operatorname{Loc}}
\renewcommand{\Mod}{\operatorname{Mod}}
\renewcommand{\Hom}{\operatorname{Hom}}
\newcommand{\Ext}{\operatorname{Ext}}
\newcommand{\coker}{\operatorname{coker}}

\newcommand\Rep{\operatorname{Rep}}
\newcommand\ext{\operatorname{ext}}
\newcommand{\heart}{\heartsuit}

\newcommand{\sph}{\operatorname{sph}}
\newcommand{\Br}{\operatorname{Br}}
\renewcommand{\Tw}{\operatorname{Tw}}
\newcommand{\RHom}{\operatorname{RHom}}

\newcommand{\torsion}{\mathcal{T}}
\newcommand{\torsionfree}{\mathcal{F}}
\newcommand{\tstr}{\mathcal{L}}

\newcommand\cy{\mathrm{CY}}
\newcommand\CY{\mathrm{CY}}

\newcommand\DQI{{\mathcal{D}_{Q_I}}}
\newcommand\AQ{{\mathcal{A}_Q}}
\newcommand\AQI{{\mathcal{A}_{Q_I}}}
\newcommand{\DQ}{{\mathcal{D}_{Q}}}
\DeclareDocumentCommand{\DQQ}{ O{}O{}}{{\mathcal{D}^{#1}_{#2}}}
\DeclareDocumentCommand{\AQQ}{ O{q} }{{\mathcal{A}_{#1}}}

\newcommand{\sgn}{\operatorname{sgn}}
\renewcommand{\GL}{\operatorname{GL}}
\renewcommand{\rk}{\operatorname{rank}}
\newcommand{\opL}{\operatorname{L}}

\newcommand\stab{\operatorname{Stab}}
\newcommand\gstab{\operatorname{GStab}}

\newcommand{\spn}{\operatorname{span}}
\newcommand{\GStab}{\operatorname{GStab}}
\newcommand{\PGStab}{\bP\mathrm{GStab}}
\newcommand{\MStab}{\operatorname{MStab}}
\newcommand{\PMStab}{\bP\mathrm{MStab}}
\newcommand{\Tilt}{\operatorname{Tilt}}

\def\oQ{\overline{Q}}
\def\bY{\mathbf{Y}}
\def\bX{\mathbf{X}}
\newcommand{\fmu}{\mu^\sharp}
\newcommand{\bmu}{\mu^\flat}
\newcommand{\Cone}{\operatorname{Cone}}
\newcommand\add{\operatorname{Add}}
\newcommand\Irr{\operatorname{Irr}}

\newcommand{\EGs}{\EG^{s}} 
\newcommand{\SEG}{\operatorname{SEG}} 
\newcommand{\pSEG}{\operatorname{pSEG}} 
\newcommand{\EGp}{\EG^\circ}       
\newcommand{\SEGp}{\SEG^\circ}       
\newcommand{\EGb}{\EG^\bullet}       
\newcommand{\SEGb}{\SEG^\bullet}       
\newcommand{\SEGV}{\SEG_{{^\perp}\calV}} 
\newcommand{\pSEGV}{\pSEG_{{^\perp}\calV}}
\newcommand{\SEGVb}{\SEGb_{{^\perp}\calV}} 
\newcommand{\pSEGVb}{\pSEG^\bullet_{{^\perp}\calV}}

\newcommand{\EGT}{\EG^\circ}
\newcommand{\uEG}{\underline{\EG}} 
\newcommand{\uCEG}{\underline{\CEG}} 

\newcommand{\CA}{\operatorname{CA}}
\newcommand{\OA}{\operatorname{OA}}
\newcommand{\wOA}{\widetilde{\OA}}
\newcommand{\wCA}{\widetilde{\CA}}
\newcommand{\BT}{\operatorname{BT}}
\newcommand{\SBr}{\operatorname{SBr}}
\newcommand\Bt[1]{\operatorname{B}_{#1}}
\newcommand\bt[1]{\operatorname{B}_{#1}^{-1}}

\newcommand{\Sim}{\operatorname{Sim}}
\def\dual{\iota}
\def\ivc{\iota_v}
\newcommand\ind{\operatorname{index}}
\newcommand{\Qgrad}{\overline{Q}}
\newcommand\diff{\operatorname{d}}
\def\Dsow{\Dcol}
\def\pvc{e \Gamma e}
\def\Qvc{Q_V^c}
\def\Gwt{\FGamma_\wt}
\def\wX{\widetilde{X}}

\def\DD{\mathbf{D}}
\def\uD{\underline{\calD}}
\def\uh{\underline{\calH}}
\def\uZ{\underline{Z}}
\def\us{\underline{\sigma}}
\newcommand{\isom}{\cong}

\newcommand{\tilt}[3]{{#1}^{#2}_{#3}}
\newcommand\Stap{\Stab^\circ} 
\newcommand\Stas{\Stab^\bullet}
\newcommand{\cub}{\operatorname{U}} 
\newcommand{\skel}{\wp} 

\newcommand{\mai}{\mathbf{i}} 

\newcommand\wT{\widetilde{\TT}} 
\newcommand{\Tri}{\Delta}
\newcommand{\deco}{\Delta}

\def\w{\mathbf{w}}
\def\wtpt{\w_{+2}}
\newcommand\surf{\mathbf{S}}  
\newcommand\surfo{{\mathbf{S}}_\Tri}  
\newcommand\surfw{\surf^\w}  
\newcommand\sow{\surf_\w}  
\newcommand\subsur{\Sigma}  
\newcommand\colsur{\overline{\surf}_\w}  
\def\Dsan{\calD_3(\surfo)}
\def\Dsow{\calD(\sow)}
\def\Dsub{\calD_3(\subsur)}
\def\Dcol{\calD(\colsur)}

\def\ww{node[white]{$\bullet$}node[red]{$\circ$}}
\def\nn{node{$\bullet$}}

\newcommand{\uk}{\mathbf{k}}
\def\sing{\operatorname{Sing}}

\newcommand\jiantou{edge[->]}
\newcommand\AS{\mathbb{A}}

\def\grad{\lambda}
\def\gms{\surf^\grad}
\def\gmsw{\surf^{\grad,\wt}_{\Tri}}
\def\iT{\TT_0}

\def\weta{\widetilde{\eta}}
\def\wzeta{\widetilde{\zeta}}
\def\walpha{\widetilde{\alpha}}
\def\wbeta{\widetilde{\beta}}
\def\wgamma{\widetilde{\gamma}}

\def\dsan{\calD_3(\surfo)}
\def\psan{\per(\surfo)}
\def\dsow{\Dsow}
\def\psow{\per(\colsur)}

\def\RP{\operatorname{RP}}
\def\Rs{\operatorname{RS}}
\newcommand{\ST}{\operatorname{ST}}  
\newcommand{\STp}{\ST^\circ}  

\newcommand{\Int}{\operatorname{Int}}

\def\hori{_{\operatorname{H}}}

\newcommand{\Note}[1]{\textcolor{red}{#1}}
\newcommand{\note}[1]{\textcolor{Emerald}{#1}}
\newcommand{\qy}[1]{\textcolor{cyan}{#1}}

\newcommand\foli[5]{
	\foreach \m in {0,#1,...,#5}{
		\coordinate (#2#4) at ($(#2)!.5!(#4)$);
		\draw[Emerald!50]plot [smooth,tension=.5] coordinates
		{(#2) ($(#2#4)!\m!(#3)$) (#4)};
		\draw[thick,gray](#2)to(#3)to(#4);
}}

\def\iA{\AS_0}
\def\iA{\AS_0}

\newcommand\Ind{\operatorname{Ind}}

\newcommand\xx{\mathbf{X}} 
\newcommand\surp{\xx^\circ}
\newcommand{\Zer}{\operatorname{Zero}}
\newcommand{\Crit}{\operatorname{Crit}}
\newcommand{\FQuad}{\operatorname{FQuad}}
\newcommand{\FQuab}{\FQuad^{\bullet}}

\newcommand{\Imgy}{\operatorname{Im}}
\newcommand{\numarc}{n}


\newcommand{\Exch}{\operatorname{Exch}}
\newcommand{\arrowIn}{
	\tikz \draw[-stealth] (-1pt,0) -- (1pt,0);
}


\newcommand{\wh}{\widehat}
\newcommand{\wt}{\widetilde}

\newcommand{\whmu}{\widehat{\mu}}
\newcommand{\whrho}{\widehat{\rho}}
\newcommand{\whLa}{\widehat{\Lambda}}

\newcommand{\ps}{\mathrm{ps}}

\newcommand{\tdpm}[1][{\Gamma}]{\mathfrak{W}_{\operatorname{pm}}(#1)}
\newcommand{\tdps}[1][{\Gamma}]{\mathfrak{W}_{\operatorname{ps}}(#1)}
\newcommand{\cal}[1]{\mathcal{#1}}

\newlength{\halfbls}\setlength{\halfbls}{.8\baselineskip}

\newcommand*{\Teichmuller}{Teich\-m\"uller\xspace}

\DeclareDocumentCommand{\MSfun}{ O{\mu} }{\mathbf{MS}({#1})}
\DeclareDocumentCommand{\MSgrp}{ O{\mu} }{\mathcal{MS}({#1})}
\DeclareDocumentCommand{\MScoarse}{ O{\mu} }{\mathrm{MS}({#1})}
\DeclareDocumentCommand{\tMScoarse}{ O{\mu} }{\widetilde{\mathbb{P}\mathrm{MS}}({#1})}

\newcommand{\kmin}{\kappa_{(2g-2)}}
\newcommand{\ktop}{\kappa_{\mu_\Gamma^{\top}}}
\newcommand{\kbot}{\kappa_{\mu_\Gamma^{\bot}}}

\newcommand\bra{\langle}
\newcommand\ket{\rangle}
\newcommand{\del}{\partial}
\newcommand{\deld}[1]{\frac{\del}{\del {#1} }}
\newcommand{\frap}{\frac{1}{2\pi\ii}}
\renewcommand{\Re}{\operatorname{Re}}
\renewcommand{\Im}{\operatorname{Im}}
\newcommand{\Rp}{\mathbb{R}_{>0}}
\newcommand{\Rm}{\R_{<0}}

\def\lift{\mathrm{lift}}
\def\ul{\underline}
\newcommand\<{\langle}
\renewcommand\>{\rangle}
\def\h{\calH}
\def\D{\calD}
\def\aut{\mathpzc{Aut}}
\tikzcdset{arrow style=tikz, diagrams={>=stealth}}
\usetikzlibrary{arrows}

\def\acf{\mathbf{k}}
\def\PTS{\mathbb{P}T\surf}
\newcommand\coho[1]{\operatorname{H}^{#1}}
\def\MTS{\mathbb{R}T\surf^{\lambda}}

\def\ora{\overrightarrow}
\def\harc{\ora{\gamma_h}}
\def\varc{\ora{\gamma_v}}
\def\cube{\mathrm{U}}

\newcommand{\on}[1]{\operatorname{#1}}
\newcommand{\iv}[1]{(#1)^{-1}}

\setcounter{tocdepth}{1}
\setcounter{secnumdepth}{3}

%% file: sec_braid.tex
\section{Introduction}

Understanding the topology of strata of quadratic differentials is a central problem in Teichmüller theory. Yet, despite their importance, even basic topological invariants---such as their fundamental groups---are difficult to compute directly.
A combinatorial approach, initiated by \cite{BS15} and expanded in \cite{KQ2}, replaces the analytic complexity of quadratic differentials with a combinatorial model built from triangulations and flips. In this setting, the \emph{exchange graph} naturally provides generators for the fundamental group of a stratum. The main problem, and the focus of this paper, is to determine the \emph{relations} among these generators.

\medskip

First, consider a stratum of meromorphic quadratic differentials on a Riemann surface whose singularities consist only of simple zeroes and poles of order at least three. After performing the real oriented blow-up at each pole, the surface becomes a marked surface with boundary components corresponding to the poles. For any saddle-free differential in the stratum, the horizontal foliation determines a canonical triangulation (see Subsection~\ref{subsection:difftoeg}), and any two such triangulations differ by sequences of flips.

In the work of \cite{BS15}, these triangulations appear as chambers in a natural wall-and-chamber decomposition of the stratum, and flips correspond to crossing codimension-one walls. This yields a surjection
\[
\pi_{1}(\text{exchange graph of triangulations})\;\twoheadrightarrow\;\pi_{1}(\text{stratum}),
\]
showing that the exchange graph provides a full set of generators for the fundamental group.

On the combinatorial side, the exchange graph contains three canonical types of loops, which we will refer to as \emph{relations}, since these are the loops we later quotient out. Each relation is determined by the relative position of two arcs (edges) in a triangulation:
\begin{itemize}
	\item the \emph{square relation}, determined by a pair of arcs that do not intersect inside any triangle;
	\item the \emph{pentagon relation}, determined by a pair of arcs that intersect once in a single triangle;
	\item the \emph{hexagon relation}, determined by a pair of arcs that intersect once in each of two triangles.
\end{itemize}
The images of the loops in the framed stratum are contractible. This was proven in \cite{KQ2}, where it is also shown that they already generate the full kernel.  
We give a reformulation of the arguments using only triangulations and quadratic differentials, avoiding cluster-theoretic language:

\begin{prop}[Proposition~\ref{theorem:quadeg}]
	There is an isomorphism
	\[
	\pi_1\EG^\star(\mathbf{S})/\Rel(\mathbf{S}) \;\simeq\; \pi_1\Quad(\mathbf{S})
	\]
	of orbifold fundamental groups.
\end{prop}

\medskip

The goal of the present paper is to extend this picture to strata with \emph{higher-order zeroes}. 
A zero of order $k$ has $k+2$ outgoing separatrices, so triangulations must be replaced by \emph{mixed-angulations} of \cite{BMQS}, such that a zero of order $k$ is represented by a $(k{+}2)$-gon. Flips generalize accordingly, yielding an exchange graph that again provides natural generators for the fundamental group of the stratum.

The classical square, pentagon, and hexagon relations generalize directly to mixed-angulations. We provide a \emph{second type of hexagon} that only appears around higher-order zeroes, determined by a pair of arcs that intersect twice inside the same polygon (see Subsection~\ref{subsec:eg_mixed}).
A result of \cite{BMQS} shows that the chamber map induces a surjection
\[
\pi_{1}(\text{exchange graph of mixed-angulations})
\;\twoheadrightarrow\;
\pi_{1}(\text{stratum}),
\]
so the exchange graph always supplies the generators.
We prove in Proposition~\ref{prop:relationsinkernel_gen} that all the before-mentioned relations---the square, pentagon, and both types of hexagons---lie in the kernel.

What remains open, in general, is whether these relations generate the \emph{entire} kernel.
As a first test, we analyse the simplest non-trivial family: genus~$0$ with four singularities.
In this situation, the fundamental group is essentially known by classical arguments, so our explicit computation provides a non-trivial cross-check of the exchange-graph method.
We show that, for every choice of zero orders and pole orders \(>2\), the
generalized relations already account for all relations among the exchange–graph generators:

\begin{theorem}[Theorem~\ref{thm:isomorphism-four-singularity}]
	For $g=0$ and four singularities we have a natural isomorphism
	\[
	\pi_1\EG^\star(\sow)/\Rel^\star(\sow)
	\;\cong\;
	\pi_1\Quad(\sow)
	\]
	of orbifold fundamental groups.
	Moreover, the resulting group can be read off from the symmetry type:
	\begin{itemize}
		\item[(a)] $\mathbb{Z} \times F_2$, if all four singularities have distinct orders;
		\item[(b)] $\langle z,c,d \mid [z,c],[z,d],\ c^2=1\rangle$, if two singularities have identical order, and the remaining two have even order not coinciding with the first two;
		\item[(c)] $\langle c,d \mid [c^2,d]\rangle$, if two singularities have identical order, and the remaining two have odd order not coinciding with the first two;
		\item[(d)] $\langle s,d \mid [d^3,s], s^2 \rangle$, if there are three zeroes of identical even order;
		\item[(e)] $B_3 = \langle s,d \mid d^3 = s^2 \rangle$, if there are three zeroes of identical odd order.
	\end{itemize}
\end{theorem}

The weighted exchange graph always supplies generators for the fundamental group of any stratum, and we have provided an explicit set of relations among them.  
An interesting open problem is whether additional relations appear when the number
of singularities increases or when the surface has positive genus

\medskip

\noindent\textbf{Orbifold convention.}
Many unframed strata of quadratic differentials are orbifolds rather than manifolds, due to residual symmetries coming from permutations of zeroes of identical order.  
Throughout this paper, all fundamental groups are orbifold fundamental groups whenever the underlying stratum carries such stabilizers.  
This convention applies in particular to the explicit computations in Section~\ref{sec:examples}.

\medskip

\noindent\textbf{Open Problems and Future Directions.}
The results of this paper suggest several directions for further investigation. A natural next step is to extend the methods developed here to strata with a larger number of singularities or to higher genus surfaces. In the cases considered in this paper the exchange graphs are still small enough that their fundamental groups can be analysed by explicit computations. For strata with more singularities
the exchange graphs grow rapidly in size, with many more vertices and edges, making such direct computations increasingly difficult.

An interesting open question is whether the relations considered in this paper—namely the square, pentagon and both type of hexagon relations—already
generate all relations in general. It is conceivable that additional types of relations appear in more complicated configurations of flips, but it may also be that the known local relations suffice even in these larger exchange graphs.

\medskip

\noindent\textbf{Acknowledgments.}
I thank my supervisor Martin Möller for his guidance and support and gratefully acknowledge support by the Deutsche Forschungsgemeinschaft (DFG, German Research Foundation) through the Collaborative Research Centre TRR 326 Geometry and Arithmetic of Uniformized Structures, project number 444845124.

\section{Braid Groups and Surfaces}\label{sec:braids}

The topology of strata of quadratic differentials naturally leads to questions about how collections of points move on a surface and how the underlying surface may be deformed.
Both phenomena are encoded by classical objects: braid groups, configuration spaces, and mapping class groups.  
These groups appear later as explicit fundamental groups of low–complexity strata and as kernels or targets of natural forgetful maps between surface models.

In this section we assemble the background needed for these applications.  
We begin with braid groups and their pure versions, together with the corresponding surface braid groups obtained from configuration spaces on general surfaces.  
We then introduce the classes of surfaces that arise as topological models for quadratic differentials---marked, decorated, and weighted decorated surfaces.  
Finally, we review the mapping class groups acting on these surfaces and the role of the surface braid group as the kernel of the forgetful map from decorated to undecorated mapping class groups.

\subsection{Braid and Surface Braid Groups} 

We recall the classical definitions and standard presentations of the braid group and pure braid group before introducing their surface analogues.

\begin{definition}[\cite{birman2016braids}, Theorem 1.8]\label{def:B_n}
	The braid group $B_n$ on $n$ strands is generated by $\sigma_1,\dots,\sigma_{n-1}$ subject to the relations
	\begin{itemize}
		\item[(i)] $\sigma_i\sigma_j=\sigma_j\sigma_i$ if $|i-j|>1$, 
		\item[(ii)] $\sigma_i\sigma_{i+1}\sigma_i=\sigma_{i+1}\sigma_i\sigma_{i+1}$ for all $i$.
	\end{itemize}
\end{definition}

The braid group admits a geometric interpretation. Consider $n$ parallel, numbered strands. The generator $\sigma_i$ corresponds to twisting the $i$-th and $(i+1)$-th strands such that their lower ends are swapped, with the $i$-th strand passing over the $(i+1)$-th strand. There is a natural surjection $B_n\to S_n$ sending $\sigma_i$ to the transposition $(i,i+1)$, whose kernel is the \emph{pure braid group} $\PB_n$.

For $1\le i<j\le n$, set $$a_{ij} := (\sigma_{j-1}\cdots\sigma_{i+1})\sigma_i(\sigma_{i+1}^{-1}\cdots\sigma_{j-1}^{-1}), \qquad A_{ij}:=a_{ij}^2.$$

\begin{theorem}[\cite{lee2010positive}, Thm. 1.1,Rem. 3.1]\label{thm:purebraid}
	The pure braid group $\PB_n$ is generated by $$\{A_{ij}\mid1\le i<j\le n\}$$ subject to the relations:
	\begin{itemize}
		\item[(P1)] $A_{ij}A_{rs}=A_{rs}A_{ij}$ if $r<s<i<j$ or $i<r<s<j$,
		\item[(P2)] $A_{ji}A_{ir}A_{rj}=A_{ir}A_{rj}A_{ji}=A_{rj}A_{ji}A_{ir}$ if $r<i<j$,
		\item[(P3)] $A_{rs}(A_{jr}A_{ji}A_{js})=(A_{jr}A_{ji}A_{js})A_{rs}$ if $r<i<s<j$.
	\end{itemize}
	Moreover there is a short exact sequence $$0\to F_n \xrightarrow{g}\PB_{n+1}\xrightarrow{f}\PB_n\to0,$$ where $F_n=\langle A_{1(n+1)},\dots,A_{n(n+1)}\rangle$ is free. This sequence splits, so $$\PB_{n+1}\cong F_n\rtimes\PB_n.$$
\end{theorem}

\emph{$n=2$.} Here $B_2=\langle\sigma_1\rangle\cong\Z$, and $\PB_2=\langle\sigma_1^2\rangle$ is the index-two subgroup of even braids.

\emph{$n=3$.} We have $$B_3=\langle a,b\mid aba=bab\rangle,$$ with $a=\sigma_1$, $b=\sigma_2$. The pure braid group is $$\PB_3=\langle A_{12},A_{23},A_{13}\mid A_{23}A_{12}A_{13}=A_{12}A_{13}A_{23}=A_{13}A_{23}A_{12}\rangle,$$ where $A_{12}=a^2$, $A_{23}=b^2$, and $A_{13}=ba^2b^{-1}$. Its central element $C=(ab)^3=(ba)^3$ commutes with all $A_{ij}$; so $\PB_3\cong\Z\times F_2$.

The braid group admits a natural generalisation to surfaces. We follow \cite{GP}. Let $S$ be a compact, connected, oriented surface of genus $g\ge0$ with $p\ge1$ boundary components. Let $$\F_n(S):=\{(p_1,\dots,p_n)\in S^n\mid p_i\ne p_j\text{ for }i\ne j\}$$
be the configuration space of $n$ ordered points on $S$.
The symmetric group $S_n$ acts freely on $\F_n(S)$ by coordinate permutation. The quotient $$\operatorname{D}_n(S):=\F_n(S)/S_n$$ is the configuration space of $n$ unordered points on $S$.

\begin{definition}[Fox-Neuwirth]\label{def:surfacebraidgroup}
	The \emph{surface braid group} on $S$ with $n$ strands is $$\B_n(S):=\pi_1(\operatorname{D}_n(S)),$$ and its pure version is $$\Pure_n(S):=\pi_1(\F_n(S)).$$
\end{definition}

\begin{rmk}
	For the disc $S=D^2$, these coincide with the classical groups: $\B_n(D^2)=B_n$ and $\Pure_n(D^2)=\PB_n$.
\end{rmk}

\begin{lemma}\label{lem:PBsplit}
	There is an isomorphism $$\PB_n\simeq \Pure_{n-2}(\bC\setminus\{0,1\})\rtimes_\varphi\Z.$$
\end{lemma}

\begin{proof}
	Consider the short exact sequence $$1\longrightarrow \Pure_{n-2}(\bC\setminus\{0,1\})\xrightarrow{j}\PB_n\xrightarrow{p_*}\PB_2\longrightarrow1,$$ where $j$ freezes the first two strands at $0$ and $1$, and $p_*$ forgets the remaining $n-2$ strands. Since $\PB_2\cong\Z$, the sequence splits, yielding the claimed semidirect product.
\end{proof}

\medskip
The groups introduced here reappear later in two ways:
as kernels of the forgetful map
$$F_M:\MCG(\mathbf{S}_\Delta)\to\MCG(\mathbf{S}),$$
and as explicit fundamental groups of certain strata of quadratic differentials.
For example, some strata give rise to a braid group on three strands, while others yield the pure braid group on three strands.

\subsection{Surfaces}\label{subsection:surfaces}

We begin by recalling the classes of surfaces on which all of our combinatorial and geometric constructions will take place.
Throughout the paper, these surfaces serve as the underlying topological models for the strata of quadratic differentials we study. In particular:
\begin{itemize}
	\item \emph{marked surfaces} model the real blow-up at higher–order poles;
	\item \emph{decorated marked surfaces} additionally record the zeroes of a differential as interior marked points; and
	\item \emph{weighted decorated marked surfaces} refine this further by encoding the order of each zero via a weight function.
\end{itemize}
Our treatment of marked and decorated marked surfaces follows §3.1-3.4 and §4.1 in \cite{KQ2}, while the weighted setup--needed to handle higher-order zeroes-- follows §3.2 in \cite{BMQS}.
These three levels of structure support three corresponding notions of combinatorics: triangulations, decorated triangulations, and  $\mathbf{w}$-mixed–angulations, which in turn give rise to the various exchange graphs studied later.

\begin{definition}\label{def:MS}
	An \emph{(unpunctured) marked surface} 
	$$\mathbf{S}=(S,\mathbf{M})$$ 
	is a compact connected oriented smooth surface~$S$ with nonempty boundary~$\partial S$, together with a finite set $\mathbf{M}\subset\partial S$ of marked points such that every boundary component contains at least one marked point.  
	Up to diffeomorphism, $\mathbf{S}$ is determined by its genus~$g$, the number~$b$ of boundary components, and the distribution of $m=|\mathbf{M}|$ marked points among them.
\end{definition}

In the framework of quadratic differentials, one often enhances a marked surface by adding interior points which record zeros of a differential.  
These lead to the notion of a decorated marked surface.

\begin{definition}\label{def:DMS}
	A \emph{decorated marked surface} 
	$$\mathbf{S}_\Delta=(\mathbf{S},\Delta)$$ 
	is a marked surface~$\mathbf{S}$ together with a finite set $\Delta\subset\operatorname{int}(\mathbf{S})$ of \emph{decorating points}.  
	The decorating points play the role of zeros of a quadratic differential.  
	For a surface of genus~$g$ with $b$ boundary components and $m$ boundary marked points, we define
	$$
	\aleph = 4g - 4 + 2b + m.
	$$
	In the simplest case, corresponding to differentials with only simple zeros, one has $|\Delta|=\aleph$.  
	In general, however, we allow any finite subset $\Delta\subseteq\operatorname{int}(\mathbf{S})$ with $|\Delta|\le\aleph$, to accommodate higher-order zeros.
\end{definition}

Finally, we introduce weights to record the orders of these zeros.  
This leads to the general notion of a \emph{weighted decorated marked surface}.

\begin{definition}\label{wDMS}
	A \emph{weighted decorated marked surface} (abbreviated \emph{wDMS})
	$$
	\mathbf{S}_{\mathbf{w}} = (\mathbf{S},\Delta,\mathbf{w})
	$$
	is a decorated marked surface~$\mathbf{S}_\Delta$ together with a \emph{weight function}
	$$
	\mathbf{w}\colon \Delta \longrightarrow \mathbb{Z}_{\ge -1}.
	$$
	We write $r = |\Delta|$ for the number of decorating points (also called \emph{finite critical points}) and
	$$
	\|\mathbf{w}\| = \sum_{Z\in\Delta} \mathbf{w}(Z)
	$$
	for their total weight.
	The weight function~$\mathbf{w}$ is said to be \emph{compatible} with~$\mathbf{S}$ if
	$$
	\|\mathbf{w}\| - (m + 2b) = 4g - 4,
	$$
	where $g$, $b$, and $m=|\mathbf{M}|$ denote the genus, the number of boundary components, and the number of marked points, respectively.  
	In the case where every zero is simple, all weights are~$1$, and we recover the ordinary decorated marked surface~$\mathbf{S}_\Delta$.
\end{definition}

\subsection{Mapping Class Groups}

Let $S$ be a compact, connected, oriented smooth surface.
Denote by $\Homeo^+(S,\partial S)$ the group of orientation-preserving homeomorphisms
of $S$ that restrict to the identity on~$\partial S$.
\begin{definition}\label{def:MCG}
	The \emph{mapping class group} of~$S$ is
	\[
	\MCG(S) := \Homeo^+(S,\partial S)/\Homeo_0(S,\partial S),
	\]
	where $\Homeo_0(S,\partial S)$ is the identity component.
	Replacing homeomorphisms by diffeomorphisms or isotopies by homotopies yields the same
	group.
\end{definition}
\medskip

Let $\mathbf{S}$ be a marked surface with marked points
$\mathbf{M}\subset\partial S$, and let
$\mathbf{S}_\Delta$ be a decorated marked surface with
decorations $\Delta$ in the interior of~$S$.  

\begin{definition}\label{def:MCG_S_S_Delta}
	We define
	\[
	\MCG(\mathbf{S}) := \Homeo^+(\mathbf{S})/\text{isotopy},\qquad
	\MCG(\mathbf{S}_\Delta) := \Homeo^+(\mathbf{S}_\Delta)/\text{isotopy},
	\]
	where $\Homeo^+(\mathbf{S})$ (resp. $\Homeo^+(\mathbf{S}_\Delta)$) denotes the group
	of orientation-preserving homeomorphisms of $\mathbf{S}$ fixing $\mathbf{M}$ (resp.
	$\mathbf{M}\cup\Delta$) setwise.  
\end{definition}

Forgetting the decorations gives a natural homomorphism
\[
F_M:\MCG(\mathbf{S}_\Delta)\longrightarrow \MCG(\mathbf{S}).
\]

\begin{lemma}\label{lem:kernel_F}
	The kernel of the forgetful map
	$F_M:\MCG(\mathbf{S}_\Delta)\to\MCG(\mathbf{S})$
	is the surface braid group (see Definition~\ref{def:surfacebraidgroup})
	\[
	\B_n(\mathbf{S})=\pi_1(D_n(S),\Delta),
 	\]
	where $D_n(S)$ is the configuration space of $n=|\Delta|$ unordered points on~$S$.
\end{lemma}

\begin{proof}
	Since $\MCG(\mathbf{S}_\Delta)$ and $\MCG(\mathbf{S})$ act in the same way on the marked points $\mathbf{M}$, we can restrict to the subgroups which fix $\mathbf{M}$ pointwise. However, any such (orientation preserving) diffeomorphism is isotopic to one that fixes the boundary $\partial S$ pointwise.
	Hence the required kernel is the same as the kernel of the forgetful map $\MCG(\mathbf{S}_\Delta, \partial S) \rightarrow \MCG(\mathbf{S}, \partial S)$ between the mapping class groups that fix the boundary. 
	Next the map $\Psi: \Homeo^+(S) \rightarrow D_n(S)$ defined by $\Psi(f) = f(\Delta)$ is a locally trivial fibre bundle, whose fibre over $\Delta$ is equal to $\Homeo^+(S, \Delta)$. Taking the long
	exact sequence in homotopy of this fibration yields:
	\begin{center}
		\begin{tikzcd}[column sep=tiny]
			& ...\rar	& \pi_1(\Homeo^+(S, \Delta)) \rar & \pi_1(\Homeo^+(S)) \rar & \pi_1(D_n(S), \Delta)\ar[out=0, in=180, looseness=2, overlay]{dll} \\
			&  & \pi_0(\Homeo^+(S, \Delta)) \rar & \pi_0(\Homeo^+(S)) \rar & \pi_0(D_n(S), \Delta) \rar & 0
		\end{tikzcd}
	\end{center}
	If $\pi_1(\Homeo^+(S)) = 1$ and $\pi_0(\D_n(S), \Delta) = 1$ hold, this reduces to the short exact sequence 
	\begin{center}
		\begin{tikzcd}[column sep=small]
			1 \rar & \pi_1(D_n(S), (\Delta)) \rar & \MCG(\mathbf{S}_\Delta, \partial S) \rar & \MCG(\mathbf{S}, \partial S) \rar & 1
		\end{tikzcd}
	\end{center}
	For more details see § 2.4 (5) from \cite{GP}. 
\end{proof}

\begin{definition}\label{def:MCG_S_w}
	Let $\mathbf{S_w}=(S,\mathbf{M},\Delta)$ be a weighted decorated marked surface.
	The \emph{weighted mapping class group} $\MCG(\mathbf{S_w})$ is the group of isotopy classes of orientation-preserving homeomorphisms of~$S$ that fix $\mathbf{M}$ and $\Delta$ setwise and preserve the weights:
	$$ \MCG(\mathbf{S_w})=\Homeo^+(\mathbf{S_w})/\text{isotopy}.$$
\end{definition}

%% file: sec_eg.tex
\section{Exchange Graphs}\label{sec:eg}

In this section we develop the combinatorial structures associated with the surfaces introduced in Subsection~\ref{subsection:surfaces}.
For each type of surface—marked, decorated, and weighted decorated—we equip it with a corresponding class of arc systems (triangulations or mixed-angulations), together with a flip operation that replaces one arc by another.
The resulting combinatorics is encoded in the exchange graph, whose vertices are the arc systems and whose edges correspond to flips.

For marked and decorated surfaces we follow the framework of
§3.1–§3.4 of \cite{KQ2}, while in the weighted setting we adopt the
mixed–angulation model of §3.3 of \cite{BMQS} to define the exchange
graph.
In each setting the exchange graph carries a set of distinguished loops, arising from the possible local configurations of pairs of arcs.
These loops—square, pentagon, and (in certain cases) hexagon loops—play a central role throughout the paper: they encode local relations among flips, and the quotient of the fundamental group of the exchange graph by these loops will be studied and later compared with the fundamental groups of strata of quadratic differentials.

We begin with the exchange graphs of marked surfaces, both in their unoriented and oriented versions, then treat the decorated case, and finally extend the construction to weighted decorated marked surfaces, where an additional type of hexagon loop appears.
This prepares the ground for Section~\ref{sec:quad}, where these combinatorial models will be related to framed quadratic differentials and to the topology of their moduli spaces.


\subsection{Unoriented Exchange Graph of $\mathbf{S}$}\label{subsection:unorientedEG}

We first treat the unoriented exchange graph of a marked surface, following~\cite{KQ2}. 
In this setting the vertices are triangulations and edges correspond to flips.  
This case contains the essential combinatorics underlying all other exchange graphs and will serve as the starting point for the oriented, decorated, and weighted versions that follow.  
At the end of this subsection we show that the quotient of the fundamental group of the exchange graph by the square and pentagon loops is trivial (Theorem~\ref{thm:MS_pi1EGmodRel_unoriented}).

\begin{definition}\label{def:MS_triangulation}
	Let $\mathbf{S}$ be a marked surface.
	\begin{itemize}
		\item An \emph{open arc} on~$\mathbf{S}$ is the isotopy class of a simple, essential curve in $\mathbf{S}\setminus\partial S$ whose endpoints lie in~$\mathbf{M}$.  
	
		\item An open arc is \emph{essential} if it is not homotopic (rel.\ endpoints) to a boundary segment or a point.  
	
	 	\item Two arcs are \emph{compatible} if they can be represented by disjoint curves, except possibly at their endpoints.
	
		\item A \emph{triangulation}~$T$ of~$\mathbf{S}$ is a maximal collection of pairwise compatible open arcs.
	\end{itemize}
\end{definition}

\medskip
\noindent
Each triangulation subdivides the surface into finitely many ideal triangles, whose sides are either open arcs or boundary segments.  
Any two triangulations are related by a finite sequence of \emph{flips}, each of which replaces a single arc by another as follows.

\begin{figure}[h]
	\centering
	\begin{tikzpicture}[scale=.3]
		\coordinate (A) at (1,1);
		\foreach \ang in {45,135,225,315}{
			\coordinate (P\ang) at ($(A) + (\ang:4)$);
			\fill (P\ang) circle (1.5mm);
		}
		\draw[thick] (P45) -- (P135);
		\draw[thick] (P225) -- (P315);
		\draw[thick] (P135) -- (P225);
		\draw[thick] (P315) -- (P45);
		\draw (P225) -- (P45);
		\draw ($(A) + (4,0)$) --($(A) + (7,0)$);
		\node at ($(A) + (5.5,-0.5)$) {\tiny flip}; 
		
		\node[above] at ($(P225)!0.5!(P45)$) {$\gamma$}; 
		
		\coordinate (A) at (12,1);
		\foreach \ang in {45,135,225,315}{
			\coordinate (P\ang) at ($(A) + (\ang:4)$);
			\fill (P\ang) circle (1.5mm);
		}
		\draw[thick] (P45) -- (P135);
		\draw[thick] (P225) -- (P315);
		\draw[thick] (P135) -- (P225);
		\draw[thick] (P315) -- (P45);
		\draw (P135) -- (P315);
		
		\node[above] at ($(P135)!0.5!(P315)$) {$\gamma'$};
	\end{tikzpicture}
	\caption{A flip: replacing one diagonal of a quadrilateral by the other.}
	\label{fig:forwardflip_unoriented}
\end{figure}
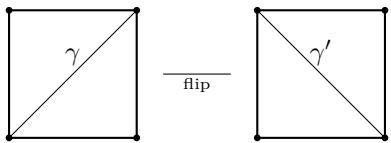

\begin{definition}\label{MS_EG_unoriented}
	The (unoriented) \emph{exchange graph} $\underline{\EG}(\mathbf{S})$ is the connected graph whose vertices correspond to triangulations of~$\mathbf{S}$ and whose edges correspond to flips between them, that is, to the replacement of one arc~$\gamma\in T$ by the opposite diagonal~$\gamma'$ of the quadrilateral formed by the two triangles adjacent to~$\gamma$ as seen in Figure~\ref{fig:forwardflip_unoriented}.
	Each vertex has valency~$|T|$, the number of arcs in any triangulation of~$\mathbf{S}$, so the graph is $|T|$--regular.
\end{definition}

\medskip
\noindent\textbf{Relations in $\pi_1\underline{\EG}(\mathbf{S})$.}

There are two basic kinds of closed loops that appear in the exchange graph.  
Each such loop is determined by a choice of two distinct arcs in a triangulation: either the two arcs intersect in no triangle, producing what we call a \emph{square} loop, or they intersect in exactly one triangle, giving rise to a \emph{pentagon} loop.  
We refer to these as square and pentagon \textit{relations} in anticipation of the later step where we quotient the fundamental group by those loops.
The terminology reflects the shape of the corresponding loop in the exchange graph.

\begin{figure}[h]
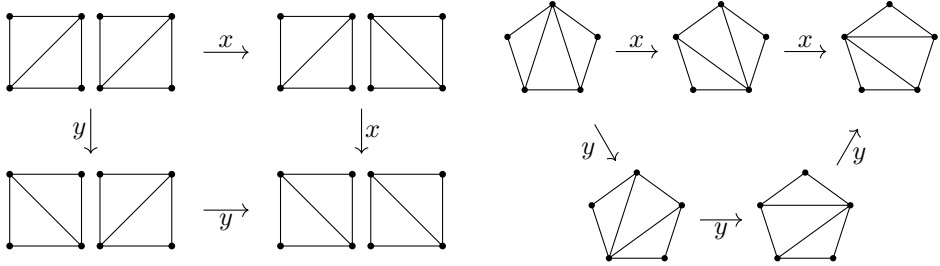
 
	\centering 
	\begin{minipage}{0.48\textwidth} 
		\centering 
		\input{3_Paper_S/Pictures/square_relation} 
	\end{minipage}
	\hfill 
	\begin{minipage}{0.48\textwidth} 
		\centering 
		\input{3_Paper_S/Pictures/pentagon_relation} 
	\end{minipage} 
	\caption{left: Square relation,  right: Pentagon relation} 
	\label{fig:squareandpentagon}
\end{figure}

\begin{center}
	\renewcommand{\arraystretch}{1.2}
	\begin{tabular}{c|c|c}
		Local configuration & Relation in $\pi_1\underline{\EG}(\mathbf{S})$ & Shape in $\underline{\EG}(\mathbf{S})$ \\ \hline
		Two arcs intersect in no triangle & $ (y^{2})^{-1} \cdot x^2$ & square: Figure~\ref{fig:squareandpentagon} \\
		Two arcs intersect once in one triangle & $ (y^{-3})^{-1} \cdot x^2$ & pentagon: Figure~\ref{fig:squareandpentagon} 
	\end{tabular}
\end{center}

\begin{definition}\label{def:MS_Rel_unoriented}
	Let $\pi_1\underline{\EG}(\mathbf{S})$ denote the fundamental group of the unoriented exchange graph.  We denote by $$\underline{\Rel}(\mathbf{S})\subseteq \pi_1\underline{\EG}(\mathbf{S})$$ the
	normal subgroup generated by all square and pentagon loops:
	\[
	\underline{\Rel}(\mathbf{S}) := \langle
	\text{square, pentagon loops in } \underline{\EG}(\mathbf{S})\rangle.
	\]
\end{definition}

We end this paragraph with the
following theorem.

\begin{theorem}[{\cite[Thm.\,3.5]{KQ2}; see also \cite[Thm.\,3.10]{FST}}]\label{thm:MS_pi1EGmodRel_unoriented}
	The fundamental group of the exchange graph $\underline{\EG}(\mathbf{S})$ is generated by the square and pentagon relations. In other words $$\pi_1\underline{\EG}(\mathbf{S})/\underline{\Rel}(\mathbf{S}) = 1.$$
\end{theorem}

Following the arguments of \cite{chekhov2007quantizing} we first define the following simplicial complex.
Let $\mathcal{A}(\mathbf{S})$ denote the simplicial complex whose vertices are isotopy classes of
arcs in~$\mathbf{S}$ and whose $k$-simplices are collections of $k+1$ pairwise compatible arcs.

\begin{lemma}\label{lem:1-skelt}
	The $1$-skeleton of the dual complex of $A(\mathbf{S})$ is exactly $\underline{\EG}(\mathbf{S})$.
\end{lemma}

\begin{proof}
	The $N$-simplices correspond to triangulations, with $N$ being the number of arcs in a maximal arc complex. The $(N-1)$-simplices are triangulations where one edge has been removed. There are only two ways (up to homotopy) to extend such an arcs system to a triangulation, and both ways differ by a flip at the new arc. Hence, each $(N-1)$-simplex lies in exactly two $N$-simplices and the arc systems corresponding to those two $(N-1)$-simplices differ by a flip.
\end{proof}

Now we can prove the theorem.

\begin{proof}[Proof of the Theorem]
	By Lemma~\ref{lem:1-skelt} a closed path in $\underline{\EG}(\mathbf{S})$ can be seen as a closed path in the interior of $N$-simplices and $(N-1)$-simplices $A(\mathbf{S})$ with base point in the interior of an $N$-simplex.
	This loop can be homotoped through the interior of $(N-2)$-simplices.
	Now there are two types of $(N-2)$-simplices. We recall that an $(N-2)$-simplices corresponds to a maximal arc system with two arcs removed. 
	\begin{enumerate}
		\item The arc system consists of two $4$-gon and the rest are triangles, or
		\item the arc system consist of one $5$-gon and the rest are triangles.
	\end{enumerate}
	In case (1) the $(N-2)$-simplex lies in exactly four $N$-simplices and four $(N-1)$-simplices. This will give us the square relation.
	In case (2) the $(N-2)$-simplex lies in exactly five $N$-simplices and five $(N-1)$-simplices. This will give us the pentagon relation.
	Together with the connectivity of the graph this proves the claim.
\end{proof}

\subsection{Oriented Exchange Graph of $\mathbf{S}$}\label{subsec:MS_EG}

We now refine the unoriented exchange graph by recording the direction of flips.  
In the oriented exchange graph of a marked surface, each flip is replaced by a pair of oppositely oriented edges, producing an oriented graph.
Besides the square and pentagon loops already present in the unoriented case, \textit{hexagon loops} appear when two arcs are adjacent in exactly two triangles. 

In this subsection we introduce the \emph{triangulation braid group}, a purely
topological analogue of the cluster braid group of \cite{KQ2}.  
We deliberately avoid any cluster-theoretic terminology and formulate everything in
terms of surfaces.  
We then prove that the triangulation braid group is precisely the quotient of the
fundamental group of the oriented exchange graph by the square, pentagon, and
hexagon relations.

\begin{definition}\label{MS_EG}
	The \emph{oriented exchange graph}~$\EG(\mathbf{S})$ is obtained from the unoriented exchange graph~$\underline{\EG}(\mathbf{S})$ by replacing each edge corresponding to a flip with a $2$--cycle of oppositely  oriented edges. In particular, every flip now appears as a pair of mutually inverse oriented moves.
\end{definition}

\medskip
\noindent\textbf{Relations in $\pi_1\EG(\mathbf{S})$.}

In the oriented exchange graph a third type of closed loop appears.  
Each loop is again determined by a choice of two distinct arcs in a triangulation.  
Besides the square and pentagon loops, there is now a \emph{hexagon} loop, which occurs precisely when the two arcs intersect in exactly two triangles, i.e. when they share a pair of consecutive triangles forming a topological annulus.  
As before, we refer to the square, pentagon, and hexagon loops as \textit{relations} in anticipation of the later quotient in which these oriented cycles are null-homotopic.  
Their names again record the combinatorial shape of the corresponding loops in the oriented exchange graph.

\begin{figure}[h]
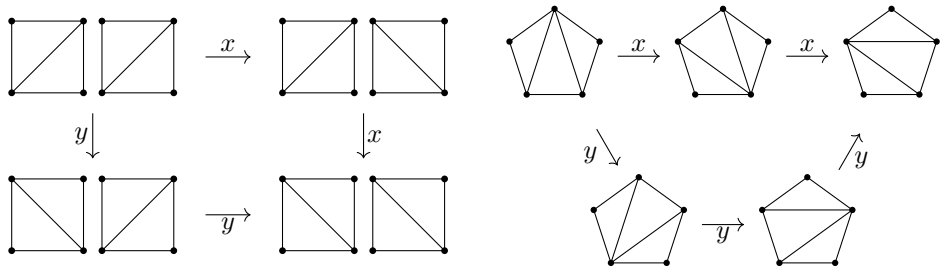
 
	\centering 
	\begin{minipage}{0.48\textwidth} 
		\centering 
		\input{3_Paper_S/Pictures/square_relation} 
	\end{minipage}
	\hfill 
	\begin{minipage}{0.48\textwidth} 
		\centering 
		\input{3_Paper_S/Pictures/pentagon_relation} 
	\end{minipage} 
	\caption{left: Square relation,  right: Pentagon relation} 
	\label{fig:squareandpentagon}
\end{figure}

\begin{figure}[ht]
	\centering
	\input{3_Paper_S/Pictures/hexagon_relation}
	\caption{Hexagon relation}
	\label{fig:hexagon}
\end{figure}

\begin{center}
	\renewcommand{\arraystretch}{1.2}
	\begin{tabular}{c|c|c}
		Local geometry & Relation in $\pi_1\EG(\mathbf{S})$ & Shape in $\EG(\mathbf{S})$ \\ \hline
		Two arcs intersect in no triangle & $(y^2)^{-1} \cdot x^2$ & square: Figure~\ref{fig:squareandpentagon} \\
		Two arcs intersect once in one triangle & $(y^3)^{-1} \cdot x^2$ & pentagon: Figure~\ref{fig:squareandpentagon} \\
		Two arcs intersect once in each of two& $(y x^2)^{-1} \cdot x^2 y$ & hexagon: Figure~\ref{fig:hexagon}\\
	 	triangles
	\end{tabular}
\end{center}

\medskip

Under the natural projection from the oriented to the unoriented exchange graph
$\EG(\mathbf{S})\!\to\!\underline{\EG}(\mathbf{S})$, we use the same terminology for the images of these
relations, referring to them again as squares, pentagons, and hexagons.  
Note, however, that in this projection the oriented paths $xy$ and $yx$ lie in the
kernel, so the hexagon relation $(y x^2)^{-1} \cdot x^2 y$ becomes trivial in $\underline{\EG}(\mathbf{S})$.

\begin{definition}\label{def:MS_Rel}
	Let $\pi_1\EG(\mathbf{S})$ denote the fundamental group of the oriented exchange graph.  We denote by $$\Rel(\mathbf{S})\subseteq \pi_1\EG(\mathbf{S})$$ the
	normal subgroup generated by all square and pentagon loops:
	\[
	\Rel(\mathbf{S}) := \langle
	\text{square, pentagon, hexagon loops in } \EG(\mathbf{S})\rangle.
	\]
\end{definition}

We now want to describe the quotient group $\pi_1\EG(\mathbf{S})/\Rel(\mathbf{S})$. Our goal is to show that this group is generated by certain canonical loops, called \emph{local twists}, which correspond to flipping an edge twice in the
same direction.  The subgroup generated by these elements will be referred to as the \emph{triangulation braid group}—a direct translation of the \emph{cluster braid group} introduced in ~\cite{KQ2} into the setting of triangulations.
The statements and proofs in the remainder of this subsection are adapted from §2.4 in \cite{KQ2}, but rewritten in the language of triangulations.

\begin{definition}\label{def:triangulationbraidgroup}
	Fix a triangulation $T=\{e_i\}$.
	For each edge $e_i\in T$, the \emph{local twist} $t_i$ at $e_i$ is the length–two loop obtained by composing the two forward flips at $e_i$.
	Thus $$t_i\in\pi_1(\EG(\mathbf{S}),T)/\Rel(\mathbf{S}).$$
	We define the \emph{triangulation braid group}
	\[
	\TB(T)\ :=\ \langle\, t_i\mid e_i\in T\,\rangle\ \subseteq \ \pi_1(\EG(\mathbf{S}),T)/\Rel(\mathbf{S}).
	\]
\end{definition}

The generators of the triangulation braid group satisfy the familiar braid relations, which justifies its name.  In particular, the local twists associated with non-adjacent edges commute, while those corresponding to two edges intersecting once in a single triangle satisfy the standard braid relation.  This is summarised in the following lemma.

\begin{lemma}\label{lem:TB_braidrelation}
	Let $e_i,e_j\in T$.
	\begin{enumerate}
		\item If $e_i$ and $e_j$ intersect in no triangle of $T$, then $$t_it_j=t_jt_i$$ in $\pi_1(\EG(\mathbf{S}),T)/\Rel(\mathbf{S})$.
		\item If $e_i$ and $e_j$ intersect once in a triangle of $T$, then $$t_it_jt_i=t_jt_it_j$$ in $\pi_1(\EG(\mathbf{S}),T)/\Rel(\mathbf{S})$.
	\end{enumerate}
\end{lemma}

\begin{proof}
	The first relation follows from the following diagram:
	
	\begin{center}
		\input{3_Paper_S/Pictures/TB_square}
	\end{center}
	Since we are modding out the square relations the diagram commutes and following the graph first right then downwards is the same as first downwards and then right. Hence, $t_it_j = t_jt_i$.
	
	The second relation follows from this diagram:
	\begin{center}
		\input{3_Paper_S/Pictures/TB_pentagon}
	\end{center}
	This time the commutativity follows from modding out the pentagons. Hence, $t_it_jt_i=t_jt_it_j$.
\end{proof}

It turns out that the triangulation braid group is already the whole group:

\begin{prop}\label{prop:MS_pi1EGmodRel}
	The fundamental group of the quotient graph is generated by local twists. In other words
	\[
	\pi_1(\EG(\mathbf{S}),T)/\Rel(\mathbf{S}) = \TB(T) .
	\]
\end{prop}

Before proving it, we first define the conjugation map on the fundamental groups.  Given a forward flip
$T^\bullet \xrightarrow{z} T^\circ$ in $\EG(\mathbf{S})$. Conjugation by $z$ yields
\begin{align*}
	\ad_z:\ \pi_1(\EG(\mathbf{S}),T^\circ)/\Rel(\mathbf{S})&\xrightarrow{\ \sim\ }\pi_1(\EG(\mathbf{S}),T^\bullet)/\Rel(\mathbf{S}),\\
	t &\mapsto z^{-1}tz.
\end{align*}
Locally we may identify the edge sets of $T^\circ$ and $T^\bullet$ and denote this common set by $T_0$.  
We write $\{t^\circ_\ell\mid \ell\in T_0\}$ and
$\{t^\bullet_\ell\mid \ell\in T_0\}$ for the local twists generating $\TB(T^\circ)$ and $\TB(T^\bullet)$, respectively.
In Proposition~\ref{prop:conjugation_formula} below we describe explicitly how $\ad_z$ acts on these generators and show that it induces an isomorphism on the triangulation braid groups.

\begin{prop}[Conjugation formula]\label{prop:conjugation_formula}
	Suppose $z$ is the forward flip at the edge $k\in T_0$.
	Then, for any $\ell\in T_0$,
	\[
	\ad_z\!\left(t^\circ_\ell\right)\ =
	\begin{cases}
		\ (t^\bullet_k)^{-1}\,t^\bullet_\ell\,t^\bullet_k, & \text{if $\ell$ precedes $k$ clockwise in a common triangle of $T^\circ$;}\\[4pt]
		\ t^\bullet_\ell, & \text{otherwise.}
	\end{cases}
	\]
	In particular, $\ad_z$ restricts to an isomorphism $\TB(T^\circ)\xrightarrow{\ \sim\ }\TB(T^\bullet)$.
\end{prop}

\begin{proof}
	We only need to show the formula. Consider the case that there are no arrows from $l$ to $k$ first. This diagram  
	\begin{center}
		\input{3_Paper_S/Pictures/conj_form_square}
	\end{center}
	commutes in $\EG(\mathbf{S})/\Rel(\mathbf{S})$. Hence, $\ad_z(t^\circ_l) = z^{-1}t^\circ_\ell z = t^\bullet_l$. Now, suppose there are arrows from $l$ to $k$. 
	We are in the situation of the following diagram
	\begin{center}
		\input{3_Paper_S/Pictures/conj_form_penta}.
	\end{center}
	This commutes in $\EG(\mathbf{S})/\Rel(\mathbf{S})$. Hence, $\ad_z(t^\circ_l) = z^{-1}t^\circ_l z = (t_k^\bullet)^{-1}t^\bullet_l t_k^\bullet$.	
\end{proof}

We can now prove Proposition~\ref{prop:triangulationbraid}.

\begin{proof}[Proof of Proposition~\ref{prop:triangulationbraid}]
	Let $t\in\TB(T^\circ)$ and let $p:T\to T^\circ$ be a path in $\EG(\mathbf{S})$.
	If $p$ has length one, say $p=z$ is a forward flip, then by Proposition~\ref{prop:conjugation_formula} we have
	$$p^{-1}tp~=~\ad_z(t)~\in~\TB(T).$$
	By induction on the length of $p$ we conclude that $\TB(T)$ is a normal subgroup of $\pi_1(\EG(\mathbf{S},T)/\Rel(\mathbf{S})$ and that every conjugate of a local twist along any path lies in $\TB(T)$.
	
	Consider the canonical covering $\EG(\mathbf{S})\to \underline{\EG}(\mathbf{S})$ that forgets orientations of flips.
	Passing to fundamental groups and then to the quotient by $\Rel(\mathbf{S})$ yields a short exact sequence
	\[
	1\ \rightarrow\ \TB(T)\ \rightarrow\ \pi_1(\EG(\mathbf{S}),T)/\Rel(\mathbf{S})\ \rightarrow\ \pi_1(\underline{\EG}(\mathbf{S}),T)/\underline{\Rel}(\mathbf{S})\ \rightarrow\ 1.
	\]
	By Theorem~\ref{thm:MS_pi1EGmodRel_unoriented}, $\pi_1(\underline{\EG}(\mathbf{S}),T)/\underline{\Rel}(\mathbf{S})=1$.
	Hence $$\pi_1(\EG(\mathbf{S}),T)/\Rel(\mathbf{S})=\TB(T).$$
\end{proof}

\subsection{Exchange Graph of $\mathbf{S}_\Delta$}
We now pass from marked surfaces to decorated marked surfaces, where the interior decorating points record the simple zeroes of a quadratic differential.  
The presence of decorations enriches the combinatorics of flips: Although the basic notion of a triangulation remains unchanged, arcs do not homotope through the decorations.
As in the oriented case (see subsection~\ref{subsec:MS_EG}), the square, pentagon, and hexagon loops arise from pairs of arcs.

Any ideal triangulation of $\mathbf{S}$ consists of
$n = 6g - 6 + 3b + m$ open arcs—where $g$ is the genus, $b$ the number of 
boundary components, and $m$ the number of marked points—and divides 
$\mathbf{S}$ into $\aleph = (2n + m)/3$ triangles.

\begin{definition}\label{def:DMS_EG}
	Given a  decorated marked surface $\mathbf{S}_\Delta$, in other words, a marked surface
	$\mathbf{S}$ together with a fixed set $\Delta$ of $\aleph$ decorating points in the
	interior of $S$. 
	\begin{itemize}
		\item An \emph{open arc} in $\mathbf{S}_\Delta$ is (the isotopy class of) a simple
		essential curve in $\mathbf{S}_\Delta\setminus \Delta$ connecting two marked
		points in~$\mathbf{M}$.
		\item A \emph{triangulation} $\TT$ of $\mathbf{S}_\Delta$ is a maximal collection of
		pairwise compatible open arcs (i.e. no intersections in
		$\mathbf{S}\setminus\mathbf{M}$) that divides $\mathbf{S}_\Delta$ into $\aleph$
		triangles, each containing exactly one decoration in~$\Delta$.
		\item Forgetting the decorating points defines a map
		$F:\mathbf{S}_\Delta\to\mathbf{S}$, which induces a corresponding map between
		the sets of open arcs, since isotopy in $\mathbf{S}_\Delta$ implies isotopy in
		$\mathbf{S}$.
		\item The \emph{forward flip} of a triangulation $\TT$ is defined by moving the endpoints of an open arc
		$\gamma$ anticlockwise along the quadrilateral to obtain a new arc~$\gamma^\#$ (see Figure~\ref{fig:forwardflipo}).
		\item The \emph{exchange graph} $\EG(\mathbf{S}_\Delta)$ is the oriented graph whose vertices are triangulations of $\mathbf{S}_\Delta$ and whose edges
		correspond to forward flips. By $\EG^\TT(\mathbf{S}_\Delta)$ we mean the connected component containing $\TT$.
	\end{itemize}
\end{definition}

\begin{figure}[ht]
	\centering
	\input{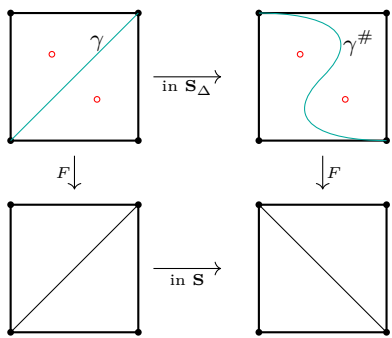}
	\caption{The forward flip in $\EG(\mathbf{S}_\Delta)$ and $\EG(\mathbf{S})$}
	\label{fig:forwardflipo}
\end{figure}

For any pair of arcs in a triangulation there are precisely three local configurations.
Each gives rise to a characteristic relation among the associated flips.

\medskip
\noindent\textbf{Relations in $\pi_1\EG(\mathbf{S_\Delta})$.} For decorated marked surfaces the same three types of loops occur in the exchange graph.   Each loop is again determined by a choice of two distinct arcs in a triangulation of $\mathbf{S}_\Delta$.  
If the arcs do not intersect in any triangle, we obtain a \emph{square relation}; if they intersect once in one triangle, we obtain a \emph{pentagon relation}; and if they intersect once in each of two triangles we obtain a \emph{hexagon relation}.  

\begin{figure}[h]
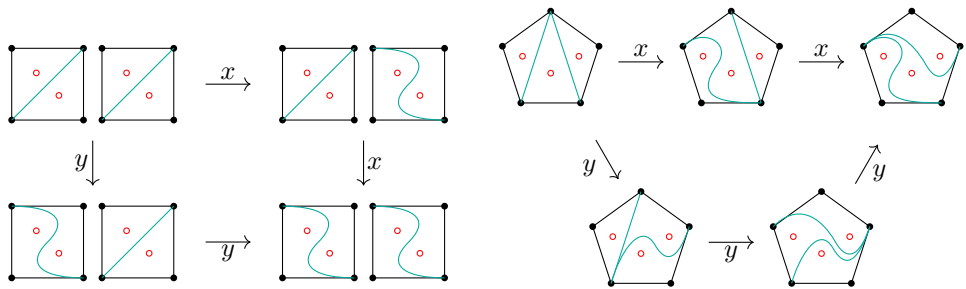
 
	\centering 
	\begin{minipage}{0.48\textwidth} 
		\centering 
		\input{3_Paper_S/Pictures/deco_square_relation} 
	\end{minipage}
	\hfill 
	\begin{minipage}{0.48\textwidth} 
		\centering 
		\input{3_Paper_S/Pictures/deco_pentagon_relation} 
	\end{minipage} 
	\caption{Left: Square relation,  Right: Pentagon relation} 
	\label{fig:dec_squareandpentagon}
\end{figure}

\begin{figure}[ht]
	\centering
	\input{3_Paper_S/Pictures/deco_hexagon_relation}
	\caption{Hexagon relation}
	\label{fig:dec_hexagon}
\end{figure}

\begin{center}
	\renewcommand{\arraystretch}{1.2}
	\begin{tabular}{c|c|c}
		Local geometry & Relation in $\pi_1\EG(\mathbf{S}_\Delta)$ & Shape in $\EG(\mathbf{S}_\Delta)$ \\ \hline
		Two arcs intersect in no triangle & $(y^2)^{-1} \cdot x^2$ & square: Figure~\ref{fig:dec_squareandpentagon} \\
		Two arcs intersect once in one triangle & $(y^3)^{-1} \cdot x^2$ & pentagon: Figure~\ref{fig:dec_squareandpentagon} \\
		Two arcs intersect once in each of two & $(y x^2)^{-1} \cdot x^2 y$ & hexagon: Figure~\ref{fig:dec_hexagon} \\ triangles
	\end{tabular}
\end{center}

\medskip

\begin{definition}\label{def:DMS_Rel}
	Let $\pi_1\EG^\TT(\mathbf{S}_\Delta)$ denote the fundamental group of the oriented exchange graph at a chosen base triangulation~$\TT$.  We denote by $$\Rel(\mathbf{S}_\Delta)\subseteq \pi_1\EG^\TT(\mathbf{S}_\Delta)$$ the normal subgroup generated by all square, pentagon and hexagon relations:
	\[
	\Rel(\mathbf{S}_\Delta) := \langle
	\text{square, pentagon, hexagon relations in } \EG^\TT(\mathbf{S}_\Delta)\rangle.
	\]
\end{definition}

\begin{rmk}
	Let $F:\EG^\TT(\mathbf{S}_\Delta)\to\EG(\mathbf{S})$ be the covering induced by forgetting the decorating points.  
	The three types of relations in the decorated exchange graph are precisely the lifts of the corresponding relations downstairs: every square, pentagon, and hexagon in $\EG^\TT(\mathbf{S}_\Delta)$ projects to, and arises as a lift of, a square, pentagon, or hexagon in $\EG(\mathbf{S})$.
	Consequently, the natural inclusion
	$\pi_1\EG^\TT(\mathbf{S}_\Delta)\hookrightarrow \pi_1\EG(\mathbf{S})$
	identifies the subgroups generated by these relations, that is,
	$$\Rel(\mathbf{S}_\Delta)\simeq\Rel(\mathbf{S}).$$
\end{rmk}

\subsection{Exchange Graph of $\sow$}\label{subsec:eg_mixed}

We now extend the exchange graph construction to weighted decorated marked surfaces. In this setting, each decorating point carries a weight, and a
$\mathbf{w}$–mixed–angulation specifies, for every decoration, a polygon whose
number of sides is determined by its weight.  
This notion, introduced in §3.3 in \cite{BMQS}, generalises triangulations on decorated marked surfaces.
The vertices of the weighted exchange graph are $\mathbf{w}$–mixed–angulations, and edges correspond to forward flips inside these $\mathbf{w}$–polygons.

Unlike in the unoriented, oriented, and decorated cases—where the local cycles
(square, pentagon, and hexagon relations) were already understood from \cite{KQ2}-the weighted setting requires a new analysis. Here we introduce a list of local loops for the first time.  
Besides the familiar square, pentagon, and hexagon relations, a second type of hexagon loop appears in the weighted setting, arising when two arcs intersect twice inside a single polygon.

\begin{definition} \label{def:wDSM_EG}
	Let $\mathbf{S}_{\mathbf{w}}=(\mathbf{S},\mathbf{M},\Delta,\mathbf{w})$ be a weighted decorated marked surface.
	
	\begin{itemize}
		\item A \emph{$\mathbf{w}$–mixed-angulation} $\MM$ of $\mathbf{S}_{\mathbf{w}}$ is an (isotopy class of) maximal collection of compatible (as defined in Definition~\ref{def:MS_triangulation}) open arcs dividing $\mathbf{S}$ into polygons, each polygon containing exactly one decorating point $Z\in\Delta$, and having precisely $\mathbf{w}(Z)$ sides.
		
		\item The \emph{forward flip} of a mixed–angulation is obtained, exactly as the forward flip of triangulation on a decorated marked surface: By moving the endpoints of an open arc $\gamma$ anticlockwise along the boundary of the $\mathbf{w}$–polygon containing~$\gamma$, producing a new arc $\gamma^\#$.
		
		\item The \emph{exchange graph} $\EG(\mathbf{S}_{\mathbf{w}})$ is the oriented graph whose vertices are $\mathbf{w}$–mixed-angulations and whose edges are forward flips. By $\EG^\circ(\sow)$ we mean a connected component.
	\end{itemize}
\end{definition}

\medskip
\noindent\textbf{Relations in $\pi_1\EG(\sow)$.} In the weighted exchange graph we have generalized versions of the three types of loops that occur for triangulations. Again each type is determined by a chosen pair of arcs. If the arcs do not intersect in any polygon, we obtain a \emph{square relation}; if they intersect once in one polygon, we obtain a \emph{pentagon relation}; and if they intersect once in each of two polygons we obtain a \emph{hexagon relation}.  
But now we have a \emph{second type of hexagon relation} when two arcs intersect twice inside a single polygon. 

We refer to these loops in $\EG(\sow)$ as the square, pentagon, and (second) hexagon relations, since we will quotient the exchange graph by the normal subgroup generated by them.

\begin{figure}[h]
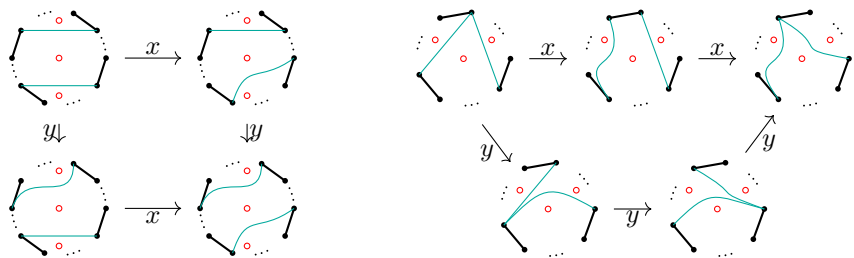
 
	\centering 
	\begin{minipage}{0.48\textwidth} 
		\centering 
		\input{3_Paper_S/Pictures/gen_square_relation_deco} 
	\end{minipage}
	\hspace{0.25em}
	\begin{minipage}{0.48\textwidth} 
		\centering 
		\input{3_Paper_S/Pictures/gen_pentagon_relation_deco} 
	\end{minipage} 
	\caption{Left: Square relation,  Right: Pentagon relation} 
	\label{fig:gen_squareandpentagon_deco}
\end{figure}

\begin{figure}[h]
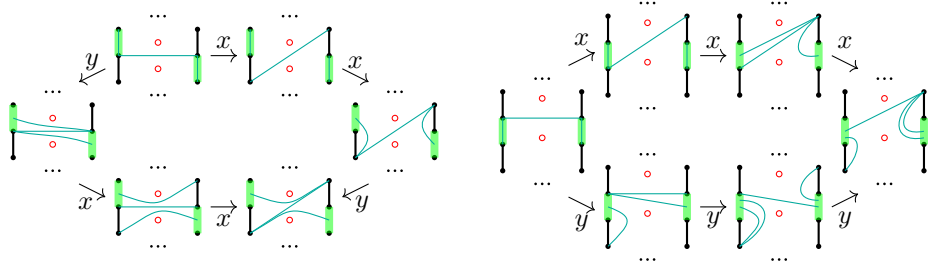
 
	\centering 
	\begin{minipage}{0.48\textwidth} 
		\centering 
		\input{3_Paper_S/Pictures/gen_hex1_relation_deco} 
	\end{minipage}
	\hspace{0.2cm} 
	\begin{minipage}{0.48\textwidth} 
		\centering 
		\input{3_Paper_S/Pictures/gen_hex2_relation_deco} 
	\end{minipage} 
	\caption{Left: First Hexagon relation,  Right: Second Hexagon relation} 
	\label{fig:gen_hexagons_deco}
\end{figure}

\begin{center}
	\renewcommand{\arraystretch}{1.2}
	\begin{tabular}{c|c|c}
		Local geometry & Relation in $\pi_1$ & Shape in $\EG(\sow)$ \\ \hline
		Two arcs intersect in no polygon & $(y^2)^{-1}\cdot x^2$ & square: Figure~\ref{fig:gen_squareandpentagon_deco} \\
		Two arcs intersect once in one polygon & $(y^3)^{-1} \cdot x^2$ & pentagon: Figure~\ref{fig:gen_squareandpentagon_deco} \\
		Two arcs intersect once in each of two  & $(y x^2)^{-1} \cdot x^2 y$ & hexagon: Figure~\ref{fig:gen_hexagons_deco}\\
		polygons & & \\
		Two arcs intersect twice in one polygon  & $(y^3)^{-1} \cdot x^3$ & hexagon: Figure~\ref{fig:gen_hexagons_deco} \\ 
	\end{tabular}
\end{center}

\begin{rmk}
	The first three relations are direct generalisations of the square, pentagon, and hexagon relations known from the triangulated case (Definition~\ref{def:MS_Rel}). The second hexagon relation, on the other hand, is new and arises only in the presence of at least one zero of order at least $2$.
\end{rmk}

\begin{definition}\label{def:wDMS_Rel}
	Let $\pi_1\EG^\circ(\sow)$ denote the fundamental group of a connected component of the exchange graph of $\mathbf{w}$-mixed-angulations.  We denote by $$\Rel(\sow)\subseteq \pi_1\EG^\circ(\sow)$$ the normal subgroup generated by all square and pentagon loops:
	\[
	\Rel(\sow) := \langle
	\text{square, pentagon, (second type) hexagon relations in } \EG(\sow)\rangle.
	\]
\end{definition}

\begin{definition}
	By forgetting the decorating points, we obtain a new exchange graph that records $\mathbf{w}$–mixed–angulations of the underlying undecorated surface. More precisely, forgetting $\Delta$ defines a covering
	$$F:\EG(\mathbf{S}_{\mathbf{w}})\longrightarrow \EG(\mathbf{S}, \mathbf{w}),$$
	where $\EG(\mathbf{S}, \mathbf{w})$ is the exchange graph whose vertices are $\mathbf{w}$–mixed–angulations of $\mathbf{S}$, considered up to isotopy through the decorating points~$\Delta$. In other words, arcs in $\EG(\mathbf{S}, \mathbf{w})$ may be homotoped through~$\Delta$. This generalises the covering $\EG^\TT(\mathbf{S}_\Delta)\to\EG(\mathbf{S})$ from the simple–weight case.
\end{definition}

Similarly as before we have four kinds of loops in the exchange graph.

\begin{figure}[h] 
	\centering 
	\begin{minipage}{0.48\textwidth} 
		\centering 
		\input{3_Paper_S/Pictures/gen_square_relation} 
	\end{minipage}
	\hspace{0.25em}
	\begin{minipage}{0.48\textwidth} 
		\centering 
		\input{3_Paper_S/Pictures/gen_pentagon_relation} 
	\end{minipage} 
	\caption{Left: Square relation,  Right: Pentagon relation} 
	\label{fig:gen_squareandpentagon}
\end{figure}

\begin{figure}[h]
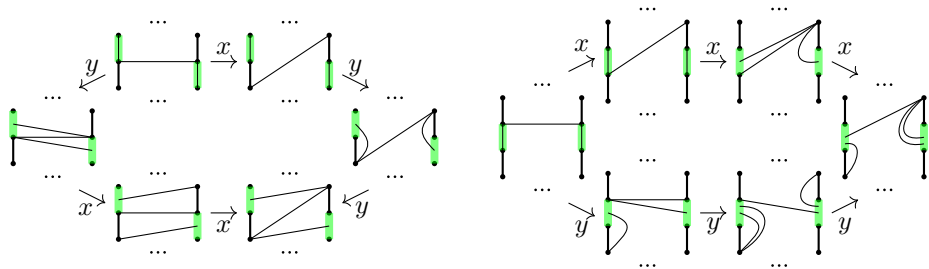
 
	\centering 
	\begin{minipage}{0.48\textwidth} 
		\centering 
		\input{3_Paper_S/Pictures/gen_hex1_relation} 
	\end{minipage}
	\hspace{0.2cm} 
	\begin{minipage}{0.48\textwidth} 
		\centering 
		\input{3_Paper_S/Pictures/gen_hex2_relation} 
	\end{minipage} 
	\caption{Left: First hexagon type,  Right: Second hexagon type} 
	\label{fig:gen_hexagons}
\end{figure}

\begin{center}
	\renewcommand{\arraystretch}{1.2}
	\begin{tabular}{c|c|c}
		Local geometry & Relation in $\pi_1$ & Shape in $\EG(\mathbf{S}, \mathbf{w})$ \\ \hline
		Two arcs intersect in no polygon & $(y^2)^{-1}\cdot x^2$ & square: Figure~\ref{fig:gen_squareandpentagon} \\
		Two arcs intersect once in one polygon & $(y^3)^{-1} \cdot x^2$ & pentagon: Figure~\ref{fig:gen_squareandpentagon} \\
		Two arcs intersect once in each of two polygons & $(y x^2)^{-1} \cdot x^2 y$ & hexagon: Figure~\ref{fig:gen_hexagons}\\
		Two arcs intersect twice in one polygon  & $(y^3)^{-1} \cdot x^3$ & hexagon: Figure~\ref{fig:gen_hexagons} \\ 
	\end{tabular}
\end{center}

\medskip

We define the local relations in $\EG(\mathbf{S}, \mathbf{w})$ as the images under~$F$ of the corresponding relations in $\EG(\mathbf{S}_{\mathbf{w}})$. 
Each square, pentagon, and hexagon relation in $\EG(\mathbf{S}_{\mathbf{w}})$ therefore projects to a relation in $\EG(\mathbf{S}, \mathbf{w})$, and every relation on the base arises in this way. See Figure~\ref{fig:gen_squareandpentagon} and Figure~\ref{fig:gen_hexagons}. The normal subgroup generated by these relations will be denoted by $\Rel(\mathbf{S}, \mathbf{w})$.

\begin{rmk}
	The natural inclusion
	$$\pi_1\EG(\mathbf{S}_{\mathbf{w}})\hookrightarrow \pi_1\EG(\mathbf{S}, \mathbf{w})$$
	induces a natural isomorphism between the corresponding relation subgroups:
	$$\Rel(\mathbf{S}_{\mathbf{w}})\simeq \Rel(\mathbf{S}, \mathbf{w}).$$
\end{rmk}

\subsection{Braid Twist Group}

In this subsection our aim is to establish two statements that will be needed later.
The first concerns the injective homomorphism
\[
F_*:\;
\pi_1\EG^\TT(\mathbf{S}_\Delta)/\Rel(\mathbf{S}_\Delta)
\longrightarrow
\pi_1\EG(\mathbf{S})/\Rel(\mathbf{S}),
\]
induced by the forgetful map
\(F:\EG^\TT(\mathbf{S}_\Delta)\rightarrow \EG(\mathbf{S})\).
We introduce the \emph{braid twist group} and show (Lemma~\ref{lem:BT_is_quotient}) that this is the quotient of $F_*$.

The second goal is to compare the braid twist group with the triangulation braid group introduced in Defintion~\ref{def:triangulationbraidgroup}.  
We prove (Proposition~\ref{prop:triangulationbraid}) that all relations occurring in the braid twist group also hold in the triangulation braid group, which will be essential later when we relate the exchange graphs to the fundamental group of the moduli space $\FQuad(\mathbf{S}_\Delta)$.

Within this triangulation context and avoiding cluster theory we verify the first six relations of Proposition~\ref{prop:triangulationbraid}, whereas the seventh relation is unclear and remains unproven here.
This does not weaken the overall mathematical narrative: in the cluster setting all seven relations are known to hold, and the triangulation braid group is simply a special case.

We follow §3.5-§3.6 in \cite{KQ2} closely, except for Proposition~\ref{prop:triangulationbraid}. In that step, we replace the cluster braid group of \cite{KQ2} by our purely topological notion of the triangulation braid group, in order to avoid any cluster-theoretic machinery. This requires a modification of the original argument, but the resulting statement remains formally analogous to its cluster-theoretic counterpart.

\begin{definition}
	Let $\mathbb{T}$ be a triangulation of $\mathbf{S}_\Delta$ consisting of $n$ open arcs.
	The \emph{dual graph} $\mathbb{T}^*$ is the collection of $n$ closed arcs in $\mathbf{S}_\Delta$ such that each closed arc intersects exactly one open arc of $\mathbb{T}$ once and no others. 
	\begin{figure}[ht]
		\centering
		\input{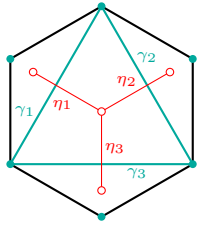}
		\caption{The dual graph of a triangulation.}
		\label{fig:dual_triangulation}
	\end{figure}
	More precisely, for each open arc $\gamma\in\mathbb{T}$, the dual closed arc $\eta\in\mathbb{T}^*$ lies in the quadrilateral with diagonal~$\gamma$, connecting the two decorating points in that quadrilateral and intersecting $\gamma$ exactly once (see Figure~\ref{fig:dual_triangulation}).
\end{definition}

\begin{definition}
	For any closed arc $\eta$ in $\mathbf{S}_\Delta$, there is the (positive) \emph{braid twist}
	$B_\eta\in\MCG(\mathbf{S}_\Delta)$ along~$\eta$ shown in
	Figure~\ref{fig:braidtwist}.
	The \emph{braid twist group} $$\BT(\mathbf{S}_\Delta)$$ is the subgroup of $\MCG(\mathbf{S}_\Delta)$ generated by all such twists.
	Given a triangulation $\mathbb{T}$, we write
	$$\BT(\mathbb{T})=\langle B_\eta\mid\eta\in\mathbb{T}^*\rangle$$ for the generating set
	associated with its dual arcs (cf.~\cite[Lem.~4.2]{qiu2016decorated}).
\end{definition}

\begin{figure}[ht]
	\centering
	\input{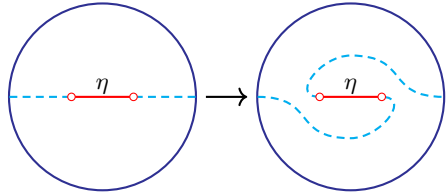}
	\caption{The braid twist $B_{\eta}$}
	\label{fig:braidtwist}
\end{figure}

\begin{lemma}\label{lem:BT_is_quotient}
	Let $F:\EG^\TT(\mathbf{S}_\Delta)\to \EG(\mathbf{S})$ be the forgetful map.
	Then the induced inclusion
	\[
	\pi_1\EG^\TT(\mathbf{S}_\Delta)/\Rel(\mathbf{S}_\Delta)
	\hookrightarrow
	\pi_1\EG(\mathbf{S})/\Rel(\mathbf{S})
	\]
	has quotient canonically isomorphic to the twisting group
	$\BT(\mathbf{S}_\Delta)$.
	Equivalently,
	\[
	\frac{\pi_1\EG(\mathbf{S})/\Rel(\mathbf{S})}{
		\pi_1\EG^\TT(\mathbf{S}_\Delta)/\Rel(\mathbf{S}_\Delta)
	}
	\;\cong\;
	\BT(\mathbf{S}_\Delta).
	\]
\end{lemma}

\begin{proof}
	Let $T:= F(\TT)$ the triangulation of the undecorated surface.
	Let $\mathcal{F}$ be the set of all decorated lifts of $T$ in
	$\EG^\TT(\mathbf{S}_\Delta)/\Rel(\mathbf{S}_\Delta)$.
	The group $\BT(\mathbf{S}_\Delta)$ does not alter the underlying triangulation of $\mathbf{S}$, hence acts on $\mathcal{F}$.
	Take any $p \in \pi_1\EG(\mathbf{S})/\Rel(\mathbf{S})$ and take its lift $\hat{p} : \TT \rightarrow \TT'$ to some $\TT' \in \cal{F}$.
	By Proposition~\ref{prop:MS_pi1EGmodRel}, $p$ can be expressed (up to homotopy) as a product of local twists at $T$. Lifted to $\EG^\TT(\mathbf{S}_\Delta)/\Rel(\mathbf{S}_\Delta)$, these local twists become braid twists $\BT(\mathbf{S}_\Delta)$ mapping $\TT$ to $\TT'$. Hence, we get a map $$\Phi: \pi_1\EG(\mathbf{S})/\Rel(\mathbf{S}) \rightarrow \BT(\mathbf{S}_\Delta).$$
	This map is surjective since every $b\in\BT(\mathbf{S}_\Delta)$ sends $\TT$ to some $\TT'\in\mathcal{F}$, which is the endpoint of the lift of an appropriate loop in $\EG(\mathbf{S})/\Rel(\mathbf{S})$.
	
	Finally, $p \in\ker\Phi$ iff its lift $p'$ starting at $\TT$ ends again at $\TT$, i.e. iff $p'$ is represented by a loop in $\EG^\TT(\mathbf{S}_\Delta)/\Rel(\mathbf{S}_\Delta)$.
	This is equivalent to $p \in\operatorname{Im}(F_*)$.
	Thus $\ker\Phi=\operatorname{Im}(F_*)$, and the claim follows from the first isomorphism theorem.
\end{proof}

We now compare:
\begin{itemize}
	\item the \emph{triangulation braid group} $\TB(T)$, the subgroup of $\pi_1(\EG(\mathbf{S}),T)$ generated by local twists, and
	\item the \emph{braid twist group} $\BT(\TT)$, generated by braid twists along closed
	arcs dual to a decorated triangulation~$\TT$,
\end{itemize}
using an explicit finite presentation of $\BT(\TT)$ from~\cite{KQ2} and verifying that
the same relations hold in $\TB(T)$ (see Proposition~\ref{prop:triangulationbraid}).

We now recall the presentation of the abstract braid group associated with a fixed triangulation~$\TT$ on $\mathbf{S}_\Delta$.

\begin{definition}\label{def:braidgroup}
	Let $\TT$ be a triangulation of $\mathbf{S}_\Delta$.
	The \emph{braid group} $\Br(\TT)$ is defined as follows:
	its generators are the open arcs in $\TT$, and the defining relations are
	\begin{itemize}
		\item[$1^\circ$.] $\Co(a,b)$, i.e.\ $ab=ba$, if $a$ and $b$ do not share a common triangle;
		\item[$2^\circ$.] $\Br(a,b)$, i.e.\ $aba=bab$, if $a$ and $b$ share exactly one common triangle;
		\item[$3^\circ$.] $\Co(a^b,c)$, i.e.\ $\Co(b^{-1}ab,c)$, in the configuration of the first diagram
		in Figure~\ref{braid-triang};
		\item[$4^\circ$.] $\Br(a^b,c)$ in the configuration of the second diagram in
		Figure~\ref{braid-triang};
		\item[$5^\circ$.] $\Co(c^{ae},b)$ in the configuration of the third diagram in
		Figure~\ref{braid-triang};
		\item[$6^\circ$.] $\Br(c^{ae},b)$ and $\Br(c^{ea},b)$ in the configuration of the
		fourth diagram in Figure~\ref{braid-triang};
		\item[$7^\circ$.] $\Br(c^{ae},b)$, $\Br(c^{ea},b)$ and $\Co(e,f^{abc})$ in the
		configuration of the fifth diagram in Figure~\ref{braid-triang}.
	\end{itemize}
\end{definition}

\begin{figure}
	\begin{center}
		\begin{tikzpicture}[scale=0.7, every node/.style={scale=0.7}]
			\node (s)   [regular polygon, regular polygon sides=5, draw,
			minimum size=33mm] {};
			\draw (s.corner 1) -- (s.corner 3);
			\draw (s.corner 1) -- (s.corner 4);
			\node at (-0.75,0) {$a$};
			\node at (0.75,0) {$b$};
			\node at (0,-1.5) {$c$};
		\end{tikzpicture}
		\begin{tikzpicture}[scale=0.7, every node/.style={scale=0.7}]
			\node (s)   [regular polygon, regular polygon sides=5, draw,
			minimum size=33mm] {};
			\draw (s.corner 1) -- (s.corner 3);
			\draw (s.corner 1) -- (s.corner 4);
			\node at (-0.75,0) {$a$};
			\node at (0.75,0) {$b$};
			\node at (0,-1.5) {$c$};
			\node at (1,1.2) {$c$};
		\end{tikzpicture}
		\begin{tikzpicture}[scale=0.7, every node/.style={scale=0.7}]
			\node (s)   [regular polygon, regular polygon sides=5, draw,
			minimum size=33mm] {};
			\draw (s.corner 1) -- (s.corner 3);
			\draw (s.corner 1) -- (s.corner 4);
			\node at (-0.75,0) {$a$};
			\node at (0.75,0) {$b$};
			\node at (0,-1.5) {$c$};
			\node at (1,1.2) {$c$};
			\node at (1.5,-0.3) {$e$};
		\end{tikzpicture}
		\begin{tikzpicture}[scale=0.7, every node/.style={scale=0.7}]
			\node (s)   [regular polygon, regular polygon sides=5, draw,
			minimum size=33mm] {};
			\draw (s.corner 1) -- (s.corner 3);
			\draw (s.corner 1) -- (s.corner 4);
			\node at (-0.75,0) {$a$};
			\node at (0.75,0) {$b$};
			\node at (0,-1.5) {$c$};
			\node at (1,1.2) {$c$};
			\node at (1.5,-0.3) {$e$};
			\node at (-1.5,-0.3) {$e$};
		\end{tikzpicture}
		\begin{tikzpicture}[scale=0.7, every node/.style={scale=0.7}]
			\node (s)   [regular polygon, regular polygon sides=5, draw,
			minimum size=33mm] {};
			\draw (s.corner 1) -- (s.corner 3);
			\draw (s.corner 1) -- (s.corner 4);
			\node at (-0.75,0) {$a$};
			\node at (0.75,0) {$b$};
			\node at (0,-1.5) {$c$};
			\node at (1,1.2) {$c$};
			\node at (-1,1.2) {$f$};
			\node at (1.5,-0.3) {$e$};
			\node at (-1.5,-0.3) {$e$};
		\end{tikzpicture}
	\end{center}
	\caption{Five cases triangulations} \label{braid-triang}
\end{figure}
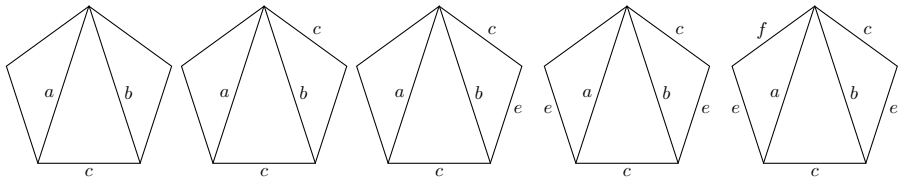

The following theorem can be found in \cite{KQ2} as Theorem~3.15. 
\begin{theorem}\label{thm:Braid}
	The braid twist group $\BT(\TT)$ is isomorphic to the
	braid group $\Br(\TT)$, where the braid twist $B_\eta$ corresponds to the dual open arc
	of $\eta$ in $\TT$, for any $\eta \in \TT^*$. Thus we obtain an explicit finite presentation of $\BT(\TT)$.
\end{theorem}

\begin{prop}\label{prop:triangulationbraid}
	If there is a subset of arcs in the triangulation $T^\circ \in \EG(\mathbf{S})$ like in one of the cases of Definition~\ref{def:braidgroup}, then the corresponding relation there holds in the triangulation braid group $\TB(T^\circ)$.
\end{prop}

\begin{proof}
	We need to show that the relations in $\BT(\TT)$ are also valid in $\TB(T^\circ)$. The map is given by $B_a \mapsto (t^\circ_a)^{-1}$.
	
	The first two relations follow from Lemma~\ref{lem:TB_braidrelation}.
	
	$3^\circ.$ For the third relation, we must show that $t_b^\circ (t_a^\circ)^{-1} (t_b^\circ)^{-1}$ and $(t_c^\circ)^{-1}$ commute. In the proof, we use the shorthand notation $a_0 := t_a^\circ$, $a := t_a^\bullet$, and similarly for the other elements. We perform a flip at edge $b_0$ in the first triangulation of Figure~\ref{lem:TB_braidrelation}:
	
	\begin{center}
		\begin{tikzpicture}[scale=0.75, every node/.style={scale=0.75}]
			\node (s) [regular polygon, regular polygon sides=5, draw, minimum size=33mm] {};
			\draw (s.corner 1) -- (s.corner 3);
			\draw (s.corner 1) -- (s.corner 4);
			\node at (-0.75,0) {$a_0$};
			\node at (0.75,0) {$b_0$};
			\node at (0,-1.5) {$c_0$}; 
			\draw[->] (1.75,0) -- (3,0);
		\end{tikzpicture}
		\begin{tikzpicture}[scale=0.75, every node/.style={scale=0.75}]
			\node (s) [regular polygon, regular polygon sides=5, draw, minimum size=33mm] {};
			\draw (s.corner 1) -- (s.corner 3);
			\draw (s.corner 3) -- (s.corner 5);
			\node at (-0.75,0) {$a$};
			\node at (0.5,0) {$b$};
			\node at (0,-1.5) {$c$};
		\end{tikzpicture}
	\end{center}
	By Proposition~\ref{prop:conjugation_formula}, the map on the corresponding triangulation braid groups is given by
	\begin{align*}
		\ad_z: \TB(T^\circ) &\to \TB(T^\bullet), \\
		a_0 &\mapsto b^{-1}ab, \\
		b_0, c_0 &\mapsto b, c.
	\end{align*}
	Additionally,
	\begin{align*}
		\ad_z(b_0 a_0^{-1} b_0^{-1}) &= b b^{-1} a^{-1} b b^{-1} = a^{-1}.
	\end{align*}
	Using that $ac = ca$ from Lemma~\ref{lem:TB_braidrelation}, we obtain
	\begin{align*}
		\ad_z((b_0 a_0^{-1} b_0^{-1}) (c_0^{-1})) &= a^{-1} c^{-1} = c^{-1} a^{-1} = \ad_z((c_0^{-1}) (b_0 a_0^{-1} b_0^{-1})).
	\end{align*}
	Thus, we conclude that
	\begin{align*}
		(b_0 a_0^{-1} b_0^{-1}) (c_0^{-1}) &= (c_0^{-1}) (b_0 a_0^{-1} b_0^{-1}).
	\end{align*}
	
	$4^\circ.$ For the fourth relation, we must verify the braid relation for $b_0 a_0^{-1} b_0^{-1}$ and $c_0^{-1}$. We perform a flip at edge $b_0$ in the second triangulation of Figure~\ref{braid-triang}:
	\begin{center}
		\begin{tikzpicture}[scale=0.75, every node/.style={scale=0.75}]
			\node (s)   [regular polygon, regular polygon sides=5, draw,
			minimum size=33mm] {};
			\draw (s.corner 1) -- (s.corner 3);
			\draw (s.corner 1) -- (s.corner 4);
			\node at (-0.75,0) {$a_0$};
			\node at (0.75,0) {$b_0$};
			\node at (0,-1.5) {$c_0$};
			\node at (1,1.2) {$c_0$};
			\draw[->] (1.75,0) -- (3,0);
		\end{tikzpicture}
		\begin{tikzpicture}[scale=0.75, every node/.style={scale=0.75}]
			\node (s)   [regular polygon, regular polygon sides=5, draw,
			minimum size=33mm] {};
			\draw (s.corner 1) -- (s.corner 3);
			\draw (s.corner 3) -- (s.corner 5);
			\node at (-0.75,0) {$a$};
			\node at (0.5,0) {$b$};
			\node at (0,-1.5) {$c$};
			\node at (1,1.2) {$c$};
		\end{tikzpicture}
	\end{center}
	We apply Proposition~\ref{prop:conjugation_formula}:
	\begin{align*}
		\ad_z: \TB(T^\circ) &\to \TB(T^\bullet), \\
		a_0 &\mapsto b^{-1} a b, \\
		b_0, c_0 &\mapsto b, c.
	\end{align*}
	Using $aca = cac$ from Lemma~\ref{lem:TB_braidrelation}, we find
	\begin{align*}
		\ad_z((b_0 a_0^{-1} b_0^{-1}) (c_0^{-1}) (b_0 a_0^{-1} b_0^{-1})) &= a^{-1} c^{-1} a^{-1} = c^{-1} a^{-1} c^{-1}, \\
		&= \ad_z((c_0^{-1}) (b_0 a_0^{-1} b_0^{-1}) (c_0^{-1})).
	\end{align*}
	Thus,
	\begin{align*}
		(b_0 a_0^{-1} b_0^{-1}) (c_0^{-1}) (b_0 a_0^{-1} b_0^{-1}) &= (c_0^{-1}) (b_0 a_0^{-1} b_0^{-1}) (c_0^{-1}).
	\end{align*}
	
	$5^\circ.$ For the fifth relation, we need to show that $e_0 a_0 c_0^{-1} a_0^{-1} e_0^{-1}$ and $b_0^{-1}$ commute.
	We first perform a flip at edge $a_0$ and then at edge $e'$ in the third triangulation of Figure~\ref{braid-triang}:
	\begin{center}
		\begin{tikzpicture}[scale=0.75, every node/.style={scale=0.75}]
			\node (s)   [regular polygon, regular polygon sides=5, draw,
			minimum size=33mm] {};
			\draw (s.corner 1) -- (s.corner 3);
			\draw (s.corner 1) -- (s.corner 4);
			\node at (-0.75,0) {$a_0$};
			\node at (0.75,0) {$b_0$};
			\node at (0,-1.5) {$c_0$};
			\node at (1,1.2) {$c_0$};
			\node at (1.6,-0.5) {$e_0$};
			\draw[->] (1.75,0) -- (3,0);
		\end{tikzpicture}
		\begin{tikzpicture}[scale=0.75, every node/.style={scale=0.75}]
			\node (s)   [regular polygon, regular polygon sides=5, draw,
			minimum size=33mm] {};
			\draw (s.corner 2) -- (s.corner 4);
			\draw (s.corner 1) -- (s.corner 4);
			\node at (-0.2,-0.2) {$a'$};
			\node at (0.75,0) {$b'$};
			\node at (0,-1.5) {$c'$};
			\node at (1,1.2) {$c'$};
			\node at (1.5,-0.4) {$e'$};
			\draw[->] (1.75,0) -- (3,0);
		\end{tikzpicture}
		\begin{tikzpicture}[scale=0.75, every node/.style={scale=0.75}]
			\node (s)   [regular polygon, regular polygon sides=6, draw,
			minimum size=33mm] {};
			\draw (s.corner 2) -- (s.corner 4);
			\draw (s.corner 2) -- (s.corner 5);
			\draw (s.corner 2) -- (s.corner 6);
			\node at (-1.4,-0.8) {$a$};
			\node at (1.4,0.8) {$a$};
			\node at (0.2,0) {$e$};
			\node at (-1,0) {$b$};
			\node at (0.2,1) {$c$};
		\end{tikzpicture}
	\end{center}
	Applying Proposition~\ref{prop:conjugation_formula}, we get
	\begin{align*}
		\ad_z: \TB(T^\circ) &\to \TB(T^\bullet), \\
		a_0, b_0, e_0 &\mapsto a, b, e, \\
		c_0 &\mapsto (ae)^{-1} c (ae).
	\end{align*}
	Using $ea = ae$ from Lemma~\ref{lem:TB_braidrelation}, we compute
	\begin{align*}
		\ad_z(e_0 a_0 c_0^{-1} a_0^{-1} e_0^{-1}) &= ea (e^{-1} a^{-1} c^{-1} ae) a^{-1} e^{-1} = c^{-1}.
	\end{align*}
	Since $cb = bc$, it follows that
	\begin{align*}
		\ad_z((e_0a_0c_0^{-1}a_0^{-1}e_0^{-1}) \cdot (b_0^{-1})) = c^{-1} b^{-1} = b^{-1} c^{-1} = \ad_z((b_0^{-1}) \cdot (e_0a_0c_0^{-1}a_0^{-1}e_0^{-1})), 
	\end{align*}
	confirming that $(e_0a_0c_0^{-1}a_0^{-1}e_0^{-1}) \cdot (b_0^{-1}) = (b_0^{-1}) \cdot (e_0a_0c_0^{-1}a_0^{-1}e_0^{-1})$.
	
	$6^\circ.$ For the sixth relation, we must show two cases. The first case requires verifying that $(a_0 e_0) c_0^{-1} (a_0 e_0)^{-1}$ and $b_0^{-1}$ satisfy the braid relation. 
	We first perform a flip at edge $a_0$ and then at edge $e'$ in the fourth triangulation of Figure~\ref{braid-triang}:
	\begin{center}
		\begin{tikzpicture}[scale=0.75, every node/.style={scale=0.75}]
			\node (s)   [regular polygon, regular polygon sides=5, draw,
			minimum size=33mm] {};
			\draw (s.corner 2) -- (s.corner 4);
			\draw (s.corner 1) -- (s.corner 4);
			\node at (-0.3,-0.2) {$e_0$};
			\node at (0.75,0.2) {$a_0$};
			\node at (0,-1.6) {$b_0$}; 
			\node at (-1.5,-0.6) {$c_0$};
			\node at (1.5,-0.6) {$c_0$};
			\node at (1,1.2) {$b_0$}; 
			\draw[->] (1.75,0) -- (3,0);
		\end{tikzpicture}
		\begin{tikzpicture}[scale=0.75, every node/.style={scale=0.75}]
			\node (s)   [regular polygon, regular polygon sides=5, draw,
			minimum size=33mm] {};
			\draw (s.corner 2) -- (s.corner 5);
			\draw (s.corner 2) -- (s.corner 4);
			\node at (-0.2,-0.2) {$e'$};
			\node at (0,0.7) {$a'$};
			\node at (0,-1.6) {$b'$}; 
			\node at (-1.5,-0.6) {$c'$};
			\node at (1.5,-0.6) {$c'$};
			\node at (1,1.2) {$b'$}; 
			\draw[->] (1.75,0) -- (3,0);
		\end{tikzpicture}
		\begin{tikzpicture}[scale=0.75, every node/.style={scale=0.75}]
			\node (s)   [regular polygon, regular polygon sides=5, draw,
			minimum size=33mm] {};
			\draw (s.corner 2) -- (s.corner 5);
			\draw (s.corner 3) -- (s.corner 5);
			\node at (0.4,-0.6) {$e$};
			\node at (0,0.7) {$a$};
			\node at (0,-1.6) {$b$}; 
			\node at (-1.5,-0.6) {$c$};
			\node at (1.5,-0.6) {$c$};
			\node at (1,1.2) {$b$}; 
		\end{tikzpicture}
	\end{center}
	Using Proposition~\ref{prop:conjugation_formula}, we obtain:
	\begin{align*}
		\ad_z: \TB(T^\circ) &\to \TB(T^\bullet), \\
		a_0, b_0, e_0 &\mapsto a, b, e, \\
		c_0 &\mapsto (ae)^{-1} c (ae).
	\end{align*}
	Additionally, we find $\ad_z((a_0e_0)c_0^{-1}(a_0e_0)^{-1}) = (ae)((ae)^{-1}c^{-1}(ae))(ae)^{-1} = c^{-1}$.
	Using $cbc = bcb$ from Lemma~\ref{lem:TB_braidrelation}, we obtain:
	\begin{align*}
		\ad_z((a_0e_0)c_0^{-1}(a_0e_0)^{-1} \cdot (b_0^{-1}) \cdot (a_0e_0)c_0^{-1}(a_0e_0)^{-1}) = c^{-1} b^{-1} c^{-1} \\
		= b^{-1} c^{-1} b^{-1} = \ad_z((b_0^{-1}) \cdot (a_0e_0)c_0^{-1}(a_0e_0)^{-1} \cdot (b_0^{-1})),
	\end{align*}
	confirming that $$(a_0e_0)c_0^{-1}(a_0e_0)^{-1} \cdot (b_0^{-1}) \cdot (a_0e_0)c_0^{-1}(a_0e_0)^{-1} = (b_0^{-1}) \cdot (a_0e_0)c_0^{-1}(a_0e_0)^{-1} \cdot (b_0^{-1}).$$
	
	For the second case, we need to show that $(e_0 a_0) c_0^{-1} (e_0 a_0)^{-1}$ and $b_0^{-1}$ satisfy the braid relation. 
	We first perform a flip at edge $e_0$, then at edge $a'$ and then at the edge $e''$ in the fourth triangulation of Figure~\ref{braid-triang}:
	\begin{center}
		\begin{tikzpicture}[scale=0.6, every node/.style={scale=0.6}]
			\node (s)   [regular polygon, regular polygon sides=5, draw,
			minimum size=33mm] {};
			\draw (s.corner 2) -- (s.corner 4);
			\draw (s.corner 1) -- (s.corner 4);
			\node at (-0.3,-0.2) {$e_0$};
			\node at (0.75,0.2) {$a_0$};
			\node at (0,-1.6) {$b_0$}; 
			\node at (-1.5,-0.6) {$c_0$};
			\node at (1.5,-0.6) {$c_0$};
			\node at (1,1.2) {$b_0$}; 
			\draw[->] (1.75,0) -- (3,0);
		\end{tikzpicture}
		\begin{tikzpicture}[scale=0.6, every node/.style={scale=0.6}]
			\node (s)   [regular polygon, regular polygon sides=5, draw,
			minimum size=33mm] {};
			\draw (s.corner 1) -- (s.corner 3);
			\draw (s.corner 1) -- (s.corner 4);
			\node at (-0.75,0.2) {$e'$};
			\node at (0.75,0.2) {$a'$};
			\node at (0,-1.6) {$b'$}; 
			\node at (-1.5,-0.6) {$c'$};
			\node at (1.5,-0.6) {$c'$};
			\node at (1,1.2) {$b'$}; 
			\draw[->] (1.75,0) -- (3,0);
		\end{tikzpicture}
		\begin{tikzpicture}[scale=0.6, every node/.style={scale=0.6}]
			\node (s)   [regular polygon, regular polygon sides=5, draw,
			minimum size=33mm] {};
			\draw (s.corner 1) -- (s.corner 3);
			\draw (s.corner 3) -- (s.corner 5);
			\node at (-0.75,0.2) {$e''$};
			\node at (0.4,-0.6) {$a''$};
			\node at (0,-1.6) {$b''$}; 
			\node at (-1.5,-0.6) {$c''$};
			\node at (1.5,-0.6) {$c''$};
			\node at (1,1.2) {$b''$}; 
			\draw[->] (1.75,0) -- (3,0);
		\end{tikzpicture}
		\begin{tikzpicture}[scale=0.6, every node/.style={scale=0.6}]
			\node (s)   [regular polygon, regular polygon sides=5, draw,
			minimum size=33mm] {};
			\draw (s.corner 2) -- (s.corner 5);
			\draw (s.corner 3) -- (s.corner 5);
			\node at (0.4,-0.6) {$a$};
			\node at (0,0.7) {$e$};
			\node at (0,-1.6) {$b$}; 
			\node at (-1.5,-0.6) {$c$};
			\node at (1.5,-0.6) {$c$};
			\node at (1,1.2) {$b$}; 
		\end{tikzpicture}
	\end{center}
	Using Proposition~\ref{prop:conjugation_formula}, we obtain:
	\begin{align*}
		\ad_z: \TB(T^\circ) &\rightarrow \TB(T^\bullet), \\
		a_0 &\mapsto e\\
		b_0 &\mapsto b\\
		e_0 &\mapsto a\\
		c_0 &\mapsto (ae)^{-1}c(ae)
	\end{align*}
	Additionally, we find $\ad_z((e_0a_0)c_0^{-1}(e_0a_0)^{-1}) = (ae)((ae)^{-1}c^{-1}(ae))(ae)^{-1} = c^{-1}$.
	Using that $cbc = bcb$ from Lemma~\ref{lem:TB_braidrelation}, we find
	\begin{align*}
		\ad_z((e_0a_0)c_0^{-1}(e_0a_0)^{-1} \cdot (b_0^{-1}) \cdot (a_0e_0)c_0^{-1}(a_0e_0)^{-1}) = c^{-1} b^{-1} c^{-1} \\
		= b^{-1} c^{-1} b^{-1} = \ad_z((b_0^{-1}) \cdot (e_0a_0)c_0^{-1}(e_0a_0)^{-1} \cdot (b_0^{-1})).
	\end{align*}
	Thus, $$\big( (e_0a_0)c_0^{-1}(e_0a_0)^{-1} \big) \cdot (b_0^{-1}) \cdot \big( (e_0a_0)c_0^{-1}(e_0a_0)^{-1} \big) = (b_0^{-1}) \cdot \big( (e_0a_0)c_0^{-1}(e_0a_0)^{-1} \big) \cdot (b_0^{-1}).$$
	
	$7^\circ.$ For the seventh relation, we need to show that $e_0^{-1}$ and $$c_0^{-1}b_0^{-1}a_0^{-1}f_0^{-1}a_0b_0c_0$$ commute.
\end{proof}

%% file: sec_quad.tex
\section{Quadratic differentials}\label{sec:quad}
We now connect the combinatorial structures of Section~\ref{sec:eg} to the geometry of
meromorphic quadratic differentials.  A saddle-free differential induces a foliation whose trajectory structure naturally gives rise to triangulations or, in the presence of higher–order zeroes, $\mathbf{w}$–mixed-angulations.  
Since the combinatorics of a triangulation depend on how the differential is placed on the surface, we need a canonical way to compare differentials.
A \textit{framing} provides this by identifying the real blow-up of the differential with one of the model surfaces of Subsection~\ref{subsection:surfaces}.

The simple and decorated cases follow §4.1 in \cite{KQ2}, while the weighted setting needed for higher–order zeroes uses the framework of \cite{BMQS}.  
Throughout this section we describe the moduli spaces of framed differentials and explain how their trajectory structures interact with the exchange graphs of Section~\ref{sec:eg}. These geometric models will form the basis for the fundamental-group computations carried out in Section~\ref{sec:egandmod}.

We begin by defining framed differentials with various additional data, then explain how saddle-free differentials determine arc systems, and finally recall the wall–chamber structure.

\subsection{Framed Quadratic Differentials}
In this subsection we define framed quadratic differentials in the simple,
decorated, and weighted settings.  The first two follow §4.1 in \cite{KQ2}, while the
weighted framework is taken from \cite{BMQS}.  These definitions supply the
moduli spaces that underlie the constructions in the rest of the section.

\begin{definition}
	Let $\mathbf{X}$ be a compact Riemann surface and $\omega_{\mathbf{X}}$ its holomorphic cotangent bundle.
	A \emph{meromorphic quadratic differential} $\phi$ on $\mathbf{X}$ is a meromorphic section of
	$\omega_{\mathbf{X}}^{\otimes 2}$.
	Locally, in a coordinate $z$, we can write $$ \phi = g(z)\,dz^2 $$
	for some meromorphic function~$g(z)$.
\end{definition}

We will only consider meromorphic quadratic differentials $\phi$ such that $\phi$ has at least one finite critical point and every pole of $\phi$ has order at least three.
Denote by $\Zer(\phi)$ the set of zeroes of $\phi$, by $\Pol_j(\phi)$ the set of poles of order~$j$, and set
$$\Pol(\phi):=\bigcup_j \Pol_j(\phi), \ \ \Crit(\phi):=\Zer(\phi)\cup\Pol(\phi).$$

We associate to each differential a marked surface (as in Definition~\ref{def:MS}) by \emph{real blow-up}. 

\begin{definition}
	The \emph{real (oriented) blow-up} of $(\mathbf{X},\phi)$ is the surface $\mathbf{X}_\phi$
	obtained by replacing each pole $P$ of order $\ge3$ by a boundary component $\partial_P$.
	Points of $\partial_P$ correspond to real tangent directions at~$P$, and we mark on $\partial_P$
	the $\ord_P(\phi)-2$ distinguished directions.
\end{definition}

Thus $\mathbf{X}_\phi$ is a marked surface, diffeomorphic to some $\mathbf{S}=(S,\mathbf{M})$ as in Definition~\ref{def:MS}.

\begin{definition}
	An \emph{$\mathbf{S}$-framed quadratic differential} is a triple
	$(\mathbf{X},\phi,\psi)$ consisting of
	\begin{itemize}
		\item[(a)] a Riemann surface $\mathbf{X}$ with
		\item[(b)] a GMN differential~$\phi$ and
		\item[(c)] a diffeomorphism
		$\psi:\mathbf{S}\to\mathbf{X}_\phi$ preserving marked points.
	\end{itemize}
	Two such framed differentials $(\mathbf{X}_1,\phi_1,\psi_1)$ and $(\mathbf{X}_2,\phi_2,\psi_2)$ are
	\emph{equivalent} if there exists a biholomorphism $f:\mathbf{X}_1\to\mathbf{X}_2$
	\begin{itemize}
		\item[(a)] with $f^*\phi_2=\phi_1$
		\item[(b)] and $\psi_2^{-1}\circ f_*\circ\psi_1\in\Diff_0(\mathbf{S})$.
	\end{itemize}
	the identity component of the diffeomorphism group preserving $\mathbf{M}$ setwise.
	The moduli space of such framed differentials is denoted $\FQuad(\mathbf{S})$.
\end{definition}

The mapping class group $\MCG(\mathbf{S})$ (see Definition~\ref{def:MCG_S_S_Delta}) acts on $\FQuad(\mathbf{S})$ by changing the framing, and the quotient gives the unframed moduli space:
$$ \Quad(\mathbf{S}) \cong \FQuad(\mathbf{S}) / \MCG(\mathbf{S}). $$
This quotient is in general not a manifold but rather an orbifold.

We now enrich the construction by recording the zeroes of~$\phi$ as decorations. The decorated marked surface $\mathbf{X}^\phi_\Delta$
(as in Definition~\ref{def:DMS}) is obtained from $\mathbf{X}_\phi$ by adding~$\Delta :=\Zer(\phi)$ as a set of decorating points.

\begin{definition}\label{def:S_DeltaFramedDifferential}
	Given a decorated marked surface $\mathbf{S}_\Delta$.
	An \emph{$\mathbf{S}_\Delta$-framed quadratic differential} is a triple
	$(\mathbf{X},\phi,\psi)$ consisting of
	\begin{itemize}
		\item[(a)] a Riemann surface $\mathbf{X}$ with
		\item[(b)] a GMN differential~$\phi$ and
		\item[(c)] a diffeomorphism
		$\psi:\mathbf{S}_\Delta\to\mathbf{X}_\Delta^\phi$ preserving both marked points and decorations.
	\end{itemize}
	Two such framed differentials $(\mathbf{X}_1,\phi_1,\psi_1)$ and $(\mathbf{X}_2,\phi_2,\psi_2)$ are
	\emph{equivalent} if there exists a biholomorphism $f:\mathbf{X}_1\to\mathbf{X}_2$
	\begin{itemize}
		\item[(a)] with $f^*\phi_2=\phi_1$
		\item[(b)] and $\psi_2^{-1}\circ f_*\circ\psi_1\in\Diff_0(\mathbf{S}_\Delta)$,
	\end{itemize}
	the identity component of the diffeomorphism group preserving $\mathbf{M}$ and~$\Delta$ setwise.
	The moduli space of such framed differentials is denoted $\FQuad(\mathbf{S}_\Delta)$.
\end{definition}

The mapping class group $\MCG(\mathbf{S}_\Delta)$ (see Definition~\ref{def:MCG_S_S_Delta}) acts naturally on it:
$$ \Quad(\mathbf{S}) \cong \FQuad(\mathbf{S}_\Delta) / \MCG(\mathbf{S}_\Delta). $$
Let $\FQuad^{\mathbb{T}}(\mathbf{S}_\Delta)$ be the connected component associated to a fixed decorated triangulation~$\mathbb{T}$.

\begin{lemma}[Covering structure]\label{lemma:Deck}
	The projection
	$$
	\FQuad^{\mathbb{T}}(\mathbf{S}_\Delta)\longrightarrow \Quad(\mathbf{S})
	$$
	is a covering with deck group $\MCG(\mathbf{S}_\Delta)$.
\end{lemma}

\begin{proof}
	The action of $\MCG(\mathbf{S}_\Delta)$ on $\FQuad^{\mathbb{T}}(\mathbf{S}_\Delta)$
	is properly discontinuous because the induced action on Teichmüller space is.
	By \cite[Prop.\,1.40]{Hatcher}, this properly discontinuous free action yields a covering with deck group $\MCG(\mathbf{S}_\Delta)$.
\end{proof}

We now allow zeroes of arbitrary order.
These orders are recorded by a weight function on the decorations, which in turn determines the number of vertices of each polygon in the associated mixed-angulation.
The following definition is taken from \cite{BMQS}.

\begin{definition}\label{def:S_wFramedDifferential}
	Given a decorated marked surface $\sow$.
	An \emph{$\sow$-framed quadratic differential} is a triple
	$(\mathbf{X},\phi,\psi)$ consisting of
	\begin{itemize}
		\item[(a)] a Riemann surface $\mathbf{X}$ with
		\item[(b)] a GMN differential~$\phi$ and
		\item[(c)] a diffeomorphism
		$\psi:\mathbf{S}_{\mathbf{w}}\to\mathbf{X}_{\mathbf{w}}^\phi$ preserving marked points, decorations and weights.
	\end{itemize}
	Two such framed differentials $(\mathbf{X}_1,\phi_1,\psi_1)$ and $(\mathbf{X}_2,\phi_2,\psi_2)$ are
	\emph{equivalent} if there exists a biholomorphism $f:\mathbf{X}_1\to\mathbf{X}_2$
	\begin{itemize}
		\item[(a)] with $f^*\phi_2=\phi_1$
		\item[(b)] and $\psi_2^{-1}\circ f_*\circ\psi_1\in\Diff_0(\mathbf{S}_{\mathbf{w}})$,
	\end{itemize}
	the identity component of the diffeomorphism group preserving $\mathbf{M}$,~$\Delta$ and their weights setwise.
	The moduli space of such framed differentials is denoted $\FQuad(\mathbf{S}_{\mathbf{w}})$, and a chosen connected component by $\FQuad^\circ(\mathbf{S}_{\mathbf{w}})$.
\end{definition}

The mapping class group
\[
\MCG(\mathbf{S}_{\mathbf{w}})
:= \Diff(\mathbf{S}_{\mathbf{w}})/\Diff_0(\mathbf{S}_{\mathbf{w}})
\]
(consisting of isotopy classes of diffeomorphisms fixing the boundary and marked points
and preserving the weighted decorations setwise)
acts on $\FQuad(\mathbf{S}_{\mathbf{w}})$, and we set
\[
\Quad(\sow) \;:=\; \FQuad(\mathbf{S}_{\mathbf{w}})/\MCG(\mathbf{S}_{\mathbf{w}}).
\]

\begin{rmk}
	In general $\Quad(\sow)$ is an orbifold rather than a manifold.
	A basic example is the stratum $\Quad(2,1,1,-8)$, which can be realised
	by choosing $\mathbf{S}$ as the disc with six marked boundary points and
	weights $\mathbf{w} = (2,1,1)$:
	It is the quotient of the labelled stratum by an $S_2$-action permuting the two simple zeroes, and this action has fixpoints.
	We include this example to indicate the natural appearance of orbifold structures in the unlabelled or unframed setting.
\end{rmk}

\subsection{From Differential to Exchange Graph}\label{subsection:difftoeg}
Here we recall how a saddle-free quadratic differential determines a canonical
collection of arcs on its real blow-up: a triangulation in the simple-zero setting and a $\mathbf{w}$–mixed–angulation when higher–order zeroes are present. The description in the simple and decorated cases follows §4.2 in \cite{KQ2}, which also quotes \cite{BS15}, while the weighted and mixed–angulation behaviour stems from \cite{BMQS}. This construction assigns to every saddle-free framed differential a vertex in the appropriate exchange graph of Section~\ref{sec:eg}.\\


Let $\phi = g(z)\,dz^2$ be a quadratic differential on a Riemann surface $\mathbf{X}$, in a coordinate $z$ and for some meromorphic function~$g(z)$.
On $\mathbf{X}^\circ := \mathbf{X}\setminus\Crit(\phi)$ there exists a distinguished local coordinate $\omega$,
unique up to $\omega\mapsto\pm\omega+\mathrm{const}$, such that $\phi=d\omega\otimes d\omega$.
In local coordinates, $\omega=\int \sqrt{g(z)}\,dz$.
This induces the \emph{$\phi$-metric} on $\mathbf{X}^\circ$ by pulling back the Euclidean metric via~$\omega$;
geodesics for this metric have constant argument of~$\sqrt{\phi}$.

A (horizontal) trajectory of~$\phi$ is a maximal horizontal geodesic
$\gamma:(0,1)\to\mathbf{X}^\circ$ with respect to the $\phi$-metric.
If $\lim_{t\to0}\gamma(t)$ (resp.\ $\lim_{t\to1}\gamma(t)$) exists in~$\mathbf{X}$,
it is called the left (resp.\ right) endpoint of~$\gamma$.
The trajectories of $\phi$ determine its \emph{horizontal foliation}.
Their local behaviour is well understood (\cite[§3]{BS15}):
\begin{itemize}
	\item Near a zero of order $k$, trajectories form $k+2$ prongs meeting at equal angles. See Figure~\ref{fig:local_trajectory_zeroes} for the local trajectory structure for zeroes of order $1,2$ and $3$.
	\begin{figure}[ht]
		\centering
		\input{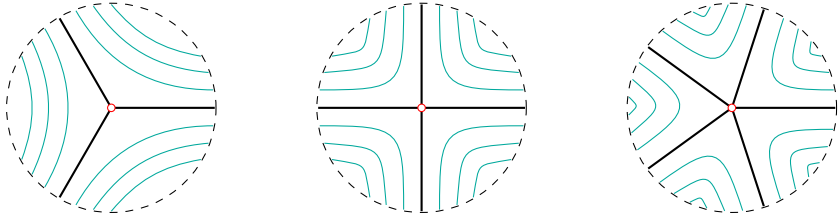}
		\caption{Local trajectories at zeroes of order $1,2,3$}
		\label{fig:local_trajectory_zeroes}
	\end{figure}
	\item Near a pole of order $m\ge3$, there are $m-2$ distinguished tangent directions $\{v_j\}$ such that trajectories approaching the pole become asymptotic to one of these directions. See Figure~\ref{fig:local_trajectory_poles} for the local trajectory structure for poles of order $3,4$ and $5$.
	\begin{figure}[ht]
		\centering
		\input{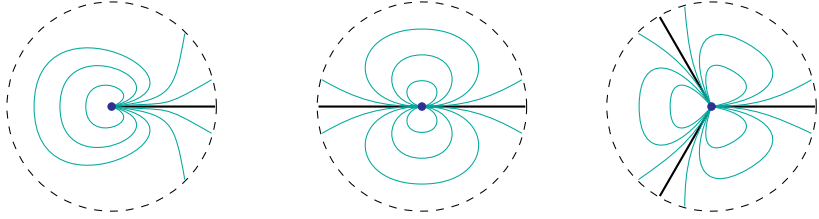}
		\caption{Local trajectories at poles of order $3,4,5$}
		\label{fig:local_trajectory_poles}
	\end{figure}
\end{itemize}

The horizontal foliation of $\phi$ decomposes $\mathbf{X}$ into standard pieces,
which later correspond to arcs and flips on the marked surface.
The trajectories of a differential fall into three types:
\begin{itemize}
	\item \emph{saddle trajectories}, joining two zeroes of $\phi$;
	\item \emph{separating trajectories}, joining a zero to a pole;
	\item \emph{generic trajectories}, joining two poles.
\end{itemize}

Removing all separating trajectories yields a disjoint union of components, each of which is either
\begin{itemize}
	\item a \emph{half-plane}, isomorphic to $\{z\in\mathbb{C}\mid \Im z>0\}$ with differential $dz^2$, swept out by trajectories from a pole to itself; or
	\item a \emph{horizontal strip}, isomorphic to $\{z\in\mathbb{C}\mid a<\Im z<b\}$ with differential $dz^2$, swept out by trajectories connecting (possibly equal) poles.
\end{itemize}
We call this union the horizontal strip decomposition of $X$ with respect to $\phi$. A differential $\phi$ on $X$ is \textit{saddle-free}, if it has no saddle trajectory. Similarly, a framed quadratic differential on $\mathbf{S}$ (or $\mathbf{S}_\Delta$) is saddle-free if the corresponding differential is saddle-free. Note the following:
\begin{itemize}
	\item In each horizontal strip, the trajectories are isotopic to each other.
	\item the boundary of any component consists of separating trajectories.
	\item In each horizontal strip, there is a unique geodesic, the saddle connection, connecting the two zeroes on its boundary.
	\item For a saddle-free differential $\phi$ on $X$, we have $\Pol(\phi) \neq \emptyset$. Then $\phi$ has no closed or recurrent trajectories by [\cite{BS15}, Lemma 3.1]. Thus, in the horizontal strip decomposition of $X$ with respect to $\phi$, there is only half-planes and horizontal strips.
\end{itemize}

By construction, the generic trajectories on $X$ (with respect to $\phi$) are inherited by $\mathbf{S}$, for any $\psi: \mathbf{S} \rightarrow X^\phi$, and all trajectories on $X$ (with respect to $\phi$) are inherited by $\mathbf{S}_\Delta$,
for any $\Psi: \mathbf{S}_\Delta \rightarrow X^\phi$. For instance, the generic trajectories become open arcs on $\mathbf{S}$ (as
well as on $\mathbf{S}_\Delta$) and saddle trajectories becomes closed arcs on $\mathbf{S}_\Delta$.

\begin{definition}
	Let $\psi: \mathbf{S} \rightarrow X^\phi$ be an $\mathbf{S}$-framed quadratic differential with simple zeroes, which is
	saddle-free. Then there is a triangulation $T_\psi$ on $\mathbf{S}$ induced from $\psi$, arcs are (isotopy classes of inherited) generic trajectories. Moreover, each triangle in $T_\psi$ contains exactly one zero, so $T_\psi$ becomes a decorated triangulation $\TT_\psi$ of $\mathbf{S}_\Delta$, with dual
	graph consisting of saddle trajectories.
\end{definition}

Similarly we have the following generalised definition taken from \cite{BMQS}.

\begin{definition}
	Let $\psi: \mathbf{S} \rightarrow X^\phi_\mathbf{w}$ be an $\sow$-framed quadratic differential, which is
	saddle-free. Then there is a $\mathbf{w}$-mixed-angulation $\MM_\psi$ on $\sow$ induced from $\psi$, arcs are (isotopy classes of inherited) generic trajectories. Each $(k+2)$-gon in $\MM_\psi$ contains exactly one zero of order (or weight) $k$ . Its dual graph consists of saddle trajectories.
\end{definition}

\subsection{Wall and Chamber Structure}

The moduli spaces of framed quadratic differentials admit natural wall–chamber
decompositions, where chambers correspond to saddle-free differentials and walls record degenerations of trajectory configurations.  
Crossing a wall changes the associated arc system by a single flip.  
For simple zeroes this picture follows \cite{BS15}, and for higher–order zeroes we use the wall–chamber structure developed in §4.2 in \cite{BMQS}.  In this subsection we describe this decomposition and explain how it realises the exchange graphs of Section~\ref{sec:eg} inside the strata of framed differentials.

We consider the top two parts of the stratification of $\FQuad(\mathbf{S})$ analogous to the stratification of Quad(S) from [\cite{BS15}, §5]:
\begin{align*}
	F_0(\mathbf{S}) &= \{[\mathbf{X},\phi,\psi]\in\FQuad(\mathbf{S})
	\mid \phi \text{ has no saddle trajectories}\},\\
	F_2(\mathbf{S}) &= \{[\mathbf{X},\phi,\psi]\in\FQuad(\mathbf{S})
	\mid \phi \text{ has exactly one saddle trajectory}\}.
\end{align*}
Then $B_0(S) := F_0(\mathbf{S})$ is open and dense, and $F_2(S)$ has codimension $1$. Furthermore,
$B_2(\mathbf{S}) := F_0(\mathbf{S}) \cup F_2(\mathbf{S})$ is also open and dense, and has complement of codimension $2$.
Let $U(T)$ be the subspace in $\FQuad(\mathbf{S})$ consisting of those saddle-free $[\psi]$ whose induced  triangulation is $T$. Then

\[
B_0(\mathbf{S})=\bigcup_{T\in\EG(\mathbf{S})}U(T),
\]
By the argument of [\cite{BS15}, Prop 4.9], we can see that $U(T) \simeq \HH^T$, where $$\mathbb{H}=\{z\in\mathbb{C}\mid \Im(z)>0\}$$ is the (strict) upper half-plane.
The coordinates $(u_\gamma)_{\gamma\in\mathbb{T}}$ give the complex modulus of
the horizontal strip with generic trajectory in the isotopy class $\gamma$ (see [BS, §4.5] for more detail). 
Thus the $U(T)$ are precisely the connected components of $B_0(\mathbf{S})$.

The boundary of $U(T)$ meets $F_2(\mathbf{S})$ in $2n$ connected components, which we denote $\partial_\gamma^\# U(T)$ and $\partial_\gamma^\flat U(T)$, where the coordinate $u_\gamma$ goes to the negative or positive real axis, respectively. Note that $u_\gamma$ cannot go to zero because that would correspond to two zeroes of $\phi$ coming together, leaving the stratum.

Similarly, in $\FQuad(\mathbf{S}_\Delta)$ there are cells $U(\TT)$ and a stratification $\{B_p(\mathbf{S})\}$. We can import a key lemma from \cite{BS15}, where the original statement is actually for $\Quad(\mathbf{S}_\Delta)$.

\begin{lemma}\label{lem:homotopchamberSDelta}
	Any path in $\FQuad(\mathbf{S})$ is homotopic (relative to its endpoints) to a path in $B_2(\mathbf{S})$. Therefore there is a surjective map $$\pi_1B_2(\mathbf{S}) \rightarrow \pi_1\FQuad(\mathbf{S}).$$
	The same holds replacing $\mathbf{S}$ by $\mathbf{S}_\Delta$.
\end{lemma}

Hence, we have the following result, showing that $\EG(\mathbf{S})$ is a skeleton for $\FQuad(\mathbf{S})$.

\begin{lemma}\label{lem:embeddingFQuadS}
	There is a (unique up to homotopy) canonical embedding
	$$
	\wp_{\mathbf{S}}: \EG(\mathbf{S}) \rightarrow \FQuad(\mathbf{S})
	$$
	whose image is dual to $B_2(\mathbf{S})$
	and which induces a surjection
	\[
	\wp_*:\pi_1\EG(\mathbf{S})
	\twoheadrightarrow \pi_1\FQuad(\mathbf{S}).
	\]
	The same holds replacing $\mathbf{S}$ by $\mathbf{S}_\Delta$ and restricting to the connected component containing some $\TT$.
\end{lemma}

Finally, the mapping class group $\MCG(\mathbf{S})$ (see Definition~\ref{def:MCG_S_w}) acts naturally on
$\FQuad(\mathbf{S})$, and therefore also on its skeleton $\EG(\mathbf{S})$.
We denote by
\[
\EG^\star(\mathbf{S}) := \EG(\mathbf{S}) / \MCG(\mathbf{S})
\]
the resulting quotient graph, which records the wall--chamber structure of the unframed moduli space $\Quad(\mathbf{S})$.
The embedding above descends to an embedding
\[
\EG^\star(\mathbf{S}) \rightarrow \Quad(\mathbf{S}),
\]
again inducing a surjection
\[
\pi_1\EG^\star(\mathbf{S})\twoheadrightarrow\pi_1\Quad(\mathbf{S}).
\]
The canonical embedding $\pi_1\EG(\mathbf{S}) \hookrightarrow \pi_1\EG^\star(\mathbf{S})$ sends $\Rel(\mathbf{S})$ to a normal subgroup of
$\pi_1\EG^\star(\mathbf{S})$, which we will denote by $\Rel^\star(\mathbf{S})$.

Now we explicitly allow higher-order zeroes. Again we consider the top two strata
\begin{align*}
	F_0(\mathbf{S}_{\mathbf{w}}) &= \{[\mathbf{X},\phi,\psi]\in\FQuad(\mathbf{S}_{\mathbf{w}})
	\mid \phi \text{ has no saddle trajectories}\},\\
	F_2(\mathbf{S}_{\mathbf{w}}) &= \{[\mathbf{X},\phi,\psi]\in\FQuad(\mathbf{S}_{\mathbf{w}})
	\mid \phi \text{ has exactly one saddle trajectory}\},
\end{align*}
and set
\[
B_0(\mathbf{S}_{\mathbf{w}}):=F_0(\mathbf{S}_{\mathbf{w}}),\qquad
B_2(\mathbf{S}_{\mathbf{w}}):=F_0(\mathbf{S}_{\mathbf{w}})\cup F_2(\mathbf{S}_{\mathbf{w}}).
\]
Then $B_0(\mathbf{S}_{\mathbf{w}})$ is open and dense, $F_2(\mathbf{S}_{\mathbf{w}})$ has codimension~1,
and $B_2(\mathbf{S}_{\mathbf{w}})$ is open and dense with complement of codimension~2.

\begin{cor}\label{cor:pathB2}
	Suppose the negative part of the signature is not $\mathbf{w}^- = (-2)$.
	Then any path in $\Quad(\sow)$ can be homotoped
	rel.\ endpoints into $B_2(\mathbf{S}_{\mathbf{w}})$.
\end{cor}

The components of $B_0(\mathbf{S}_{\mathbf{w}})$ are labeled by $\mathbf{w}$-mixed-angulations.

\begin{prop}\label{prop:EGB2}
	Two components of $B_0(\mathbf{S}_{\mathbf{w}})$ are joined in
	$B_2(\mathbf{S}_{\mathbf{w}})$ by an arc with exactly one non-saddle-free point
	if and only if the corresponding $\mathbf{w}$-mixed-angulations
	are related by a forward flip.
\end{prop}

\begin{proof}
	Suppose that the two components of $B_0$ are connected by such an arc, which
	we may homotope to be a small rotation of a saddle connection near the real axis
	while fixing the geometry of the rest of the surface. The question is thus local, in the neighbourhood of this saddle connection. Using a metrically correct drawing, as in the middle of Figure~\ref{cap:flip} one checks that rotating in anticlockwise (clockwise) direction has the effect of passing from the leftmost to the rightmost picture in terms of horizontal strip decompositions. Picking a generic trajectory from the strips, we observe that this changes the mixed-angulation by a forward flip (backward flip).
	Conversely, if two mixed-angulations differ by a forward flip we take differentials locally as indicated in the metric picture and rotate the saddle connection to produce a path as required.
\end{proof}

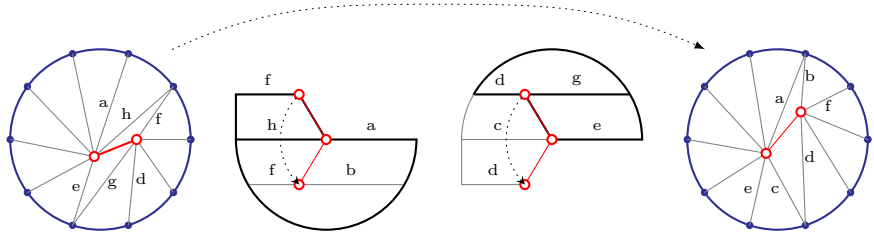
\begin{figure}[ht]
	\begin{tikzpicture}[scale=1.2,cap=round,>=latex]
		\draw[Blue, thick] (0,0) circle [radius=1cm];
		
		\foreach \x in {0,36,...,360} {
			\filldraw[Blue] (\x:1cm) circle(1pt);
		}
		
		\coordinate (Z1) at (0:0.4cm);
		\coordinate (Z2) at (250:0.2cm);

		\draw[red, thick] (Z1)  -- (Z2);
		
		\foreach \x in {36,72,108,144,180,216,252} {
			\draw[gray] (Z2)  -- (\x:1cm);
		}
		
		\foreach \x in {0,36,252,288,324} {
			\draw[gray] (Z1)  -- (\x:1cm);
		}

		\draw[red,thick,fill=white](Z1)circle(00.05);
		\draw[red,thick,fill=white](Z2)circle(00.05);
		
		\node at ($(Z2)!0.5!(82:1cm)$) {\tiny a};
		\node at ($(Z2)!0.5!(242:1cm)$) {\tiny e};
		\node at ($(Z2)!0.5!(50:1cm)$) {\tiny h};
		\node at ($(Z1)!0.5!(28:1cm)$) {\tiny f};
		\node at ($(Z1)!0.5!(262:1cm)$) {\tiny g};
		\node at ($(Z1)!0.5!(300:1cm)$) {\tiny d};
		
		\draw[->,dotted] (0.8,1) to[out= 25, in=180] (3.75,1.5) to[out= 0, in=155] (6.7,1);
		
		
		\coordinate (B) at (2.5,0);
		
		\coordinate (Z1) at ($(B) + (0,0)$);
		\coordinate (Z2) at ($(B) + (-0.3,0.5)$);
		\coordinate (Z3) at ($(B) + (-0.3,-0.5)$);
		
		\draw[thick] (Z1) ++(180:1cm) arc [start angle=180, end angle=360, radius=1cm];
		
		\draw[thick] (Z1)  -- (Z2);
		\draw[red] (Z1)  -- (Z2);
		
		\draw[red] (Z1)  -- (Z3);
		
		\draw[thick] (Z1) -- ($(Z1) + (-1,0)$);
		\node[above] at ($(Z1) + (-0.6,0)$) {\tiny h};
		
		\draw[thick] (Z1) -- ($(Z1) + (1,0)$);
		\node[above] at ($(Z1) + (0.5,0)$) {\tiny a};
		
		\draw[thick] (Z2) -- ($(Z2) + (-0.7,0)$) --($(Z2) + (-0.7,-0.5)$);
		\node[above] at ($(Z2)+(-0.35,0)$) {\tiny f};
		
		\draw[gray] ($(Z3) + (-0.55,0)$) -- (Z3) -- ($(Z3) + (1.15,0)$);
		\node[above] at ($(Z3)+(-0.3,0)$) {\tiny f};
		\node[above] at ($(Z3) + (0.575,0)$) {\tiny b};
		
		\draw[red,thick,fill=white](Z1)circle(00.05);
		\draw[red,thick,fill=white](Z2)circle(00.05);
		\draw[red,thick,fill=white](Z3)circle(00.05);
		
		\draw [stealth-,dotted] (Z3) to [out=135,in=225] (Z2);
		
		
		\coordinate (C) at (5,0);
		
		\coordinate (Z1) at ($(C) + (0,0)$);
		\coordinate (Z2) at ($(C) + (-0.3,0.5)$);
		\coordinate (Z3) at ($(C) + (-0.3,-0.5)$);
		
		\draw[thick] (C) ++(0:1cm) arc [start angle=0, end angle=150, radius=1cm];
		
		\draw[gray] (C) ++(150:1cm) arc [start angle=150, end angle=180, radius=1cm];

		\draw[thick] (Z1)  -- (Z2);
		\draw[red] (Z1)  -- (Z2);
		\draw[red] (Z1)  -- (Z3);
		
		\draw[gray] (Z1) -- ($(Z1) + (-1,0)$);
		\node[above] at ($(Z1)+(-0.6,0)$) {\tiny c};
		
		\draw[thick] (Z1) -- ($(Z1) + (1,0)$);
		\node[above] at ($(Z1)+(0.5,0)$) {\tiny e};
		
		\draw[thick] (Z2) -- ($(Z2) + (-0.55,0)$);
		\node[above] at ($(Z2)+(-0.275,0)$) {\tiny d};
		
		\draw[thick] (Z2) -- ($(Z2) + (1.15,0)$);
		\node[above] at ($(Z2)+(0.575,0)$) {\tiny g};
		
		\draw[gray] ($(Z3) + (-0.7,0.5)$) -- ($(Z3) + (-0.7,0)$) -- (Z3);
		\node[above] at ($(Z3) + (-0.35,0)$) {\tiny d};
		
		\draw[red,thick,fill=white](Z1)circle(00.05);
		\draw[red,thick,fill=white](Z2)circle(00.05);
		\draw[red,thick,fill=white](Z3)circle(00.05);
		
		\draw [stealth-, dotted] (Z3) to [out=135,in=225] (Z2);
		
		
		\coordinate (D) at (7.5,0);
		
		\draw[Blue, thick] (D) circle [radius=1cm];
		
		\foreach \x in {0,36,...,360} {
			\filldraw[Blue] ($(D)+(\x:1cm)$) circle(1pt);
		}
		
		\coordinate (Z1) at ($(D)+(50:0.4cm)$);
		\coordinate (Z2) at ($(D)+(230:0.2cm)$);
		
		\foreach \x in {0,36,72,288,324} {
			\draw[gray] (Z1)  -- ($(D)+(\x:1cm)$);
		}
		
		\foreach \x in {72,108,144,180,216,252,288} {
			\draw[gray] (Z2)  -- ($(D)+(\x:1cm)$);
		}
		
		\node at ($(Z2)+(0.15,0.6)$) {\tiny a};
		\node at ($(Z1)+(0.1,0.4)$) {\tiny b};
		\node at ($(Z2)+(0.1,-0.4)$) {\tiny c};
		\node at ($(Z1)+(0.1,-0.5)$) {\tiny d};
		\node at ($(Z2)+(-0.2,-0.4)$) {\tiny e};
		\node at ($(Z1)+(0.3,0.07)$) {\tiny f};

		\draw[red] (Z1)  -- (Z2);
		
		\draw[red,thick,fill=white](Z1)circle(00.05);
		\draw[red,thick,fill=white](Z2)circle(00.05);
		
	\end{tikzpicture}
	\caption{Horizontal foliation before and after rotating} \label{cap:flip}
\end{figure}

Thus the oriented exchange graph $\EG(\mathbf{S}_{\mathbf{w}})$ of $\mathbf{w}$-mixed-angulations
again encodes the adjacency of chambers in $B_0(\mathbf{S}_{\mathbf{w}})$.

\begin{lemma} \label{lem:embeddingFQuadS_w}
	There exists a canonical embedding (unique up to homotopy)
	\[
	\wp_{\mathbf{S}_{\mathbf{w}}}:
	\EG^\circ(\mathbf{S}_{\mathbf{w}})
	\longrightarrow
	\FQuad^\circ(\mathbf{S}_{\mathbf{w}})
	\]
	whose image is dual to $B_2(\mathbf{S}_{\mathbf{w}})$ and which induces a surjection
	\[
	\wp_*:\pi_1\EG^\circ(\mathbf{S}_{\mathbf{w}})
	\twoheadrightarrow \pi_1\FQuad^\circ(\mathbf{S}_{\mathbf{w}}).
	\]
\end{lemma}

\begin{rmk}\label{rem:EG_wstar}
	As in the simple-zero case, $\MCG(\mathbf{S}_{\mathbf{w}})$ (see Definition~\ref{def:MCG_S_w}) acts on
	$\FQuad^\circ(\mathbf{S}_{\mathbf{w}})$ and on $\EG(\mathbf{S}_{\mathbf{w}})$.
	We write
	\[
	\EG^\star(\sow)
	:= \EG(\sow)/\MCG(\mathbf{S}_{\mathbf{w}})
	\]
	for the resulting quotient graph.
	Since stabilisers may be non-trivial, $\EG^\star(\mathbf{S}_{\mathbf{w}})$ is naturally a $1$-dimensional orbifold (a graph of groups in the sense of Bass–Serre theory), reflecting the orbifold structure of $\Quad(\sow)$.
	The canonical embedding $\pi_1\EG(\sow) \hookrightarrow \pi_1\EG^\star(\sow)$ sends $\Rel(\sow)$ to a normal subgroup in $\pi_1\EG^\star(\sow)$, which we will denote by $\Rel^\star(\sow)$.
\end{rmk}

\begin{lemma} \label{lem:embeddingQuadS_w}
	There exists a canonical embedding (unique up to homotopy)
	\[
	\wp^\star_{\mathbf{S}_{\mathbf{w}}}:
	\EG^\star(\sow)
	\longrightarrow
	\Quad(\sow)
	\]
	which induces a surjection
	\[
	\wp_*:\pi_1\EG^\star(\sow)
	\twoheadrightarrow \pi_1\Quad(\sow).
	\]
	of their orbifold fundamental groups.
\end{lemma}

\section{Exchange Graphs and Moduli Spaces}\label{sec:egandmod}

In Section~\ref{sec:quad} we associated to each saddle-free framed quadratic differential a triangulation or, in the higher-order case, a $\mathbf{w}$-mixed-angulation.
This yields a canonical map from the exchange graphs introduced in
Section~\ref{sec:eg} to the corresponding strata of (framed) quadratic differentials.  
The purpose of this section is to compare the (orbifold) fundamental groups on both sides.

Our first objective is to show that the relations in the appropriate exchange graphs, established in  Section~\ref{sec:eg}, vanish under the induced surjection $\rho^*$ on the fundamental groups.
This is carried out first in the simple-zero (decorated) case in Subsection~\ref{subsec:4.1}, following the strategy of~\cite{KQ2} but formulated here directly in the language of quadratic differentials rather than stability conditions.
The corresponding result in the higher--order--zero (weighted) case is proved in Subsection~\ref{subsec:4.3} using the framework of~\cite{BMQS}.  
In this weighted setting the fact that all these relations map to zero is new.

A key difference between the two settings is that, for simple zeroes, these relations \emph{generate the entire kernel}.  
Consequently, the induced map on fundamental groups becomes an isomorphism after quotienting by the square, pentagon, and hexagon relations.  
Although this is implicit in~\cite{KQ2}, the statement does not appear explicitly there; we give a direct proof in Subsection~\ref{subsec:4.2}.

In the unframed situation the strata are often \emph{orbifolds} rather than manifolds, due to symmetries permuting zeroes of identical order.  
Accordingly, the exchange graph also becomes an orbifold object, and its fundamental group must be interpreted as an \emph{orbifold fundamental group}.  
To handle this systematically we include Subsection~\ref{subsection:bass-serre}, which reviews the necessary background in the simplified setting relevant for us: Graphs with finite stabilizer groups at vertices and trivial stabilizers on edges.  
There we describe the orbifold fundamental group via the Bass-Serre theory of graphs of groups and record an explicit presentation (Proposition~\ref{prop:orbpi1-graph}) that will be used in our computations in Section~\ref{sec:examples}.

\subsection{Simple Zeroes}\label{subsec:4.1}

We begin with the decorated setting, where the underlying surface is
$\mathbf{S}_\Delta$ and the relevant loops are those of
Definition~\ref{def:DMS_Rel}.  
Our goal is to show that these square, pentagon, and hexagon relations vanish
under the map 
\[
\pi_1\EG^{\mathbb{T}}(\mathbf{S}_\Delta)
\longrightarrow
\pi_1\FQuad^{\mathbb{T}}(\mathbf{S}_\Delta)
\]
from Lemma~\ref{lem:embeddingFQuadS}.
In this case the trajectory structure involved in each local configuration is fully analysed in §4.2 and §4.5 (in particular Proposition 4.14) \cite{KQ2}, and the vanishing of all these loops follows by the same ideas, rewritten here in the language of quadratic differentials rather than stability conditions.  

Unlike in the weighted setting, these relations in fact generate the entire kernel, which we will show in subsection~\ref{subsec:4.2}.

\begin{prop}\label{prop:RelTrivial}
	Let $\mathbf{S}_\Delta$ be a decorated marked surface and
	$\EG^{\mathbb{T}}(\mathbf{S}_\Delta)$ its oriented exchange graph.
	Under the embedding
	\[
	\wp_{\mathbf{S}_\Delta}:
	\EG^{\mathbb{T}}(\mathbf{S}_\Delta)
	\hookrightarrow
	\FQuad^{\mathbb{T}}(\mathbf{S}_\Delta),
	\]
	the loops corresponding to the square, pentagon, and hexagon relations (see Definition~\ref{def:DMS_Rel}) in
	$\EG^{\mathbb{T}}(\mathbf{S}_\Delta)$ are null-homotopic.
	Equivalently, the subgroup
	$\Rel(\mathbf{S}_\Delta)$
	lies in the kernel of the induced map
	\[
	\pi_1\EG^{\mathbb{T}}(\mathbf{S}_\Delta)
	\longrightarrow
	\pi_1\FQuad^{\mathbb{T}}(\mathbf{S}_\Delta).
	\]
\end{prop}

\begin{proof}
	\noindent\textbf{Strategy of the proof.}
	We construct three contractible product strips in
	$\FQuad^{\mathbb{T}}(\mathbf{S}_\Delta)$ whose boundaries glue to a topological disc.
	Via the identification of chambers with triangulations of $\mathbf{S}_\Delta$, the remaining boundary of this disc is precisely the hexagon loop in $\EG^{\mathbb{T}}(\mathbf{S}_\Delta)$, yielding an explicit null-homotopy.
 	The square and pentagon cases are analogous and simpler.
	\medskip
	
	We now carry out this construction for a fixed hexagon.
	Fix a hexagon subgraph in $\EG^\mathbb{T}(\mathbf{S}_\Delta)$ determined by two arcs that intersect twice in a triangle. Choose $q_0\in \FQuad^\mathbb{T}(\mathbf{S}_\Delta)$ whose triangulation $\TT_0$ is the initial vertex of this hexagon, and let
	$\alpha_1,\alpha_2$ be the corresponding saddle connections.
	
	\smallskip
	\noindent\textbf{Period normalization.}
	Fix small $\delta>0$, a positive $m < 1$, and a large $M\gg1$,
	and arrange
	\begin{align*}
	Z_{q_0}(\alpha_1)=e^{(1-\delta) \pi i},\qquad
	Z_{q_0}(\alpha_2)=m\,e^{(1-3\delta) \pi i},\qquad\\
	Z_{q_0}(\beta)=M\,e^{1/2 \pi i}\quad\text{for all other  standard saddles }\beta.
	\end{align*}
	By standard saddle we mean the saddle connections inside inside the horizontal strips given by $q_0$.
	This choice ensures that, along suitably chosen deformations, the only intended wall contacts on the \emph{boundaries} below are those producing the two flips adjacent to $\TT_0$.
	
	\smallskip
	\noindent\textbf{Three product strips.}
	We construct three continuous maps
	\[
	\Phi_j:\ I_j \times \ell_j\ \longrightarrow\ \mathrm{FQuad}^\TT(\mathbf{S}_\Delta)\qquad (j=1,2,3),
	\]
	where each $\ell_j:[0,1]\to \mathrm{FQuad}^\TT(\mathbf{S}_\Delta)$ is a path of differentials (a
	\emph{period deformation}) and each $I_j$ is a \emph{rotation window}:
	\[
	I_1=[0,4\delta],\qquad I_2=[-2\delta,0],\qquad I_3=[-2\delta,2\delta].
	\]
	We set
	\[
	\Phi_j(\theta,t)\ :=\ e^{\pi i\theta}\cdot \ell_j(t).
	\]
	
	\begin{itemize}
		\item \emph{Strip 1: $[0,4\delta] \times \ell_1$ (Figure~\ref{hexagoninquad1})}.
		In all pictures for this proof the gray area denotes a half-plane.
		Choose $\ell_1$ with $\ell_1(0)=q_0$ so that along $\ell_1$ the saddle connection $\alpha_1$ rotates by $-2\delta$ and the saddle connection $\alpha_2$ by $2\delta$, while all other standard saddles of $q_0$ stay fixed. We end up with a differential $q'_0$, with
		\begin{align*}
			Z_{q'_0}(\alpha_1)=e^{\pi i(1-3\delta)},\qquad
			Z_{q'_0}(\alpha_2)=m\,e^{\pi i(1-1\delta)},\qquad\\
			Z_{q'_0}(\beta)=M\,e^{1/2 \pi i}\quad\text{for all other }\beta.
		\end{align*}
		The line is chamber preserving. The strip consists of all differentials we encounter by rotating each point in $\ell_1$ by $4\delta$. Let $e^{4\delta \pi i} q'_0 =: q_1$
		
		The saddle connections $\alpha_1$ and $\alpha_2$ are also standard saddle connections for the rotated differential $e^{4\pi i \delta} q_0$, but with reversed orientation. Hence, moving along the path $\ell_1 \times [4\delta]$ amounts to rotating these standard saddles. During this rotation the phases of $\alpha_1$ and $\alpha_2$ remain away from the real axis, so they never become horizontal. Hence this path is chamber preserving as well.
		
		In the second row $[0,4\delta] \times \ell_1(1)$ we crosses infinite walls, whereas in the first row $[0,4\delta] \times \ell_1(0)$ we only cross two walls. There is a codimension 2 wall inside the strip through which we can move homotope to avoid the infinite wall-crossings.
		
		\input{3_Paper_S/Pictures/fig_hexagon1}
		
		\item \emph{Strip 2: $[-2\delta,0] \times \ell_1$ (Figure~\ref{hexagoninquad3}).}
		We now start with $q_1$.
		\begin{align*}
			Z_{q_1}(\alpha_1)=e^{\delta \pi i},\qquad
			Z_{q_1}(\alpha_2)=m\,e^{3\delta \pi i},\qquad\\
			Z_{q_1}(\beta)=M\,e^{\pi i/2}\quad\text{for all}\beta.
		\end{align*}
		Choose $\ell_2$ with $\ell_2(0) = q_1$ so that along $\ell_2$ the saddle connections $\alpha_2$ rotates by $1-4\delta$, while all other standard saddles of $q_1$ including $\alpha_1$  stay fixed. We end up with a differential $q'_1$, with
		\begin{align*}
			Z_{q'_1}(\alpha_1)=e^{\pi i \delta},\qquad
			Z_{q'_1}(\alpha_2)=m\,e^{(1-\delta)\pi i},\qquad\\
			Z_{q'_1}(\beta)=M\,e^{((1/2) + 4\delta)\pi i}\quad\text{for all other }\beta.
		\end{align*}
		The line is chamber preserving. The new strip consists of all differentials we encounter by rotating each point in $\ell_2$ by~$-2\delta$.
		
		After rotating $q_1$ by $-2\delta$, the standard saddle $\alpha_1$ is replaced by $\alpha_3$, which becomes standard only after this rotation.
		The path $\ell_2 \times [-2\delta,0]$ amounts to rotating $\alpha_2$ by $1-4\delta$ while not changing $\alpha_1$, but $\alpha_3$ gets deformed as a consequence.
		The phases of $\alpha_3$ and $\alpha_2$ remain away from the real axis, so they never become horizontal. Hence this path is chamber preserving as well.
		
		\input{3_Paper_S/Pictures/fig_hexagon3}
		
		\item \emph{Strip 3: $[-4\delta,0] \times \ell_2$. (Figure~\ref{hexagoninquad2}).}
		
		We now start with $q_2 := e^{2\delta \pi i} q'_1$.
		\begin{align*}
			Z_{q_2}(\alpha_1)=e^{\delta \pi i}, \; \;
			Z_{q_2}(\alpha_2)=m\,e^{3\delta \pi i}, \\
			Z_{q_2}(\beta)=M\,e^{\pi i/2} \text{ for all other } \beta.
		\end{align*}
		
		Choose $\ell_2$ with $\ell_1(0)=q_2$ so that along $\ell_2$ the saddle connection $\alpha_1$ rotates by $2\delta$ and the saddle connection $\alpha_2$ by $-2\delta$, while all other standard saddles of $q_0$ stay fixed. Let $q'_2$ be the new differential.
		The line is chamber preserving. The strip consists of all differentials we encounter by rotating each point in $\ell_1$ by $-4\delta$. 
		After rotating both ends of $\ell_2$ by $-4\delta$ the standard saddles $\alpha_1$ and $\alpha_2$ are still saddles but with reversed orientation. Hence, moving along the path $\ell_2 \times [-4\delta,0]$ amounts to rotating these standard saddles. During this rotation the phases of $\alpha_1$ and $\alpha_2$ remain away from the real axis, so they never become horizontal. Hence this path is chamber preserving as well.
		
		The rotation in the first row leads to infinitely many wall-crossings but the second row only two walls are crossed.
		\input{3_Paper_S/Pictures/fig_hexagon2}
	\end{itemize}
	
	\smallskip
	\noindent\textbf{Gluing.}
	
	The both horizontal lines of the second strip each coincide with a boundary segments (where the infinite wall crossings happen) of the first and third strip. 
	
	Choose the endpoints of $\ell_1,\ell_2,\ell_3$ so that the following identifications hold:
	\[
	\Phi_1(1,\theta + 2\delta)=\Phi_2(0,\theta - 2\delta)\quad(\theta\in[0, 2\delta]),
	\]
	and
	\[
	\Phi_2(1,\theta - 2\delta)=\Phi_3(0,\theta - 4\delta)\quad(\theta\in[0,2\delta]).
	\]
	Gluing the three rectangles along these rotation sides yields a topological $2$--cell (a disc)
	$D$. Define $\Phi:D\to\mathrm{FQuad}^\circ(S_\triangle)$.
	this is well-defined and continuous by construction.
	
	\smallskip
	\noindent\textbf{Boundary is the hexagon loop.}
	The boundary $\partial D$ consists of six horizontal segments (the $\theta$-rotation edges)
	and three chamber preserving arcs (the vertical sides). The six horizontal segments are precisely
	the \emph{single-flip} passages realizing the six oriented edges of the hexagon in $\EG^\mathbb{T}(\mathbf{S}_\Delta)$.
	Thus, under the identification of cells and walls with triangulations of $\mathbf{S}_\Delta$, $\Phi|_{\partial D}$ represents the image $\mathcal \rho_*(H)$.

	\smallskip
	\noindent\textbf{Conclusion.}
	Since $D$ is a disc and $\Phi:D\ \to \FQuad^\mathbb{T}(\mathbf{S}_\Delta)$ is continuous, the loop
	$\Phi|_{\partial D}$ is null-homotopic in $\FQuad^\mathbb{T}(\mathbf{S}_\Delta)$. Hence the image
	of the hexagon is trivial in $\pi_1\FQuad^\mathbb{T}(\mathbf{S}_\Delta)$. The cases of squares and pentagons are simpler: In both cases one such product strip suffices (with rotation windows contained in $[0,2\delta]$). See Figure~\ref{pentagoninquad} for the pentagon case.

	\input{3_Paper_S/Pictures/fig_pentagon}

\end{proof}

\subsection{Fundamental Group of $\Quad(\mathbf{S})$}\label{subsec:4.2}

Our goal in this subsection is to compute the (orbifold) fundamental group of the moduli space $\Quad(\mathbf{S})$ by comparing it with the corresponding exchange graph. More precisely, we show in Proposition~\ref{theorem:quadeg} that the
(orbifold) fundamental group $\pi_1\Quad(\mathbf{S})$ is isomorphic to the
(orbifold) fundamental group of
\[ \EG^\star(\mathbf{S}) = \EG(\mathbf{S}) / \MCG(\mathbf{S}), \]
after imposing the square, pentagon, and hexagon relations of Section~\ref{sec:eg}.
This statement is implicit in \cite{KQ2}; here we give a direct proof using
only triangulations and quadratic differentials, without appealing to cluster
algebras.

We will first establish that
\[ \pi_1\EG^\TT(\mathbf{S}_\Delta) / \Rel(\mathbf{S}_\Delta) = 1
\qquad\text{(Theorem~\ref{thm:EGdecsimpl-conn})}. \]
From this it follows immediately that the moduli space
$\FQuad(\mathbf{S}_\Delta)$ is simply-connected
(Theorem~\ref{Theorem:pi1trivial}).  
This result will be the main tool in proving Proposition~\ref{theorem:quadeg}.

\begin{theorem}\label{thm:EGdecsimpl-conn} 
	The group $\pi_1 \EG^\TT(\mathbf{S}_\Delta)/\Rel(\mathbf{S}_\Delta)$ is trivial. 
\end{theorem}

\begin{proof}
	From Lemma~\ref{lem:BT_is_quotient} we obtain the short exact sequence
	\[1 \rightarrow \pi_1(\EG^\TT(\mathbf{S}_\Delta), \TT)/\Rel(\mathbf{S}_\Delta)
	\rightarrow
	\pi_1(\EG(\mathbf{S}), T)/\Rel(\mathbf{S}) \xrightarrow{i_\TT} \BT(\TT) \rightarrow 1.\]
	Denote by $\{\gamma_i\}$ the open arcs in $\TT$ and by $\{\eta_i\}$ their dual closed arcs. The group $\pi_1(\EG(\mathbf{S}), T)/\Rel(\mathbf{S})$ is generated by the local twists by Proposition~\ref{prop:MS_pi1EGmodRel}, and the map $i_\TT$ is determined by sending the local twist with respect to an open arc $F(\gamma_i)$ in $T$ to the braid twist $B_{\eta_i}^{-1}$. 
	
	A finite presentation of $\BT(\TT)$ is given in Theorem~\ref{thm:Braid} with respect to the standard generators $\{B_\eta \mid \eta \in \TT^* \}$.
	The (generating) relations are precisely those given in Definition~\ref{def:braidgroup} and we know, by Proposition~\ref{prop:triangulationbraid}, that these relations are satisfied by the local twists in $\TB(T)$. Hence $i_\TT$ has a well-defined inverse, also taking generators to
	generators. Hence,  $i_\TT$ is an isomorphism and $\pi_1(\EG^\TT(\mathbf{S}_\Delta),\TT)/\Rel(\mathbf{S}_\Delta)~=~1$, as required.
\end{proof}

\begin{theorem}\label{Theorem:pi1trivial}
	The fundamental group $\pi_1\FQuad^\TT(\mathbf{S}_\Delta)$ is trivial.
\end{theorem}

\begin{proof}
	By Proposition~\ref{prop:RelTrivial} the square, hexagon and pentagon relations are in the kernel of the surjection $\wp_*: \pi_1 \EG^\mathbb{T}(\mathbf{S}_\Delta) \rightarrow \pi_1 \FQuad^\mathbb{T}(\mathbf{S}_\Delta)$
	from Lemma~\ref{lem:embeddingFQuadS}. By the universal property of the kernel the surjection $\wp_*$ factors through a map
	\[\pi_1\EG^\TT(\mathbf{S}_\Delta)/\Rel(\mathbf{S}_\Delta) \rightarrow \pi_1\FQuad^\TT(\mathbf{S}_\Delta).\]
	By Theorem~\ref{thm:EGdecsimpl-conn} the fundamental group $\pi_1\EG^\TT(\mathbf{S}_\Delta)/\Rel(\mathbf{S}_\Delta)$ is trivial. Hence, the surjective map $\pi_1\EG^\TT(\mathbf{S}_\Delta)/\Rel(\mathbf{S}_\Delta) \rightarrow \pi_1\FQuad^\TT(\mathbf{S}_\Delta)$ is the trivial map.
	
\end{proof}

\begin{prop}\label{theorem:quadeg}
	There is an isomorphism \[\pi_1 \Quad(\mathbf{S}) \simeq \pi_1\EG^\star(\mathbf{S})/\pi_1\EG^\TT(\mathbf{S}_\Delta)\]
	of orbifold fundamental groups, where $\pi_1\EG^\TT(\mathbf{S}_\Delta) \simeq \Rel^\star(\mathbf{S}) \subset \pi_1\EG^\star(\mathbf{S})$ is the normal subgroup generated by the square, pentagon, and hexagon relations.
\end{prop}

\begin{proof}
	We recall that $\MCG(\mathbf{S}_\Delta)$ acts on $\FQuad^\mathbb{T}(\mathbf{S}_\Delta)$ with quotient $\Quad(\mathbf{S})$. By Lemma~\ref{lemma:Deck}, this action defines a covering  
	\[
	F: \FQuad^\mathbb{T}(\mathbf{S}_\Delta) \to \Quad(\mathbf{S})
	\]  
	with Deck group $\MCG(\mathbf{S}_\Delta)$. As a consequence of this and Lemma~\ref{lem:embeddingFQuadS}, we obtain the following commutative diagram:  
	
	\[
	\begin{tikzcd}
		B_2(\mathbf{S}_\Delta) \rar[hook] \dar[two heads]{\MCG(\mathbf{S}_\Delta)} & \FQuad^\mathbb{T}(\mathbf{S}_\Delta) \dar[two heads]{\MCG(\mathbf{S}_\Delta)} \\
		F(B_2(\mathbf{S}_\Delta)) \rar[hook] & \Quad(\mathbf{S})
	\end{tikzcd}
	\]
	
	Moreover, Lemma~\ref{lem:embeddingFQuadS} implies the isomorphisms  
	\[
	\pi_1(B_2(\mathbf{S}_\Delta)) \simeq \pi_1\EG^\mathbb{T}(\mathbf{S}_\Delta), \quad \text{and} \quad \pi_1(F(B_2(\mathbf{S}_\Delta))) \simeq \pi_1 \EG^\star(\mathbf{S}).
	\] 
	Combining this with the previous diagram ~\ref{lem:embeddingFQuadS} yields the following commutative diagram.
	
	\[
	\begin{tikzcd}
		& \ker(f) \dar & \ker(g) \dar & 1 \dar & \\
		1 \rar & \pi_1\EG^\mathbb{T}(\mathbf{S}_\Delta) \rar{i_1} \dar[two heads]{f} & \pi_1 \EG^\star(\mathbf{S}) \rar{p_1} \dar[two heads]{g} & \MCG(\mathbf{S}_\Delta) 
		\rar \dar{\id} & 1 \\
		1 \rar & \pi_1\FQuad^\mathbb{T}(\mathbf{S}_\Delta) \rar{i_2} & \pi_1\Quad(\mathbf{S}) \rar{p_2} & \MCG(\mathbf{S}_\Delta) \rar & 1
	\end{tikzcd}
	\]
	
	By Theorem~\ref{Theorem:pi1trivial}, the group $\pi_1\FQuad^\mathbb{T}(\mathbf{S}_\Delta)$ is trivial. Consequently, $\ker(f) = \pi_1\EG^\mathbb{T}(\mathbf{S}_\Delta)$ .
	Since the right vertical identity on the mapping class group is obviously injective, by the snake lemma we have a short exact sequence of the kernels
	\[0 \rightarrow \ker(f) \rightarrow \ker(g) \rightarrow 1 \rightarrow 1.\]
	Hence, $\ker(g) = \ker(f) = \pi_1\EG^\mathbb{T}(\mathbf{S}_\Delta)$.
	
	
	By Proposition~\ref{prop:RelTrivial} the subgroup, $\Rel^\star(\mathbf{S})$, generated by squares, hexagons and pentagons lies in $\ker(g)$ and by  Theorem~\ref{thm:EGdecsimpl-conn} this subgroup lies in $\pi_1\EG^\mathbb{T}(\mathbf{S}_\Delta)$.
	Hence, $\ker(g)~=~\pi_1\EG^\mathbb{T}(\mathbf{S}_\Delta)$ is the subgroup generated by squares, hexagons and pentagons.
\end{proof}

Let $$\FQuad^\TT(\mathbf{S}_\Delta) \longrightarrow Q \quad\text{and}\quad \EG^\TT(\mathbf{S}_\Delta) \longrightarrow E$$ be any other choice of base maps, where $E$ denotes the (quotient) graph that is homotopy equivalent to the wall–and–chamber structure in $Q$, defined as before. As in the previous proof, the group $\pi_1\EG^\TT(\mathbf{S}_\Delta)$ naturally embeds as a subgroup of~$\pi_1(E)$. This subgroup is generated by all square, pentagon, and hexagon relations in~$E$. 
We will denote this subgroup in the following corollary by~$\Rel(E)$.

\begin{cor}\label{remKQ}
	There is an isomorphism $$\pi_1(Q) \;\cong\; \pi_1(E)\big/\Rel(E).$$
\end{cor}

\begin{proof}
	Replace $\Quad(\mathbf{S})$ by $Q$ in the proof of Proposition~\ref{theorem:quadeg}.
\end{proof}

In the following sections, we will either work with $\Quad(\mathbf{S})$, the space of quadratic differentials obtained by forgetting the framing and decorations, or with $\Quad^{\mathrm{lab}}(\mathbf{S})$, where the zeroes are labelled.

\subsection{Higher-Order Zeroes}\label{subsec:4.3}
Our goal in this subsection is to show that all loops in
Definition~\ref{def:wDMS_Rel} vanish under the map
\[
\rho_{\mathbf{S}_{\mathbf{w}}} : 
\pi_1 \EG^\circ(\mathbf{S}_{\mathbf{w}}) \longrightarrow 
\pi_1 \FQuad^\circ(\mathbf{S}_{\mathbf{w}})
\]
from Lemma~\ref{lem:embeddingFQuadS_w}.  
As the trajectory structure locally behaves similar to the simple-zero case for the square, pentagon, and the first hexagon configuration, the corresponding parts of the proof can be copied verbatim.  
Thus it remains only to verify the vanishing of the second hexagon relation, which occurs solely in the presence of a higher–order zero.

We do not claim this time that  these relations generate the full kernel of
$\rho_{\mathbf{S}_{\mathbf{w}}}$, only their vanishing is established here.

\begin{prop}\label{prop:relationsinkernel_gen}
	Let $\mathbf{S}_{\mathbf{w}}$ be a weighted decorated marked surface and
	$\EG^\circ(\mathbf{S}_{\mathbf{w}})$ the oriented exchange graph of
	$\mathbf{w}$-mixed-angulations.
	Under the embedding
	\[
	\wp_{\mathbf{S}_{\mathbf{w}}}:
	\EG^\circ(\mathbf{S}_{\mathbf{w}})
	\hookrightarrow
	\FQuad^\circ(\mathbf{S}_{\mathbf{w}}),
	\]
	the loops corresponding to all generalized square, pentagon, and hexagon
	(both types) relations in $\EG^\circ(\mathbf{S}_{\mathbf{w}})$ (see Definition~\ref{def:wDMS_Rel})
	are null-homotopic.
	Equivalently, the subgroup
	$\Rel(\sow)$
	lies in the kernel of the induced map
	\[
	\pi_1\EG^{\circ}(\sow)
	\longrightarrow
	\pi_1\FQuad^{\circ}(\sow).
	\]
\end{prop}

\begin{proof}
	Locally near the relevant pair of saddle connections, the horizontal foliation
	is the same as in the simple-zero case, up to the presence of additional half-planes
	which do not affect the local wall-crossing pattern.
	Thus the arguments used for squares, pentagons, and the first hexagon relation
	in Proposition~\ref{prop:RelTrivial} carry over unchanged.
	We are left to show the second hexagon relation.
	Fix a hexagon $H$ sub graph in $\EG^\circ(\sow)$ determined by two arcs intersecting twice in the same polygon.
	
	Choose $q_0\in \FQuad^\circ(\sow)$ whose mixed-angulation is the initial vertex of this hexagon, and let
	$\alpha_1,\alpha_2$ be the corresponding saddle connections.
	Fix small $\delta>0$, a positive $m < 1$, and a large $M\gg1$,
	and arrange
	\begin{align*}
		Z_{q_0}(\alpha_1)=e^{(1-\delta) \pi i},\qquad
		Z_{q_0}(\alpha_2)=m\,e^{(1-3\delta) \pi i},\qquad\\
		Z_{q_0}(\beta)=M\,e^{1/2 \pi i}\quad\text{for all other }\beta.
	\end{align*}
	
	\input{3_Paper_S/Pictures/fig_2ndhexagon}
	
	See Figure~\ref{2ndhexagoninquad}. Here the black bold lines denotes slits separating the two halves with boundaries labelled by $a$ and $b$. The boundaries with same labelled should be thought as being glued together.
	The grey area is a part of the surface that plays no important role for this proof.
	The horizontal path in Figure~\ref{2ndhexagoninquad} translates to rotating the differential. The two left differentials in the picture are in the same chamber, since combinatorially both pictures just differ by cutting out the lower horizontal strip rotating it by $\pi$ and gluing it back in. Same goes for the two differentials on the right hand side.
	
	Let $\ell$ be the chamber preserving line realized by rotating $Z_{q_0}(\alpha_1)$ by $-2\delta$ and $Z_{q_1}(\alpha_1)$ by~$2\delta$ and ending in the differential $q_1$ with
	\begin{align*}
		Z_{q_1}(\alpha_1)=e^{(1-3\delta) \pi i},\qquad
		Z_{q_1}(\alpha_2)=m\,e^{(1-\delta) \pi i},\qquad\\
		Z_{q_1}(\beta)=M\,e^{1/2 \pi i}\quad\text{for all other }\beta.
	\end{align*}
	In each point of this line segment $\ell$ we can rotate by $4 \delta$. This gives us a contractible rectangle in $\FQuad^\circ(\sow)$. The line $e^{4\delta \pi i} \ell$ (right vertical arrow in the picture) is chamber preserving as well.
	
	The boundary of the rectangle consists of six horizontal segments, corresponding to rotations, and two vertical segments, corresponding chamber-preserving deformations. The horizontal segments are exactly the \emph{single-flip} passages and realise the six oriented edges of the hexagon in $\EG^\circ(\sow)$.
	
	Under the identification of cells and walls with mixed-angulations, the boundary of the strip therefore represents the image $\mathcal \wp_{\sow}(H)$. More precisely, along the boundary each differential admits two relevant saddle connections, which correspond to the two edges involved in the hexagon flip. Traversing the upper and lower boundaries of the strip thus realises the flip sequences along the upper and lower halves of the hexagon, respectively; compare the hexagon
	\begin{center}
		\input{3_Paper_S/Pictures/gen_hex2_relation_deco2}
	\end{center}
	with Figure~\ref{2ndhexagoninquad}.

	\end{proof}
	
	\begin{cor} \label{cor:relationsinkernel_gen}
		Similarly the corresponding loops $\Rel^\star(\sow)$ in the unframed stratum are in the kernel of
		\[
		\wp_*:\pi_1\EG^\star(\mathbf{S}_{\mathbf{w}})
		\twoheadrightarrow \pi_1\Quad(\sow).
		\]
	\end{cor}

\medskip
\noindent\textbf{On the kernel in the higher-order case.}
As in the simple-zero situation, we obtain a commutative diagram of short exact sequences
\[
\begin{tikzcd}
	& \ker(f) \dar & \ker(g) \dar & 1 \dar & \\
	1 \rar & \pi_1\EG^\circ(\sow) \rar{i_1} \dar[two heads]{f} & \pi_1 \EG^\star(\sow) \rar{p_1} \dar[two heads]{g} & \MCG(\sow) 
	\rar \dar{\id} & 1 \\
	1 \rar & \pi_1\FQuad^\circ(\sow) \rar{i_2} & \pi_1\Quad(\sow) \rar{p_2} & \MCG(\sow) \rar & 1
\end{tikzcd}
\]
where $g$ is induced by the embedding of the skeleton. This shows that $\ker(f)=\ker(g)$.
Proposition~\ref{prop:relationsinkernel_gen} shows that all generalized square,
pentagon, and hexagon relations therefore lie in $\ker(g)$.

In the simple-zero case, the additional input that
$\pi_1\FQuad^\TT(\mathbf{S}_\Delta)=1$ allowed us to identify $\ker(g)$
\emph{precisely} with, $\Rel(\mathbf{S}_\Delta)$ the normal subgroup generated by these local relations.
For higher-order zeroes, such a global simple-connectedness result is not available in general, so we only conclude that these relations are contained in $\ker(g)$, not that they generate it.

\subsection{Orbifold fundamental groups of EGs with symmetries}
\label{subsection:bass-serre}

When passing from framed differentials or differentials with labelled zeroes to
unframed differentials with unlabelled zeroes, symmetries of the underlying
triangulation or mixed--angulation may identify several chambers in the stratum,
or equivalently several vertices of the exchange graph.
This phenomenon occurs precisely when two or more zeroes have equal order.
As a result, the exchange graph of unframed differentials with unlabelled zeroes
naturally acquires the structure of an \emph{orbifold graph}:
vertices may carry finite stabilizer groups, while all edges have trivial stabilizers.

\medskip

To formalize this situation we use Bass--Serre theory for graphs of groups.
For background we refer to \cite{dicks2006groups};
however, throughout this paper we only require the elementary special case in which
all edge groups are trivial and all vertex groups are finite.
In this setting, the orbifold fundamental group coincides with the Bass--Serre
fundamental group of the associated graph of groups.

\medskip

Let $X$ denote the exchange graph in the case of labelled zeroes, and let $G$ be the
finite group acting on $X$ by permuting zeroes of equal order.
The quotient
\[
Y := X/G
\]
is the exchange graph of the unlabelled stratum.
Its orbifold structure is encoded by the finite stabilizer groups
\[
G_v := \{\, g \in G \mid g(\tilde v) = \tilde v \,\}
\]
attached to vertices $v\in Y$, where $\tilde v$ is any chosen lift of $v$ to $X$.
Since no nontrivial symmetry fixes a flip, all edge stabilizers are trivial.

\medskip

The appropriate notion of fundamental group for such a space is the
\emph{orbifold fundamental group} $\pi_1^{\mathrm{orb}}(Y)$,
defined as the group of deck transformations of the universal orbifold cover.
As in the classical case, every connected graph admits a universal covering tree;
in the orbifold setting, this tree carries an action of $\pi_1^{\mathrm{orb}}(Y)$
whose quotient is $Y$ and whose vertex stabilizers are precisely the groups~$G_v$.

\medskip

The same data can equivalently be described as a \emph{graph of groups}:
the underlying graph is $Y$, the vertex group at $v$ is $G_v$, and all edge groups are
trivial.
Bass--Serre theory associates to this graph of groups the Bass--Serre fundamental group
$\pi_1^{\mathrm{BS}}(Y)$, which acts on a tree (the Bass--Serre tree) with quotient $Y$
and vertex stabilizers $G_v$.
In the case of trivial edge groups, this group admits a particularly simple
presentation (see Proposition~\ref{prop:orbpi1-graph} below).
Since the Bass--Serre tree is canonically identified with the universal orbifold cover,
the two constructions produce the same group, yielding a canonical identification
\[
\pi_1^{\mathrm{orb}}(Y) \;\cong\; \pi_1^{\mathrm{BS}}(Y).
\]

\begin{prop}[Bass--Serre, cf.\ \S 4.1 \cite{dicks2006groups}]
	\label{prop:orbpi1-graph}
	Let $Y$ be a connected graph equipped with finite vertex stabilizers $G_v$ and
	trivial edge stabilizers, and let $T \subset Y$ be a spanning tree.
	Then the Bass--Serre fundamental group of the associated graph of groups is
	\[
	\pi_1^{\mathrm{BS}}(Y)
	\;\cong\;
	\langle\, t_e \mid e \in E(Y)\setminus E(T) \,\rangle
	\;*\;
	\Big( *_{v\in V(Y)} G_v \Big),
	\]
	that is, the free product of all vertex stabilizer groups together with a free group generated by the edges not contained in~$T$.
	In particular, this group acts on a tree with quotient $Y$ and vertex stabilizers~$G_v$.
	
	In our situation, this tree is canonically identified with the universal orbifold cover, and hence $\pi_1^{\mathrm{BS}}(Y)$ coincides with the orbifold fundamental group $\pi_1^{\mathrm{orb}}(Y)$.
\end{prop}

\medskip

In particular, a vertex with trivial stabilizer contributes nothing, while a vertex with
stabilizer $\mathbb{Z}_2$ contributes a factor of~$\mathbb{Z}_2$.
This description will be used later to compute the orbifold fundamental groups of certain strata explicitly, where the nontrivial vertex stabilizers encode the monodromy arising from permuting zeroes of equal order in mixed--angulations.

%% file: sec_ex.tex
\section{Examples and the New Case}
\label{sec:examples}

In Proposition~\ref{prop:relationsinkernel_gen} we constructed, for any undecorated and unframed stratum of meromorphic quadratic differentials with higher-order zeroes, a natural
surjective morphism
\[ \pi_1\EG^\star(\sow)/\Rel^\star(\sow)
\;\twoheadrightarrow\; \pi_1\Quad(\sow), \]
where $\pi_1\EG^\star(\sow)$ is the quotient graph by the mapping class group  (see Remark~\ref{rem:EG_wstar}).
We focus on the stratum of quadratic differentials on the sphere with $4$ singularities. In other words, we have the cases
\begin{itemize}
	\item[(i)] one pole: the marked surface $\mathbf{S}$ is a disc and there are $3$ decorations,
	\item[(ii)] two poles: the marked surface $\mathbf{S}$ is a cylinder and there are $2$ decorations.
\end{itemize}
The goal of this section is to examine whether this surjective morphism is in fact an isomorphism.

We proceed in three steps.  
First, we compute the fundamental group of $\Quad(\sow)$ directly in the
case of genus $g=0$ with a single pole.  This provides a baseline
calculation independent of the exchange-graph formalism, which we will use
for comparison.  
Next, as a sanity check, we test the exchange graph approach in the
well-studied example $\Quad(1,1,1,-7):= \Quad(\sow)$ for $\mathbf{w} = (1,1,1)$ and $\mathbf{S}$ the disc with $5$ marked points on its boundary, showing that the group obtained from $\pi_1\EG^\star(\sow) \big/ \Rel^\star(\sow)$ agrees with the known presentation.  
Finally, we turn to new examples, namely strata on the sphere with higher-order zeroes and at most two poles, where—as usual—the poles have order at least~$3$. In this setting, we show that the surjective map above is in fact an isomorphism.

\subsection{One Pole on the Sphere: Direct Computation}

Let $\mathbf{S_w}$ be a wDMS with only one boundary component and $n>2$ labelled zeroes with zero orders given by weights $w~=~(k_1,...,k_n)$ and $-(k_1+...+k_n+2)$ many marked points.

\begin{prop}\label{prop:piQuad=PZ}
	The fundamental group of the stratum with labelled zeroes is the direct product
	\begin{align*}
		\pi_1\Quad^\mathrm{lab}(\sow)&=  \Pure_{n-2}(\bC\backslash\{0,1\}) \times \Z.
	\end{align*}
	of a pure surface braid group over the twice punctured plane and $\Z$.
\end{prop}

Although this decomposition resembles the structure
\[\PB_n \simeq \Pure_{n-2}(\bC\backslash\{0,1\}) \rtimes_\varphi \Z\] 
of the pure braid group as shown in Lemma~\ref{lem:PBsplit}.
The product in our case is \emph{direct} rather than semi direct, reflecting the fact that the $\mathbb{C}^\times$-action on $\Quad^\mathrm{lab}_{\mathbf{w}}(\mathbf{\mathbf{S}})$ is geometrically independent of the configuration of zeroes.

For $n=3$ this semi direct product is a direct one.
So in this instance the pure braid group $\PB_3$ coincides with the fundamental group $\pi_1\Quad^\mathrm{lab}(\sow) = \Z \times F_2$.

\begin{proof}[Proof of the Proposition]
Fix a coordinate sending the pole to $\infty$ and two labelled zeroes to
$0$ and $1$. Every $q\in\Quad^\mathrm{lab}(\sow)$ can then be written uniquely as
\[
q = c\,\prod_{i=1}^n (z-\lambda_i)^{k_i}\,dz^2,
\qquad
\lambda_1=0,\ \lambda_2=1,\ 
\lambda_3,\dots,\lambda_n\in\bC\setminus\{0,1\},
\]
with $c\in\bC^\times$. Hence
\[
\PP\Quad^\mathrm{lab}(\sow) \;\cong\;
\F_{n-2}(\bC\setminus\{0,1\})\]
and
\[
\Quad^\mathrm{lab}(\sow) \;\cong\;
\PP\Quad^\mathrm{lab}(\sow) \times\bC^\times.
\]
Therefore,
\[
\pi_1\Quad^\mathrm{lab}(\sow) \cong
\pi_1(F_{n-2}(\bC\setminus\{0,1\}))\times\Z.
\]
\end{proof}


\subsection{The Sanity Check: $\Quad(1,1,1,-7)$}
\label{sec:111-7}

We now want to check Corollary~\ref{remKQ} for the case three simple zeroes and one pole of order $7$ on the sphere. 

\medskip
\noindent\textbf{Labelled Zeroes.} From the previous subsection we know that $$\pi_1\Quad^\mathrm{lab}(1,1,1,-7) = \Z \times F_2.$$

We now recover the same fundamental group from the exchange graph
(see Remark~\ref{remKQ}).
We construct the exchange graph $(\EG^\star)^\mathrm{lab}(\sow)$, whose
vertices correspond to triangulations of a pentagon, considered up to
rotation, with each triangle labelled.

From any vertex one may perform a forward flip along either the left or
the right interior edge; denote these directed edges by $\ell$ and $r$,
respectively. The resulting exchange graph is shown in
Figure~\ref{eg1117}.

\begin{figure}[htbp!]
	\centering
	\input{3_Paper_S/Pictures/EG111}
	\caption{Exchange\\
		 graph for $\Quad^{\mathrm{lab}}(1,1,1,-7)$.}
	\label{eg1117}
\end{figure}

The fundamental group of the underlying coarse graph is free on the edges
which are not contained in a chosen spanning tree. Choosing the bold edges
as a spanning tree in Figure~\ref{eg1117}, we obtain the free generators
\[
d_1, d_2, d_3, d_4, a, b, c.
\]
Since each triangulation has exactly two interior edges, and these edges
are adjacent only once, all relations arise from pentagon moves. They read
\[
d_3 d_1
= c
= d_2 d_4
= a
= d_1 d_3
= b
= d_4 d_2.
\]
Thus, after quotienting by these relations, we obtain
\[
\langle z, d_1, d_2 \mid [z,d_1], [z,d_2] \rangle
=
\langle z, \alpha, \beta \mid [z,\alpha], [z,\beta] \rangle
\cong \mathbb{Z} \times F_2,
\]
where $z := a=b=c$, $\alpha := d_1 c$, and $\beta := a d_2$.

These generators admit the following geometric interpretation:
\begin{itemize}
	\item $z$ corresponds to a full rotation, i.e.\ the natural $\mathbb{C}^\star$–action,
	\item $\alpha$ is the pure braid twist around the first and second zero,
	\item $\beta$ is the pure braid twist around the second and third zero.
\end{itemize}
This agrees with the direct computation above and confirms
Proposition~\ref{theorem:quadeg} in this case.

\medskip
\noindent\textbf{Unlabelled Zeroes.} We now forget the ordering of the zeroes.
Consider the covering
\[
\Quad^{\mathrm{lab}}(\sow) \longrightarrow \Quad(\sow)
\]
with covering group $S_3$.
It is a classical fact that
\[
\pi_1\Quad(1,1,1,-7)
\simeq
B_3
=
\langle a,b \mid aba = bab \rangle,
\]
which we now verify again via the exchange graph.
Passing to the quotient $\EG^\star(\sow)$ yields a
graph with a single vertex and two loops, labelled $a$ and $b$,
corresponding to the two possible flips of the interior edges
(Figure~\ref{EG111_unlabelled}).

\begin{figure}[htbp!]
	\centering
	\begin{tikzpicture}[scale=.35]
		\coordinate (A) at (1,1);
		\foreach \ang in {18,90,162,234,306}{
			\coordinate (P\ang) at ($(A) + (\ang:2.25)$);
			\fill (P\ang) circle (1.2mm);
		}
		\draw (P18) -- (P90);
		\draw (P90) -- (P162);
		\draw (P162) -- (P234);
		\draw (P234) -- (P306);
		\draw (P306) -- (P18);
		\draw (P90) -- (P234);
		\draw (P90) -- (P306);
		\draw [->] ($(A) + (3,1)$) to [in=330,out=30,looseness=5] ($(A) + (3,-1)$);
		\node at ($(A) + (5,0)$) {\small $a$};
		\draw [->] ($(A) + (-3,-1)$) to [in=150,out=210,looseness=5] ($(A) + (-3,1)$);
		\node at ($(A) + (-5,0)$) {\small $b$};
	\end{tikzpicture}
	\caption{The quotient $\EG^\star(\sow)$.}
	\label{EG111_unlabelled}
\end{figure}

The fundamental group of this graph is the free group on two generators,
\[
\pi_1\EG^\star(\sow)
=
\langle \ell, r \rangle.
\]
We describe the induced map
\[
g \colon
\pi_1\EG^\star(\sow)
\longrightarrow
\pi_1\Quad(\sow).
\]

We interpret the labelled triangles as numbered strands.
In $\B_3$, the generator $a$ is the braid twist between strands 1 and 2,
and $b$ is the twist between strands 2 and 3.
Inspecting Figure~\ref{eg1117}, we find:
\begin{itemize}
	\item $r$ corresponds to $ba$ (counterclockwise rotation),
	\item $\ell$ corresponds to $aba$ (the twist between strands 1 and 3).
\end{itemize}
Thus
\[
g(\ell) = aba,
\qquad
g(r) = ba.
\]

There are no square relations, and the only pentagon relation is
$\ell^2 = r^3$.
Using $aba = bab$ in $\B_3$, we compute
\[
g(\ell^{-2} r^3)
=
g(\ell)^{-2} g(r)^3
=
(aba)^{-2} (ababab)
=
1,
\]
so $g$ factors through the quotient
\[
\pi_1\EG^\star(\sow)/\langle \ell^2=r^3\rangle
=
\langle \ell, r \mid \ell^2 = r^3 \rangle
\longrightarrow
\B_3
=
\pi_1\Quad(1,1,1,-7).
\]
This map is an isomorphism with inverse
\[
a \mapsto r\ell^{-1},
\qquad
b \mapsto \ell^{-1} r^2.
\]

\medskip
\noindent\textbf{Conclusion.}
The quotient of the fundamental group of the exchange graph by the single pentagon relation $\ell^2 = r^3$ is isomorphic to $\pi_1\Quad(1,1,1,-7)$ via the induced map.
This serves as a consistency check for Corollary~\ref{remKQ}.

\subsection{Three zeroes and one pole on $\mathbb{P}^1$}
\label{sec:3zero1pole}

We now turn to strata of unframed and unlabelled differentials on the genus $0$ surface with a single boundary component and weight $\mathbf{w} = (i,j,k)$.
In other words, quadratic differentials on $\mathbb{P}^1$ with a single pole, and three zeroes which need not have distinct orders.

By Proposition~\ref{prop:relationsinkernel_gen} the map from Lemma~\ref{lem:embeddingQuadS_w} factors through the quotient
\[
\pi_1\EG^\star(\sow)/\Rel^\star(\sow)
\;\twoheadrightarrow\;
\pi_1\Quad(\sow).
\]
Our aim is to describe the exchange graph $\EG^\star(\sow)$ for the unframed and unlabelled stratum $\Quad(\sow)$, identify the generators of its fundamental group, record the local relations $\Rel^\star(\sow)$ (squares, pentagons, and both kinds of hexagons), and finally to compute the resulting quotients in all cases including those where identical zeroes or parity conditions on zero orders produce additional symmetries that must be taken into account.

\medskip
\noindent\textbf{All Three Zeroes of Different Order.}
We begin with the case in which the three zeroes have distinct orders.
The quotient $\EG(\mathbf{S}_\mathbf{w})/\MCG(\mathbf{S}_\mathbf{w})$ is shown in Figure~\ref{EGijk}.

\begin{figure}[ht]
	\centering
	\input{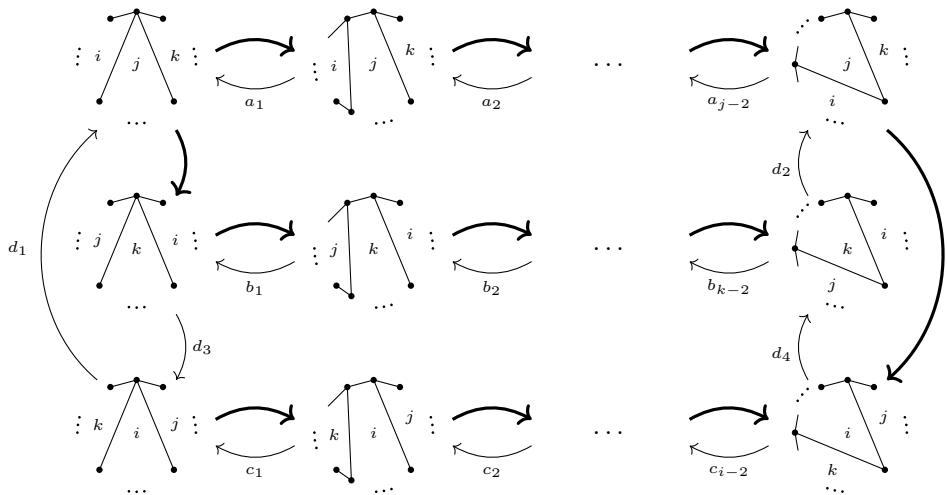}
	\caption{Exchange graph of the stratum $(i,j,k)$ with $i,j,k$ pairwise distinct.}
	\label{EGijk}
\end{figure}

The fundamental group of this graph is a free group generated by closed paths corresponding to the edges in the complement of a chosen spanning tree. Fixing such a spanning tree and a base vertex yields generators
\[
a_1,\dots,a_{j-2},\;
b_1,\dots,b_{k-2},\;
c_1,\dots,c_{i-2},\;
d_1,d_2,d_3,d_4.
\]
Each generator is represented by an oriented edge outside the spanning tree together with the associated closed path based at the chosen vertex, obtained by traversing the unique path in the spanning tree to the initial vertex of the edge, crossing the edge, and returning along the spanning tree. By abuse of notation, we denote both the edge and the corresponding closed path by the same symbol. In Figure~\ref{EGijk}, the spanning tree is indicated by bold arrows.

At each vertex one obtains either a pentagon relation (if the two interior
edges intersect) or a square relation (if they are disjoint):
\begin{align*}
	d_1 d_3
	&= a_1 = \dots = a_{j-2} = d_2 d_4, \\
	d_1 d_3
	&= b_1 = \dots = b_{k-2} = d_4 d_2, \\
	d_3 d_1
	&= c_1 = \dots = c_{i-2} = d_2 d_4.
\end{align*}
Setting
\begin{align*}
z := a_1 = \dots = a_{j-2}
= b_1 = \dots = b_{k-2}
= c_1 = \dots = c_{i-2},
\\
\alpha := d_1 c_1 \cdots c_{i-2},
\qquad
\beta := a_1 \cdots a_{j-2} d_2,
\end{align*}
we obtain the quotient
\[
\langle z,\alpha,\beta \mid [z,\alpha],[z,\beta] \rangle
\cong \mathbb{Z} \times F_2.
\]

When some of the zeroes have the same order, the corresponding unlabelled
stratum is obtained as a quotient of the labelled stratum by the symmetric
group permuting the indistinguishable zeroes. This permutation action
extends to the wall-chamber decomposition and hence to the exchange
graph. For $i=j\neq k$ the relevant group is $S_2$, and for $i=j=k$ it is
$S_3$. In both cases the quotient graph encodes the wall-chamber
structure of the unlabelled stratum.

A key point is that an orbifold point appears precisely when there is a
chamber fixed by the symmetric group action. Geometrically, a chamber is
fixed if and only if the pole has an even number of horizontal prongs,
i.e.\ if and only if the corresponding pole order is even. In our notation
this occurs exactly when $k$ is even. Thus:
\begin{itemize}
	\item if $k$ is even, the pole has an even number of prongs, there is
	a fixed chamber, and the quotient acquires a single orbifold point
	of order~$2$;
	\item if $k$ is odd, the pole has an odd number of prongs, no chamber
	is fixed, and the action is free, so the quotient has no orbifold
	point.
\end{itemize}

\medskip
\noindent\textbf{Two Zeroes of the Same Order.}
\label{sec:2equal1different}

\medskip
\noindent\textbf{Case (A): $k$ even.}

\begin{figure}[htbp!]
	\centering
	\input{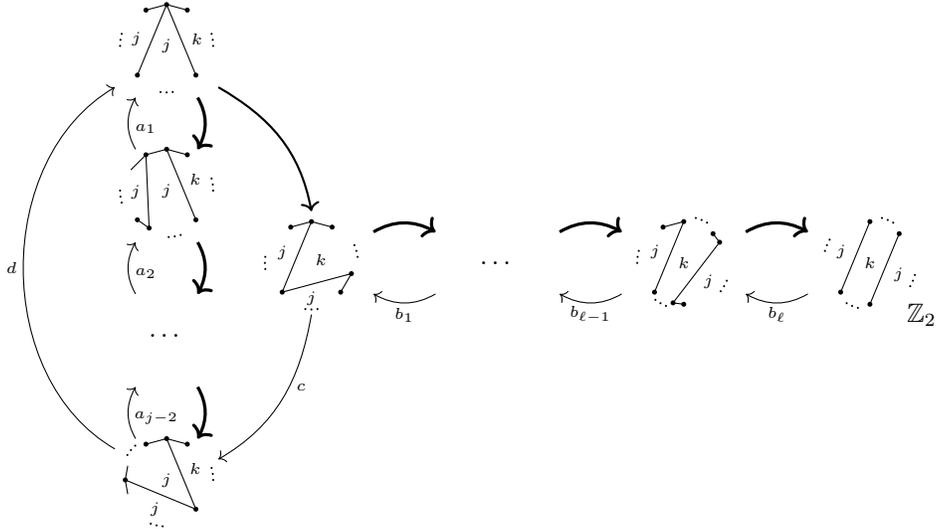}
	\caption{Coarse exchange graph of the stratum $(j,j,k)$ for $k$ even. Here $\ell: = \frac{k-2}{2}$.}
	\label{EGjjkeven}
\end{figure}

Assume $i=j \neq k$ and $k$ even. Figure~\ref{EGjjkeven} shows the coarse
exchange graph, obtained by naively quotienting the graph in Figure~\ref{EGijk} by the $S_2$-action. The stabilizers, when non-trivial, are indicated below the corresponding vertices.

The actual exchange graph has a single orbifold point of order~$2$,
corresponding to the unique fixed chamber under $S_2$. By Proposition~\ref{prop:orbpi1-graph} the orbifold
fundamental group is the free product of $\mathbb{Z}_2$ (generated by $s$)
with the free group on
\[
a_1,\dots,a_{j-2},\quad
b_1,\dots,b_\ell,\quad
c,\ d.
\]
Here $s$ is the loop in that chamber representing the local $\mathbb{Z}_2$
symmetry which swaps the two equal-order zeroes.

The pentagon and square relations are as before, except at the orbifold
vertex. There we obtain the square relation
\[
b_\ell = s^{-1} b_\ell s,
\]
which is the image of the square relation from the fundamental group of the labelled exchange graph into the one of the unlabelled one.
Intuitively: the flip $b_\ell$ acts on an edge fixed by $S_2$; to flip the
opposite edge one first applies $s$, then $b_\ell$, then $s^{-1}$.

Collecting relations
\[
dc = a_1 = \dots = a_{j-2} = cd,
\qquad
dc = b_1 = \dots = b_\ell = s^{-1} b_\ell s,
\]
and setting
\[
z := a_1 = \dots = a_{j-2}
= b_1 = \dots = b_\ell
= cd = dc,
\]
we obtain
\[
\langle z,s,d \mid [z,s],[z,d], s^2=1 \rangle.
\]

\medskip
\noindent\textbf{Case (B): $k$ odd.}

\begin{figure}[htbp!]
	\centering
	\input{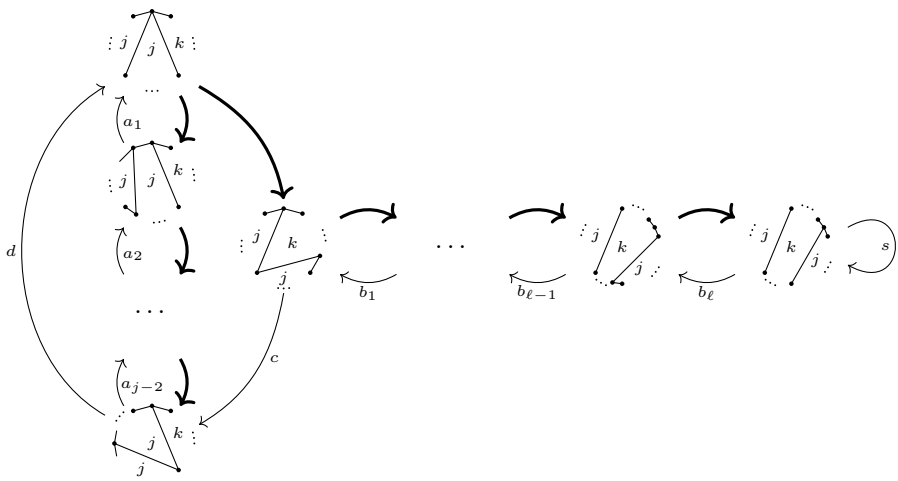}
	\caption{Coarse exchange graph of the stratum $(j,j,k)$ for $k$ odd. Here $\ell = \frac{k-3}{2}$.}
	\label{EGjjkodd}
\end{figure}

Now assume $i=j \neq k$ and $k$ odd. Figure~\ref{EGjjkodd} shows the coarse
exchange graph. In this case the $S_2$-action is free, so there is no
orbifold point.

The fundamental group is now the free group on
\[
a_1,\dots,a_{j-2},\quad
b_1,\dots,b_\ell,\quad
c,\ d,\ s.
\]
The relations are
\[
dc = a_1 = \dots = a_{j-2} = cd,
\qquad
dc = b_1 = \dots = b_\ell = s^2.
\]
Setting
\[
z := a_1 = \dots = a_{j-2}
= b_1 = \dots = b_\ell
= cd = dc,
\]
we obtain
\[
\langle z,s,d \mid [z,s],[z,d], s^2 = z \rangle
=
\langle s,d \mid [s^2,d] \rangle.
\]

\medskip
\noindent\textbf{All Three Zeroes of the Same Order.} Finally, consider $i=j=k$. The acting group is now~$S_3$.

\medskip
\noindent\textbf{Case (A): $k$ even.}

\begin{figure}[ht]
	\centering
	\input{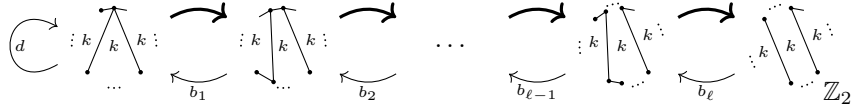}
	\caption{Coarse exchange graph of the stratum $(k,k,k)$ for $k$ even. Here $\ell = \frac{k-2}{2}$.}
	\label{EGkkkeven}
\end{figure}

The quotient has a single orbifold point of order~$2$.
By Proposition~\ref{prop:orbpi1-graph} the orbifold fundamental group is the free product of $\mathbb{Z}_2$
(generated by $s$) and the free group on
\[
b_1,\dots,b_\ell,\ d.
\]
The relations are
\[
d^3 = b_1 = \dots = b_\ell = s^{-1} b_\ell s.
\]
Setting
\[
z := b_1 = \dots = b_\ell = d^3,
\]
we obtain
\[
\langle z,s,d \mid [z,s],\ d^3=z,\ s^2=1 \rangle
=
\langle s,d \mid [d^3,s],\ s^2=1 \rangle
\]

\medskip
\noindent\textbf{Case (B): $k$ odd.}

\begin{figure}[htbp!]
	\centering
	\input{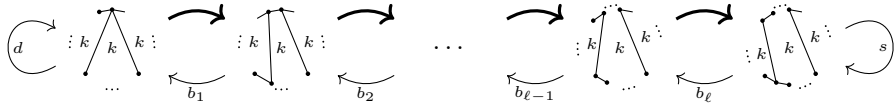}
	\caption{Coarse exchange graph of the stratum $(k,k,k)$ for $k$ odd. Here $\ell = \frac{k-3}{2}$.}
	\label{EGkkkodd}
\end{figure}

Now the $S_3$-action is free, so there is no orbifold point.
The fundamental group is free on
\[
b_1,\dots,b_\ell,\ d,\ s,
\]
and we have
\[
d^3 = b_1 = \dots = b_\ell = s^2.
\]
Letting $z := b_1 = \dots = b_\ell = d^3$, we obtain
\[
\langle z,s,d \mid [z,s],[z,d] \rangle
=
\langle s,d \mid d^3 = s^2 \rangle.
\]

\subsection{Two Zeroes and Two Poles on $\mathbb{P}^1$}
\label{sec:2zero2pole}

We now consider strata with two poles of orders $-k$ and $-l$ and two
zeroes of orders $i$ and $j$, with $i<j$ and $k<l$. Throughout we pass to the quotient by the mapping class group, so mixed-angulations which differ by a shear (after an appropriate cut--paste) are identified. Geometrically this quotients out Dehn twists along the cylinder.

By Corollary~\ref{cor:relationsinkernel_gen} there is a surjection
\[
\pi_1\EG^\star(\sow)/ \Rel^\star(\sow)
\;\twoheadrightarrow\;
\pi_1\Quad(\sow).
\]
The coarse shape of the exchange graph $\EG^\star(\sow)$ depends on the differences $k-i$ and $l-i$. We first treat the generic situation $l-i>3$ (the borderline case $l~-~i~=~3$ is handled separately; the case $l-i<3$ does not occur under our standing assumptions).

\medskip
\noindent\textbf{Left Half of the Exchange Graph (governed by $l-i$).}

\begin{figure}[htbp!]
	\centering
	\input{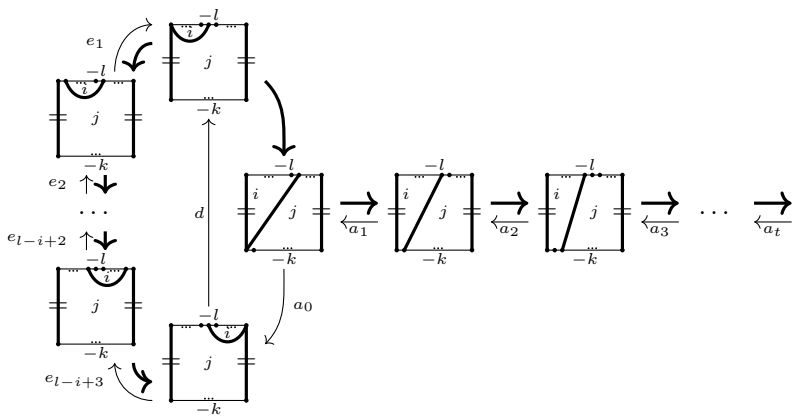}
	\caption{Left half of the exchange graph for $l-i>3$.}
	\label{EGijkl}
\end{figure}

Figure~\ref{EGijkl} shows the part of the exchange graph controlled by
$l-i>3$. The fundamental group of this subgraph is free on
\[
e_1,...,e_{l-i+3},d,a_0,a_1,...,a_t,
\quad
\text{together with generators from the right half.}
\]
At each vertex in this left half we have the usual pentagon/square
relations, which together imply
\[
z := a_0 d
= e_{l-i+3}
= ...
= e_2
= e_1
= d a_0
= a_1
= ...
= a_t.
\]
Thus this half collapses to the abelian subgroup
$\langle z,d\rangle$ with $[z,d]=1$.

\medskip
\noindent\textbf{Right Half (governed by $k-i$).} We distinguish three cases according to the value of $k-i$.

\medskip
\noindent\textbf{Case (A): $k-i\ge4$ or $k-i\le0$.}

\begin{figure}[htbp!]
	\centering
	\input{3_Paper_S/Pictures/cylinder1} \;
	\input{3_Paper_S/Pictures/cylinder5}
	\caption{Right half of the Exchange graph for $k-i\ge4$ (left) and $k-i\le0$ (right).}
	\label{cylinder15}
\end{figure}
See Figure~\ref{cylinder15} for the exchange graphs.
In both sub cases we add generators
$$a_{t+1}, c, b_1,\dots,b_{l-j+3}$$
and obtain the relations
\[
c \cdot a_{t+1}
= b_{l-j+3}
= \cdots
= b_2
= b_1
= a_{t+1} \cdot c
= a_t.
\]
Combining with the left half, we obtain
\[
\langle z,c,d \mid [z,c],[z,d] \rangle,
\]
i.e.\ $z$ is central and the group is
$\mathbb{Z} \times F_2$ with $F_2=\langle c,d\rangle$.

\medskip
\noindent\textbf{Case (B): $k-i=3$ or $k-i=1$.}

\begin{figure}[htbp!]
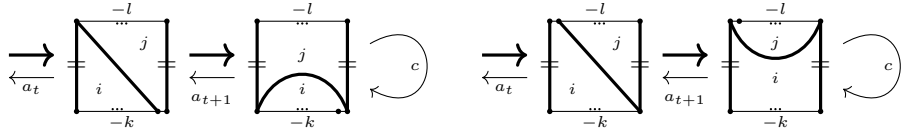

	\centering
	\input{3_Paper_S/Pictures/cylinder2} \;
	\input{3_Paper_S/Pictures/cylinder4}
	\caption{Right half of the Exchange graph: $k-i=3$ (left) and $k-i=1$ (right).}
	\label{cylinder24}
\end{figure}
See Figure~\ref{cylinder24} for the exchange graphs.

In both sub cases we add generators $a_{t+1}$ and $c$.
We have a square relation
$a_t = a_{t+1} \cdot c$
and at the right most vertex the hexagon (second type):
$a_{t+1} \cdot c = c \cdot a_{t+1}$.
Thus, we again obtain as the quotient
\[
\langle z,c,d \mid [z,c],[z,d] \rangle.
\]

\medskip
\noindent\textbf{Case (C): $k-i=2$.}
\begin{figure}[htbp!]
	\centering
	\input{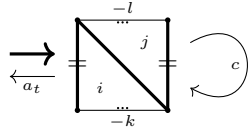}
	\caption{Right half of the Exchange graph for $k-i=2$.}
	\label{cylinder3}
\end{figure}

See Figure~\ref{cylinder3} for the right half of the exchange graph. We add $c$ and a hexagon relation (first type): $c \cdot a_t = a_t \cdot c$.
Hence, the quotient is
\[
\langle z,c,d \mid [z,c],[z,d] \rangle.
\]

\medskip
\noindent\textbf{Borderline Case on the Left: $l-i=3$.}

\begin{figure}[htbp!]
	\centering
	\input{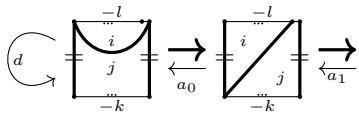}
	\caption{Left half for $l-i=3$.}
	\label{cylinderl-i3}
\end{figure}

Replacing the left half by the $l-i=3$ subgraph (Figure~\ref{cylinderl-i3})
produces the same relations on $z$ and $d$, and the same coupling to the
right half, so the quotient remains
\[
\langle z,c,d \mid [z,c],[z,d] \rangle.
\]

\medskip
\noindent\textbf{Equal Order Poles and Different Order Zeroes: $l=k$ and $i<j$.} As in Section~\ref{sec:3zero1pole}, orbifold points in the $S_2$-quotient
exchange graph arise precisely when the corresponding differential has a
nontrivial stabilizer under the permutation of identical singularities.
In this case this occurs when $i$ (and automatically $j$) is even.
For odd $i$ the action is free and the quotient has no orbifold points.

\medskip
\noindent\textbf{Case (A): $i$ even.}

\begin{figure}[htbp!]
	\centering
	\input{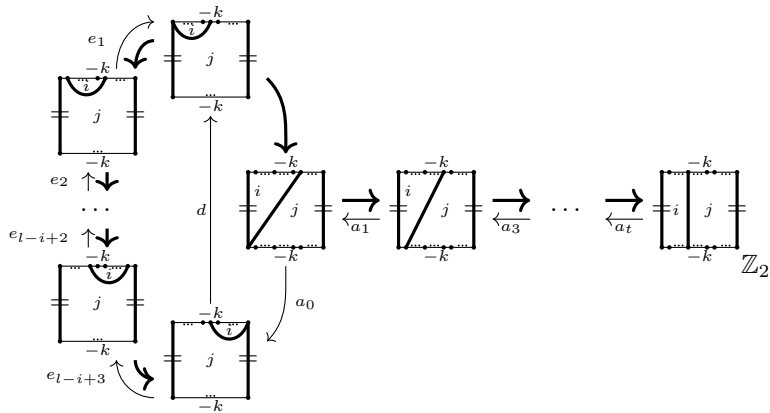}
	\caption{Coarse exchange graph of $(i,j,-k,-k)$ for $i,j$ even.}
	\label{fig:EGijkk}
\end{figure}
Assume $i<j$, $k=l$ and $i$ even. Figure~\ref{fig:EGijkk} shows the coarse
exchange graph, obtained by naively quotienting
Figure~\ref{EGijk} by the $S_2$-action. The stabilizers, when non-trivial, are indicated below the corresponding vertices.
Passing to the $S_2$-quotient swapping the equal poles yields a single
orbifold point at the outer right vertex precisely when $i$ is even. 

The orbifold loop $c$ generates $\mathbb{Z}_2$.
The only new square relation is
$a_t = c^{-1} a_t c$, i.e.\ $[z,c]=1$.
Hence, the quotient of the orbifold fundamental group is
\[
\langle z,c,d \mid [z,c],[z,d],\ c^2=1 \rangle.
\]

\medskip
\noindent\textbf{Case (B): $i$ odd.}

Now assume $i<j$, $k=l$ and $i$ odd. In this case the $S_2$-action is free, so there is no orbifold point. 
The exchange graph is similar to the coarse graph in Figure~\ref{fig:EGijkk}, but with an additional self-loop $c$ at the outer right vertex. This loop represents a genuine free generator rather than an orbifold symmetry. The new square relation is $c^2=a_t=z$, so $z=c^2$ is central. We obtain
\[
\langle c,d \mid [c^2,d] \rangle.
\]

\medskip
\noindent\textbf{Equal Order Zeroes and Different Order Poles: $i=j$ and $l > k$.}

\begin{figure}[htbp!]
	\centering
	\input{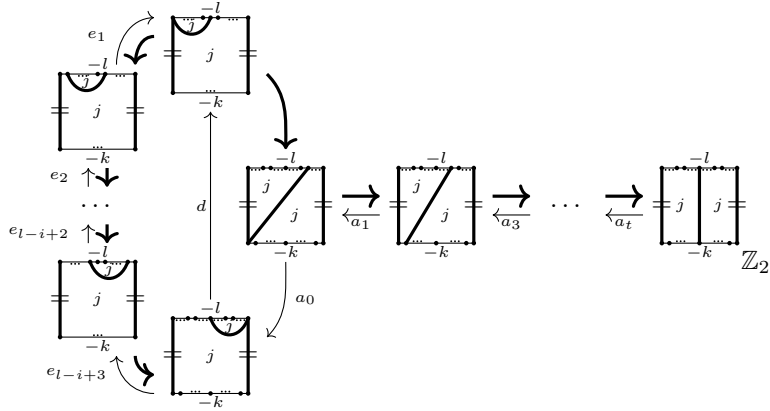}
	\caption{Coarse exchange graph of $(j,j,k,l)$ (displayed for $k,l$ even).}
	\label{fig:EGjjkl}
\end{figure}

The situation with equal zeroes is analogous: there is an $S_2$-action
interchanging the two equal-order zeroes. 

\medskip
\noindent\textbf{Case (A): $k$ even.}
If $k$ is even (and hence $l$ even), the quotient has an orbifold point of order~$2$ at the outer right vertex. The exchange graph is in Figure~\ref{fig:EGjjkl}. The resulting group is
\[
\langle z,c,d \mid [z,c],[z,d],\ c^2=1 \rangle.
\]

\medskip
\noindent\textbf{Case (B): $k$ odd.}
If $k$ is odd do not have any orbifold points.
The exchange graph is similar to the coarse graph in Figure~\ref{fig:EGjjkl}, but with an additional self-loop $c$ at the outer right vertex.
The resulting group is 
\[
\langle c,d \mid [c^2,d] \rangle.
\]

\medskip
\noindent\textbf{Both pairs equal: $i=j$ and $k=l$.} In this case the $S_2$-action on the zeroes and the $S_2$-action define the same action on the graph.
Hence the $S_2 \times S_2$ quotient and the $S_2$ quotient coincide.

\begin{figure}[htbp!]
	\centering
	\input{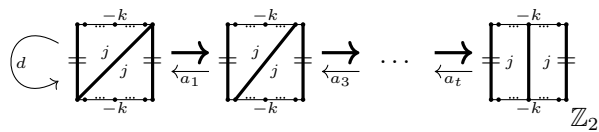}
	\caption{Coarse exchange graph of $(j,j,-k,-k)$ (displayed for $k$ even).}
	\label{fig:EGjjkk}
\end{figure}

\medskip
\noindent\textbf{Case (A): $k,i$ even.} If $k$ and $i$ are even, the quotient has an orbifold point of order~$2$ at the outer right
vertex. The exchange graph is in Figure~\ref{fig:EGjjkk}. The resulting group is
\[
\langle z,c,d \mid [z,c],[z,d],\ c^2=1 \rangle.
\]

\medskip
\noindent\textbf{Case (B): $k,i$ odd.}
If $k$ and $i$ are odd there are not any orbifold points.
The exchange graph is similar to the coarse graph in Figure~\ref{fig:EGjjkl}, but with an additional self-loop $c$ at the outer right vertex.
The resulting group is 
\[
\langle c,d \mid [c^2,d] \rangle.
\]

\subsection{Comparison and Result}
\label{sec:comparison-four}

By Corollary~\ref{cor:relationsinkernel_gen}, there is a natural surjective homomorphism
\[
\pi_1\EG^\star(\sow)/\Rel^\star(\sow)
\;\longrightarrow\;
\pi_1\Quad(\sow),
\]
obtained by sending a closed path in the exchange graph to the corresponding loop in the
stratum defined by the induced sequence of flips and rotations of the associated
mixed-angulation.

In the case of genus~\(0\) with four singularities, the relations collected in
\(\Rel^\star(\sow)\) already generate the full kernel of this map:

\begin{theorem}
	\label{thm:isomorphism-four-singularity}
	For $g=0$ and four singularities we have a natural isomorphism
	\[
	\pi_1\EG^\star(\sow)/\Rel^\star(\sow)
	\;\cong\;
	\pi_1\Quad(\sow).
	\]
	Moreover, the resulting group can be read off from the symmetry type:
	\begin{itemize}
		\item[(a)] $\mathbb{Z} \times F_2$, if all four singularities have distinct orders;
		\item[(b)] $\langle z,c,d \mid [z,c],[z,d],\ c^2=1\rangle$, if two singularities have identical order, and the remaining two have even order not coinciding with the first two;
		\item[(c)] $\langle c,d \mid [c^2,d]\rangle$, if two singularities have identical order, and the remaining two have odd order not coinciding with the first two;
		\item[(d)] $\langle s,d \mid [d^3,s], s^2 \rangle$, if there are three zeroes of identical even order;
		\item[(e)] $B_3 = \langle s,d \mid d^3 =s^2 \rangle$, if there are three zeroes of identical odd order.
	\end{itemize}
\end{theorem}

\begin{proof}
	We show that the natural surjection
	\[
	\pi_1\EG^\star(\sow)/\Rel^\star(\sow)\longrightarrow \pi_1\Quad(\sow)
	\]
	is an isomorphism.
	
	\smallskip
	In all cases, the center of
	$\pi_1\EG^\star(\sow)/\Rel^\star(\sow)$ is infinite cyclic, generated by one side of a
	square or pentagon relation. Geometrically, this generator corresponds to the global
	$2\pi$-rotation of the quadratic differential induced by the $\mathbb{C}^\star$-action.
	Consequently, in each case we obtain a short exact sequence
	\[
	1 \longrightarrow \Z \longrightarrow \pi_1\EG^\star(\sow)/\Rel^\star(\sow)
	\longrightarrow \Gamma \longrightarrow 1,
	\]
	where the quotient group $\Gamma$ depends only on the symmetry type of the singularities.
	Explicitly,
	\begin{itemize}
		\item[(a)] $\Gamma = F_2$,
		\item[(b)] $\Gamma = \Z \ast \Z_2$,
		\item[(c)] $\Gamma = \Z \ast \Z_2$,
		\item[(d)] $\Gamma = \mathrm{PSL}_2(\Z)$,
		\item[(e)] $\Gamma = \mathrm{PSL}_2(\Z)$.
	\end{itemize}
	
	\smallskip
	On the geometric side, projectivisation identifies $\PP\Quad(\sow)$ with the moduli space
	of the underlying marked sphere:
	\[
	\PP\Quad(\sow)\;\cong\;
	\begin{cases}
		M_{0,4},\\
		M_{0,[2]+2},\\
		M_{0,[3]+1},
	\end{cases}
	\]
	depending on the symmetry type. In each case,
	$\pi_1\Quad(\sow)$ fits into the central extension
	\[
	1 \longrightarrow \pi_1(\mathbb{C}^\star)\cong\Z \longrightarrow
	\pi_1\Quad(\sow) \longrightarrow \pi_1\PP\Quad(\sow) \longrightarrow 1,
	\]
	where the left map is induced by the $\mathbb{C}^\star$-rotation.
	
	\medskip\noindent
	\emph{Uniform reduction.}
	Since the two central $\Z$-subgroups are canonically identified under the natural map,
	it suffices to show that the induced map
	\[
	\bar\Phi:\;
	\Gamma
	=
	\Bigl(\pi_1\EG^\star(\sow)/\Rel^\star(\sow)\Bigr)\big/\Z
	\;\longrightarrow\;
	\pi_1\PP\Quad(\sow)
	\]
	is an isomorphism.
	
	\smallskip\noindent
	\emph{Infinite-order generators.}
	In all symmetry types, $\Gamma$ contains infinite-order generators represented by explicit sequences of flips in the exchange graph, obtained by repeatedly flipping a fixed edge until the initial mixed-angulation is recovered. Under the natural map to $\PP\Quad(\sow)$, such a flip cycle is sent to a braid-type loop in the corresponding moduli space of marked points. After fixing three marked points at $0,1,\infty$, this loop is given by the minimal motion in which one marked point winds once around another.
	
	If the two points are distinguishable (for instance, if they have different orders), this loop is a full twist and already closes in the labelled moduli space. If the two points are indistinguishable, the primitive loop is a half-twist exchanging them; it lifts to a path in the labelled cover whose endpoints differ by a relabelling, but it is closed in the orbifold quotient. In both cases, these braid-type loops represent infinite-order elements and generate the free part of $\pi_1\PP\Quad(\sow)$.

	\smallskip\noindent
	\emph{Finite-order generators.}
	Finite-order elements in $\Gamma$ arise in exactly two ways.
	
	\smallskip\noindent
	\textbf{(i) Orbifold stabilisers.}
	If a vertex of the exchange graph has a nontrivial stabiliser under the relabelling
	action, Proposition~\ref{prop:orbpi1-graph} introduces an additional torsion generator.
	This element is represented by a path which lifts to a path between distinct sheets of the
	labelled covering and closes only after applying the stabiliser. Its image in
	$\PP\Quad(\sow)$ is a loop around the corresponding orbifold point and therefore has the
	same finite order.
	
	\smallskip\noindent
	\textbf{(ii) Self-loops becoming torsion modulo $\Z$.}
	Even when all stabilisers are trivial, torsion may appear after quotienting by the
	central $\Z$. If a self-loop $\ell$ satisfies $\ell^m=z$, where $z$ is the central
	$2\pi$-rotation, then $\ell$ becomes an element of order $m$ in $\Gamma$. Geometrically,
	these elements correspond to cyclic permutations of indistinguishable singularities, such
	as transpositions of two points or cyclic permutations of three points, and map to the
	standard orbifold generators of $\pi_1\PP\Quad(\sow)$.
	
	\smallskip
	\emph{Example.}
	In the case of two poles of identical order and two zeroes, the infinite-order generator
	is represented by the cycle obtained by repeatedly flipping the vertical edge starting at
	the top mixed-angulation in Figure~\ref{fig:EGijkk}. Under the induced map, this cycle
	corresponds to a loop in $M_{0,[2]+2}$ in which one labelled point travels once around one
	of the unlabelled points, yielding the standard $\Z$-generator. The order~$2$ generator
	arises either from a nontrivial stabiliser or from a self-loop whose square is the central
	$2\pi$-rotation, depending on the parity of the zero orders. In both cases, its image is
	the canonical $\Z_2$-generator of $\pi_1(M_{0,[2]+2})$.
	
	\smallskip
	All remaining symmetry types are treated analogously. In each case, the generators on the
	exchange-graph side map to the standard geometric generators of
	$\pi_1\PP\Quad(\sow)$. Hence $\bar\Phi$ is an isomorphism in all cases, and since the
	central $\Z$ is identified on both sides, the original map
	\[
	\pi_1\EG^\star(\sow)/\Rel^\star(\sow)\longrightarrow \pi_1\Quad(\sow)
	\]
	is an isomorphism.
	
\end{proof}